\documentclass[11pt]{amsart}
\usepackage{fullpage}
\usepackage{amsmath,amssymb,amsfonts,amsthm}
\usepackage{array}
\usepackage{pdfsync}
\usepackage{dsfont}
 \usepackage{color}
\newtheorem{theorem}{Theorem}
\newtheorem{corollary}{Corollary}
\newtheorem{lemma}{Lemma}
\newtheorem{proposition}{Proposition}
\newtheorem{remark}{Remark}
\makeatletter

\@addtoreset{equation}{section}
\makeatother
\newcommand{\T}{{\mathbb T}}
\newcommand{\N}{{\mathbb N}}
\newcommand{\Z}{{\mathbb Z}}

\newcommand{\R}{{\mathbb R}}

\newcommand{\pa}{{\partial}}
\newcommand{\na}{{\nabla}}

\newcommand{\eps}{{\varepsilon}}
\newcommand{\Hc}{\mathcal{H}}
\newcommand{\Nc}{\mathcal{N}}

\newcommand{\Rc}{\mathcal{R}}

\renewcommand{\a}{\alpha}
\renewcommand{\b}{\beta}
\newcommand{\g}{\gamma}

\newcommand{\e}{\varepsilon}

\newcommand{\Tc}{\mathcal{T}}

\title{Quasineutral limit for  Vlasov-Poisson \\ with  Penrose stable data}

\author{Daniel Han-Kwan}
 \address{D. Han-Kwan \\ CNRS $\&$ Centre de Math\'ematiques Laurent Schwartz (UMR 7640), \'Ecole polytechnique, 91128 Palaiseau Cedex,  France}
\email{daniel.han-kwan@polytechnique.edu}
 \author{Fr\'ed\'eric Rousset}
  \address{F. Rousset \\ Laboratoire de Math\'ematiques d'Orsay (UMR 8628), Universit\'e Paris-Sud et Institut Universitaire de France, 91405 Orsay Cedex, France}
  \email{frederic.rousset@math.u-psud.fr}

\begin{document}

  \begin{abstract}
  We study the quasineutral limit of  a Vlasov-Poisson system that describes the dynamics of ions in a plasma.
   We handle data with Sobolev regularity under the sharp assumption that the profile of the initial data
    in the velocity variable satisfies a Penrose stability condition.
    
    As a by-product of our analysis, we obtain a well-posedness theory for the limit  equation (which is a Vlasov equation
     with Dirac distribution as interaction kernel) for such data.

\end{abstract}

  \maketitle
  \tableofcontents
  
  \section{Introduction and main results}
  We study the quasineutral limit, that is the limit $\eps\to 0$, for the following Vlasov-Poisson system describing the dynamics of ions in the presence of massless electrons:
   \begin{equation} \label{VP}
\left\{
\begin{aligned}
& \pa_t f_\e \: + \: v \cdot \na_x f_\e \: + \:   E_\e \cdot \na_v f_\e \:  = \: 0, \\
& E_\e = - \na_x V_\e, \\
& V_\e- \eps^2 \Delta V_\e= \int_{\R^d} f_\e \, dv - 1, \\
&f_\e \vert_{t=0} = f^0_{\e}.
\end{aligned}
\right. 
\end{equation}
In these equations, the function $f_\eps(t,x,v)$ stands for the distribution functions of the ions in phase space $\T^d \times \R^d$, $d \in \N^*$. We assumed that the density of the electrons $n_e$ satisfies a linearized Maxwell-Boltzmann law, that is 
$
n_e = e^{V_\eps} \sim 1+ V_\eps
$
which accounts for the source $-(1+ V_\eps)$ in the Poisson equation. Such a model was recently studied for instance in \cite{HK10,HK11, HKH, Hauray-Nouri}.
 Though we have focused on this simplified law, the arguments in this paper could  be easily adapted to the model
  where the potential is given by the Poisson equation $ - \eps^2 \Delta V_{\eps}=  \int_{\R^d} f_\e \, dv -  e^{V_\eps}$.

The dimensionless parameter $\eps$ is defined by the ratio between the Debye length of the plasma and the typical observation length. It turns out that in most practical situations, $\eps$ is very  small, so that the limit $\eps \to 0$, which bears the name of quasineutral limit, is  relevant from the physical point of view.
Observe that in the regime of small $\eps$, we formally have that the density of ions is almost equal to that of electrons, hence the name quasineutral.
This regime is so fundamental that it is even sometimes included in the very definition of a plasma, see e.g. \cite{Chen}.

The quasineutral limit for the  Vlasov-Poisson system with the Poisson equation
   \begin{equation} \label{VP2}
- \eps^2 \Delta V_\e= \int_{\R^d} f_\e \, dv - \int_{\T^d \times \R^d} f_\e \, dv \, dx
\end{equation}
that describes the dynamics of electrons in a fixed neutralizing background of ions
is also very interesting. Nevertheless, we shall focus in this paper on the study of  \eqref{VP}. The study of \eqref{VP2} combines the difficulties already present in this  paper linked
 to kinetic instabilities and the presence of high frequency waves due to the large electric field
  that do not occur  in the case of \eqref{VP}. The study of the combination of these two phenomena is postponed to future work.
%The reason is that the convergence result we shall prove is slightly less technical for \eqref{VP}. However, we expect that the arguments explained here can apply to this other system, see \cite{HKR}.
%
%Although we shall focus on \eqref{VP}, we will also explain in the end of the paper how to deal with the classical Vlasov-Poisson system
%     \begin{equation} \label{VP2}
%\left\{
%\begin{aligned}
%& \pa_t f_\e \: + \: v \cdot \na_x f_\e \: + \:   E_\e \cdot \na_v f_\e \:  = \: 0, \\
%& E_\e = - \na_x V_\e, \\
%& - \eps^2 \Delta V_\e= \int_{\R^d} f_\e \, dv - 1, \\
%&f_\e \vert_{t=0} = f^0_{\e}.
%\end{aligned}
%\right. 
%\end{equation}
  
  \medskip
  
  It is straightforward to obtain the formal quasineutral  limit of \eqref{VP} as $\eps\to 0$, we expect that  $\eps^2 \Delta V_{\eps}$  tends to zero  and
   hence if $f_{\eps}$ converges in a reasonable way to some $f$, then $f$ should solve 
       \begin{equation} \label{formel}
\left\{
\begin{aligned}
& \pa_t f \: + \: v \cdot \na_x f \: + \:   E \cdot \na_v f \:  = \: 0, \\
& E = - \na_x \rho, \quad  \rho=\int_{\R^d} f \, dv, \\
&f \vert_{t=0} = f^0.
\end{aligned}
\right. 
\end{equation}
  This system was named  \emph{Vlasov-Dirac-Benney} by Bardos \cite{BardosX} and studied in \cite{Bardos-Nouri,Bardos-Besse}.
  It was also referred to as the \emph{kinetic Shallow Water system} in \cite{HK11} by analogy with the classical Shallow Water system of fluid mechanics. In particular, it was shown in \cite{Bardos-Nouri}  that  the  semigroup of the linearized system around \emph{unstable} equilibria is unbounded in Sobolev spaces (even with loss of derivatives). This yields the illposedness of \eqref{formel} in Sobolev spaces, see in particular the recent work \cite{HKN}. In \cite{Bardos-Besse}, it was nevertheless  shown in dimension one, i.e. for $d= 1$   that \eqref{formel}  is wellposed in the class of functions $f(x,v)$ such that for all $x \in \T$, $v\mapsto f(x,v)$ is compactly supported and  is  increasing  for $v \leq m(t,x)$ and then decreasing for $v \geq m(t,x)$,  that is to say for functions that for all $x$ have the shape of  \emph{one bump}.
   The method in this paper is to reduce the problem to an infinite number of  fluid type  equations by using a water bag  decomposition.  As we shall see below a by-product of the main result of this paper is the well-posedness
    of the system \eqref{formel} in any dimension for smooth data with finite Sobolev regularity  such that for every $x$,  the profile\
     $v \mapsto f^0(x, v)$ satisfies a Penrose stability condition. This condition is automatically satisfied 
      by smooth ``one bump for all $x$''  functions in dimension one. It is also satisfied for small enough data for example.

\medskip

The mathematical study of the quasineutral limit started in the nineties with pioneering works of Brenier and Grenier for Vlasov with the Poisson equation \eqref{VP2}, first  with a limit involving defect measures \cite{BG,Gr95}, then with a full justification of the quasineutral limit for initial data with uniform analytic regularity \cite{Gr96}. The work \cite{Gr96} also included a description of the so-called plasma waves, which are time oscillations of the electric field of frequency  and amplitude  $O(\frac{1}{\eps})$. As already said, such oscillations actually do not occur in the quasineutral limit of  \eqref{VP}. 
More recently, in \cite{HKI,HKI2}, relying on Wasserstein stability estimates inspired from \cite{Hauray,Loe2}, it was  proved that exponentially small but rough perturbations are allowed in the main result of \cite{Gr96}.

In analytic regularity, it turns out that instabilities for the Vlasov-Poisson system, such as two-stream instabilities,
do not have any effect, whereas in the class of Sobolev functions, they definitely play a crucial role.
It follows  that the  quasineutral approximation  both for \eqref{VP} and \eqref{VP2} is not always valid.
 In particular, the convergence of \eqref{VP} to \eqref{formel} does not hold in general: we refer to \cite{Gr99,HKH}.

Nevertheless, it can be expected that the formal limit can be justified  in Sobolev spaces for stable situations.
 We shall soon be more explicit about what we mean by stable data, but this should at least be included in  the class of data
  for which the expected limit system \eqref{formel} is well-posed. The first result in this direction is due to Brenier \cite{Br00} (see also \cite{Mas} and \cite{HK11}), in which he justifies the quasineutral limit for initial data converging to a monokinetic distribution, that is a Dirac mass in velocity. This corresponds to a stable though singular case since the Dirac mass can be seen as an extremal case of a Maxwellian, that is a function with one bump. Brenier introduced the so-called modulated energy method to prove this result.
 Note that in this case the limit system is a fluid system (the incompressible Euler equations in the case of~\eqref{VP2} or the shallow water equations in the case of~\eqref{VP}) and not a kinetic equation.
  This result is coherent with the fact that the instabilities present at the kinetic level do not show up at the fluid level, 
  for example the quasineutral limit of the Euler-Poisson system can be justified in Sobolev spaces and has been for example
   studied  in
   % \cite{Grenier-pseudo}
    \cite{CG}, \cite{Loe}.

For  non singular stable data with Sobolev regularity, there are  only  few available results  which are all
 in the one-dimensional case  $d=1$.
\begin{itemize}
\item In \cite{HKH}, using the modulated energy method, the quasineutral limit is justified for  very special initial data namely
 initial  data converging to one bump functions that are furthermore \emph{symmetric} and space homogeneous (thus that are  stationary  solutions to \eqref{VP} and \eqref{formel}). It is also proved that this is  the best we could hope for with this method.
\item Grenier sketched in \cite{Gr99} a result of convergence for data such that for every $x$  the profile  in $v$ has only one bump. The proposed proof involves a  functional taking advantage of  the monotonicity in the one bump structure.
 Such kind of functionals have been recently used  in other settings, for example in the  study of  the hydrostatic Euler equation or the Prandtl equation, see
  for example \cite{Masmoudi-Wong1,Masmoudi-Wong2,GVM}.
\end{itemize}

The main goal of this work is  to justify the quasineutral limit that is to justify the derivation of \eqref{formel} from \eqref{VP} in the general stable case and in any dimension.
 As a byproduct, we shall also obtain the well-posedness of \eqref{formel}
  for stable data with Sobolev regularity.

\bigskip

To state our results, we shall  first introduce  the  Penrose stability condition \cite{Penrose}  for homogeneous equilibria $\mathbf{f}(v)$. Let us define for the profile $\mathbf{f}$ the Penrose function 
$$ \mathcal{P}(\gamma, \tau, \eta, \mathbf{f})= 1 -  \int_{0}^{+ \infty} e^{-(\gamma + i \tau)s}\, {i \eta \over 1 + |\eta|^2}\cdot  ( \mathcal{F}_{v} \nabla_{v}  \mathbf f)(\eta s ) \, d s, \quad \gamma >0,\, \tau \in \mathbb{R}, \,\eta \in \mathbb{R}^d\backslash\{0\}$$
where $ \mathcal{F}_{v}$ denotes the Fourier transform in $v$, whose definition is recalled below.
 We shall say that the profile $\mathbf f$ satisfies the Penrose stability condition if 
\begin{equation}
\label{Penrose0}
\inf_{(\gamma, \tau, \eta) \in (0,+\infty) \times \R \times \mathbb{R}^d\setminus\{0\}} \left| \mathcal{P}(\gamma, \tau, \eta, \mathbf{f}) \right| >0.
\end{equation}
It will be also convenient to say that $\mathbf f$ satisfies the $c_{0}$ Penrose stability condition for some $c_{0}>0$ if
\begin{equation}
\label{Penrose}
\inf_{(\gamma, \tau, \eta) \in (0,+\infty) \times \R \times \mathbb{R}^d\setminus\{0\}} \left| \mathcal{P}(\gamma, \tau, \eta, \mathbf{f}) \right| \geq c_{0}.
\end{equation}
The non-vanishing of $ \mathcal{P}$ only has to be checked on a compact subset if $f$ is smooth and localized enough
(for example if $f \in \Hc^{3}_{\sigma}$, $\sigma >d/2$ with the notation below) and thus
 if $\mathbf f$ verifies the \eqref{Penrose0} stability condition then it also satisfies the \eqref{Penrose} stability
  condition for some $c_{0}>0$.
This condition is necessary for the large time  stability of the profile $\mathbf f$ in the unscaled Vlasov-Poisson  equation
(that is to say for \eqref{VP} with $\eps = 1$).
 Note that it was recently proven in \cite{MV}, \cite{BMM} that
  Landau damping holds in small Gevrey neighborhood of such stable solutions, which means that such profiles
   are  nonlinearly stable and even asymptotically stable in a suitable sense  with respect to small perturbations  in Gevrey spaces (see also \cite{FR} for Landau damping in Sobolev spaces for the Vlasov-HMF equations).
 %   As shown in \cite{HKH}, this condition on $\mathbf f$ is also necessary in order to justify
   %  the quasineutral limit of \eqref{VP} for data of Sobolev regularity  $f^0_{\eps}$ converging to $\mathbf f$.
     
\begin{remark} 
\label{r-penrose} The assumption~\eqref{Penrose} is automatically satisfied in a small data regime.  In a one-dimensional setting, \eqref{Penrose} is also satisfied for the ``one bump'' profiles described previously.
   In any dimension,  \eqref{Penrose} is verified for any radial non-increasing  function (therefore, local Maxwellians are included)  and there exist more sophisticated criteria based
    on  the one bump structure  of the  averages of the function  along  hyperplanes.  We refer to \cite{MV} for other conditions ensuring~\eqref{Penrose}. Note that any sufficiently small perturbation of a Penrose stable profile
     is also Penrose stable.
 \end{remark}
  
 %Note that in particular it is always satisfied if  $M_{0}$ is sufficiently small
    %   and  in dimension one if  the data is a one bump for all $x$ profile.
     
%Note that this is the multidimensional generalization of the Penrose condition for the stability of the profile $\mathbf f$  for \eqref{formel}
% obtained in \cite{Bardos-Besse}. As we shall see in section \ref{Penrose} this criterion must be verified for the stability of 
%  $\mathbf f$ in the  usual Vlasov-Poisson equation.
%  Note that $\mathcal{P}(\cdot, \mathbf f)$ is homogeneous degree zero in its argument 
%   so that it suffices to check the condition for the parameters $(\gamma, \tau, \eta)$ in $S_{+}$ with 
%  $S_+ = \{(\tilde \gamma, \tilde \tau, \tilde \eta), \, \tilde \gamma^2 + \tilde \tau^2 +  \tilde \eta^2= 1, \, \tilde \gamma \geq 0 \}$.

\bigskip

Throughout this paper, we consider take the following normalization for the Fourier transform on $\T^d$ and $\R^m$:
$$
\begin{aligned}
&\mathcal{F}( u) (k) := (2\pi)^{-d} \int_{\T^d} u(x) e^{-i k \cdot x} \, dx, \quad k \in \Z^d, \\
&\mathcal{F}( v) (\xi) := (2\pi)^{-m} \int_{\R^m} v(y) e^{-i \xi \cdot y} \, dy, \quad \xi \in \R^m.
\end{aligned}
$$
With this convention, the inverse Fourier transform yields
$$
\begin{aligned}
&u(x) = \sum_{k \in \Z^d} \mathcal{F}( u) (k) e^{i k \cdot x}, \quad x \in \T^d, \\
&v(y)= \int_{\R^m} \mathcal{F}( v) (\xi) e^{i \xi \cdot y} \, d\xi, \quad y \in \R^m.
\end{aligned}
$$

For $k \in \N, r \in \R$, introduce the weighted Sobolev norms
\begin{equation}
\| f \|_{\Hc^{k}_{r}} := \left(\sum_{|\a| + |\b| \leq k} \int_{\T^d} \int_{\R^d} (1+ |v|^2)^{r} |\pa^\a_x \pa^\b_v f|^2 \, dv dx \right)^{1/2}
\end{equation}
where  for $\a = (\a_1, \cdots, \a_d), \b = (\b_1, \cdots, \b_d) \in \N^d$, we write
$$
|\a|= \sum_{i=1}^d \a_i, \quad |\b|= \sum_{i=1}^d \b_i,
$$
$$
\pa^\a_x := \pa^{\a_1}_{x_1} \cdots \pa^{\a_d}_{x_d},\quad \pa^\b_v := \pa^{\b_1}_{v_1} \cdots \pa^{\b_d}_{v_d}.
$$
We will also use the standard Sobolev norms, for functions $\rho(x)$ that depend only on $x$
\begin{equation}
\| \rho \|_{H^{k}} =  \| \rho \|_{H^{k}_{x}} := \left(  \sum_{|\a|  \leq k} \int_{\T^d} |\pa^\a_x \rho|^2 \, dx\right)^{1/2}.
\end{equation}  
Setting $\rho_\eps(t,x) := \int_{\R^d} f_\eps (t,x,v) \, dv$, where $f_\eps$ satisfies~\eqref{VP}, we introduce the key quantity  for $m \in \N^*$  and  $r \in \mathbb{R}_{+}$ (they  will be  taken sufficiently large)
$$
\Nc_{2m,\, 2r}(t, f_\e) := \| f_\e \|_{L^\infty((0,t), \Hc^{2m-1}_{2r})} + \| \rho_\e \|_{L^2((0,t), H^{2m})}.
$$
Let us fix our regularity indices.  We define 
 \begin{equation}
 \label{seuil}
 m_{0}= 3 + {d \over 2} + p_{0}, \quad p_{0}=\lfloor{d\over 2}\rfloor + 1, \quad r_{0}= \max( d,  2+ {d\over 2})
 \end{equation}
  and we shall mainly work with $2m>m_{0}$ and $2r>r_{0}$.

\bigskip

The main result of this paper is a  uniform in $\eps$ local existence result in Sobolev spaces for \eqref{VP} in the case of 
 data for which the profile $v \mapsto f^0(x,v)$ satisfy  the Penrose stability condition \eqref{Penrose0} for every $x$.
  More precisely, we shall prove the following theorem.
\begin{theorem}
\label{theomain}
%Assume that $f^0_{\e}= f^0$ for all $\eps \in (0,1]$, with  $f^0 \in \Hc^{2m}_{2n_{0}}$ with $2m>$,  $2n_{0}>1+d$  and that  for all $x \in \T^d$, $v \mapsto f^0(x,v)$ satisfies the Penrose stability condition \eqref{Penrose}.
%\max (d+3, 4+ {d \over 2})
Assume that for all $\eps \in (0,1]$,   $f^0_{\e} \in \Hc^{2m}_{2r}$ with $2m> m_{0}$,  $2r> r_{0}$  and that
there is $M_{0}>0$ such that for all $\eps \in (0,1]$, $\| f^0_{\e}\|_{\Hc^{2m}_{2r}} \leq M_{0}$.
 Assume moreover that there is $c_0>0$ such that for every $x \in \T^d$ and for every $\eps \in (0, 1],$  the profile  $v \mapsto f^0_{\e} (x,v)$ satisfies  the
  $c_{0}$ Penrose stability condition \eqref{Penrose}.

Then there exist $T>0$,  $R>0$ (independent of $\eps$)   and a unique solution $f_\eps \in \mathcal{C}([0, T], \Hc^{2m}_{2r})$  of \eqref{VP} such that
 $$\sup_{\eps \in (0, 1]} \Nc_{2m, 2r}(T, f_\eps) \leq R$$
  and $f_{\eps}(t, x, \cdot)$ satisfies the $c_{0}/2$  Penrose stability condition \eqref{Penrose}
   for every $t\in [0, T]$ and $x \in \mathbb{T}^d$.
\end{theorem}
%such that $f_\e \to_{\e\to 0} f$ in $C([0,T], ...)$, where $f_\e$ is the solution to \eqref{VP} and $f$ is the solution to \eqref{formel}.
%\end{theorem}

As already mentioned, the Penrose stability condition that we assume is sharp in the sense that   
it is necessary      
in order to justify the quasineutral limit for data with Sobolev regularity, see \cite{HKH}. 
 By Remark~\ref{r-penrose}, %our needed assumption besides smoothness on $f^{0}$
the assumption that for all $x \in \T^d$, $v \mapsto f^0_\eps (x,v)$ satisfies the $c_{0}$ Penrose stability condition \eqref{Penrose}  is verified
  if $f^0_{\eps}$ converges  to a function $f^0$  under the form
 $$ f^{0}(x,v)= F(x, |v-u(x)|^2) + g(x,v)$$
 where $F: \,  \mathbb{T}^d \times \mathbb{R}_{+} \rightarrow \mathbb{R}_{+}$  is such that for every $x$, $F(x, \cdot)$ is non-increasing, 
 $u$ is a  smooth function and $g$ is a sufficiently small perturbation (in $\mathcal{H}^s_{\sigma}$ for $s>2$ and  $\sigma > d/2$). This class includes
 in particular small perturbations of local Maxwellians:
 $$
 M(x,v)= \frac{\rho(x)}{(2\pi T(x))^{d/2}} \exp \left(-\frac{|v-u(x)|^2}{T(x)}\right).
 $$

Note also that though it is not stated in the  Theorem,  the solution $ f_\eps$ remains in $\Hc^{2m}_{2r}$ on $[0, T]$.
 Nevertheless,   the  $\Hc^{2m}_{2r}$ norm is not controlled uniformly in $\eps$, only the quantity $\Nc_{2m, \, 2r}(T, f_\eps)$ is.
 
 \bigskip
 
 From this uniform existence result, we are then able to justify the quasineutral limit for \eqref{VP}.  
\begin{theorem}
\label{theoquasi}
%Under the assumptions of the above Theorem.
Let $f^0_{\eps} \in \Hc^{2m}_{2r}$  satisfying the assumptions of Theorem \ref{theomain}. 
Assume that in addition there is $f^0 \in L^2(\mathbb{T}^d \times \mathbb{R}^d)$ such that $f^0_{\e} \to f^0$ in $L^2(\mathbb{T}^d \times \mathbb{R}^d)$. Then,  
 on the interval  $[0, T]$ with $T>0$ defined in Theorem~\ref{theomain}, we have that 
%$$
%f_\eps \rightharpoonup f \text{  weakly}-\star \text{  in } L^\infty([0,T],  \Hc^{2m-1}_{2n_{0}}),
%$$
%{\color{red} F:  Je me contenterai d'enoncer la convergence forte ? }
%where $f$ satisfies the Vlasov-Dirac-Benney equation \eqref{formel} with initial condition $f^0$.
%
%Furthermore, there holds the strong convergence 
$$
\sup_{[0,T]} \| f_\eps - f \|_{L^2_{x,v}\cap L^\infty_{x,v}} + \sup_{[0, T]} \| \rho_{\eps} - \rho \|_{L^2_{x} \cap L^\infty_{x}} \to_{\eps\to 0} 0
$$
where $f$ is a solution of \eqref{formel} with initial data $f^0$ such that $f \in  \mathcal{C}([0,  T], \Hc^{2m-1}_{2r})$, 
 $\rho = \int_{\R^d} f dv \in L^2([0, T], H^{2m})$ and  that satisfies the $c_{0}/2$  Penrose stability condition \eqref{Penrose}
   for every $t\in [0, T]$ and $x \in \mathbb{T}^d$.

%Under the assumptions of Theorem \ref{theomain}, the quasineutral limit holds, that is
\end{theorem}

%\begin{remark}. {\color{red} D: peut-on amŽliorer la convergence forte et/ ou rendre la convergence quantitative ?
%
%F: je crois que oui, éventuellement avec un peu plus de r\'egularit\'e, on peut \'ecrire Vlasov Poisson sous la forme
%$$\partial_{t} f  + v \cdot \nabla_{x} f + \nabla_{x} \rho \cdot \nabla_{v}f= - \eps^2 \nabla_{x} \Delta V \nabla_{v} f = O (\eps^2)$$
%et donc l'estimation de stabilit\'e sur deux solutions de l'\'equation limite  qui donne l'unicit\'e donnera aussi la converge au moins sur les densit\'es avec
% un taux $\eps^2+$ la vitesse de convergence de la donn\'ee initiale. 
%  Si on veut une convergence sur $f$  ca peut marcher aussi mais il faudrait au moins montrer une stabilit\'e $H^1$ sur les densité\'es
%  pour ensuite avoir une convergence $L^2$ donc il faudrait reproduire pas mal des arguments de la preuve du theoeme principal
%  on peut le mentionner mais ce n'est peut être pas la peine de trop s'embêter, je pense que \c{c}\`a va allonger le papier
%  pour un gain limite
%}
% \end{remark}

 As a by-product of our analysis, we obtain local well-posedness for Vlasov-Dirac-Benney, in the class of Penrose stable initial data.  {Existence is a consequence of the statement of Theorem~\ref{theoquasi}, while uniqueness is more subtle and is rather a consequence of the analysis that is used  to prove Theorem~\ref{theomain}}.
 \begin{theorem}
 \label{WP}
 Let $f^0 \in \Hc^{2m}_{2r}$ with $2m>m_{0}$,  $2r>r_{0}$  be such that  for all $x \in \T^d$, $v \mapsto f^0(x,v)$ satisfies the $c_{0}$ Penrose stability condition \eqref{Penrose}. Then, there exists $T>0$ for which there is a unique solution to \eqref{formel} with initial condition $f^0$ and such that    $f \in \mathcal{C}([0, T], \Hc^{2m-1}_{2r})$,  $ \rho \in L^2([0, T], H^{2m})$ and  $v \mapsto f(t,x,v)$ satisfies  the $c_{0}/2$ Penrose condition  for every $t\in [0, T]$ and $x \in \mathbb{T}^d$.
 \end{theorem}

      \begin{remark}
 We have focused on periodic boundary conditions in $x$. Nevertheless, our results could
       be extended to the case $x \in \mathbb{R}^d$ without major changes.
      \end{remark}
 
 \section{Strategy}
 
Let us explain in this section the main strategy  that we follow in this paper. %we follow as well as the organization of this paper. %We first introduce in Section~\ref{prepre} some other notations and some standard estimates for weighted Sobolev spaces. 
The main part of this work consists of Sections~\ref{secPre1}, \ref{secPre2}, \ref{secH2m} and \ref{secconclusion} where we provide the proof of Theorem~\ref{theomain}. 

Our proof is based on a bootstrap argument, which we initiate in Section~\ref{secPre1}.
 The main difficulty is to derive a suitable a priori uniform estimate for $\Nc_{2m, \, 2r} (T, f_\eps)$ for some $T>0$.
  Note that if we consider data with a better localization in the velocity space, 
       we could rely on the fact that for all $\eps>0$, there is a unique global classical solution of \eqref{VP} (see \cite{Batt-Rein}). However, such a result is not useful in view of the quasineutral limit, since it does not  provide estimates that are uniform in $\eps$.
       %$T^*= +\infty$ by using the result of \cite{Batt-Rein} (see also \cite{Pfaffelmoser}, \cite{Perthame-Lions}).
      %  {\color{red}D: pas tres important, mais pas clair pour moi que Batt-Rein assure que la norme $\Hc$ n'explose pas ?? Perthame-Lions bof a cause du tore?}
       
%Note that since they do not provide estimates that are uniform in $\eps$ 
         %most of the work remains anyway to be done anyway.
 
  Assuming a control of $  \| \rho_\e \|_{L^2((0,t),H^{2m})}$ for some $t >0$, the  estimate for $ \| f_\e \|_{L^\infty((0,t),\Hc^{2m-1}_{2r})}$ can be obtained from a standard energy estimate (see Lemma~\ref{lemfacilesans}). Consequently,  the difficulty is to estimate $  \| \rho_\e \|_{L^2((0,t),H^{2m})}$. 
 From now on, we will forget the subscript $\e$ to reduce the amount of notation.
 
%  In order to explain the main idea, let us focus on this task for the Vlasov-Dirac-Benney equation.
   A natural idea would be to use the fact that up to commutators, given $f(t,x,v)$ satisfying~\eqref{VP},
    $\partial_{x}^{2m}f$  evolves according to the linearized equation about
     $f$, that is
       \begin{equation}
       \label{naive}
        \partial_t \partial_{x}^{2m} f + v \cdot \na_x \partial_{x}^{2m} f  +  \pa_x^{2m} E \cdot \nabla_v f 
     + E \cdot \nabla_v \partial_{x}^{2m} f + \cdots =0,
   %  \partial_t \partial_{x}^{2m} f + v \cdot \na_x \partial_{x}^{2m} f  + \na_x \pa_x^{2m} \rho \cdot \nabla_v f 
   %  + \nabla_x  \rho \cdot \nabla_x \partial_{x}^{2m} f + \na_x \pa_x \rho \cdot \na_v \partial_{x}^{2m-1} f + \cdots =0,
     \end{equation}
     where $\cdots$ should involve  remainder terms only.
      One thus has first  to understand this linearized equation.
     When $f\equiv \overline{f}(v)$ does not depend on $t$ and $x$, then the  linearized equation reduces to
  \begin{equation}
       \label{naive2}
     \partial_t \dot  f + v \cdot \na_x \dot f    +  \dot  E \cdot \nabla_v \overline{f} =S
\end{equation}
     and one can  deduce
      an integral equation for $ \dot \rho =  \int_{\mathbb{R}^d} \dot f dv$ by solving the free transport equation and integrating in $v$.
       This was used for example in the study of Landau damping by Mouhot and Villani \cite{MV}.
        Then by using Fourier analysis (in time and space) and assuming that $ \overline{f}(v)$ satisfies a Penrose stability condition,  one can derive relevant estimates from this integral equation  and thus estimate $\dot \rho $ in $L^2_{t,x}$ with respect to the source term (without loss of derivatives).
        
         Nevertheless, there are two major  difficulties to overcome in order to apply this strategy in Sobolev spaces and in  the general case where $f$ depends also on $t$ and $x$.

\medskip

  \noindent $\bullet$        The first one is due to subprincipal terms  in~\eqref{naive}. This comes from the fact that we do not expect that $2m$ derivatives of $f$ can be controlled uniformly in $\eps$  (only $2m$ derivatives of $\rho$
           and $2m-1$ derivatives of $f$ are). %, we cannot apply $2m$ derivatives to the equation
         %   and expect that up to a harmless commutator term $\partial^{2m}_w f$  evolves according to the linearized equation
         %    around $f$. Indeed, 
        In~\eqref{naive}, there are actually  subprincipal terms under the form  $ \partial_x  E \cdot \nabla_{v} \partial^{2m-1} f$  that involves
              $2m$ derivatives of $f$ and thus cannot be considered as a remainder.
                The idea would be to replace the fields  $\partial_{x_i}$,  by more general
                 vector fields in  order to kill this subprincipal term. However, since these fields have to depend on $x$,
                  they do not commute with $v\cdot \nabla_{x}$ anymore and thus we would  recreate a bad subprincipal term.
               A way to overcome this issue consists in
               applying to the equation powers of some well-chosen  second order differential operators
                designed to kill all bad subprincipal terms. These operators are introduced and studied   in Lemma \ref{cons}
                and   in Lemma \ref{lem-constraint}.
                By applying these operators to $f$, we obtain functions $f_{i,j}$ that satisfy two key properties. First, they can be used to control
                 $\rho$ in the sense that
                 $$
                 \int_{\R^d} f_{i,j} \, dv =\partial_x^{\alpha(i,j)} \rho +R,
                 $$
                 where $\partial_x^{\alpha(i,j)}$ is of order $2m$ and $R$ is a remainder (that is small  and well controlled in small time), see Lemma \ref{lemrestes}.
                  Furthermore, $f_{i,j}$ evolves according to the linearized equation about
                  $f$  at leading order, that is
                  $$
                   \partial_t  f_{i,j} + v \cdot \na_x  f_{i,j} +   E \cdot \nabla_v  f_{i,j}
     + \partial_x^{\alpha(i,j)} E \cdot \nabla_v f + \cdots =0,
                  $$
                  where $\cdots$ is here a shorthand for lower order terms that we can indeed handle.
           
           \medskip       
                  
       \noindent $\bullet$   Since $f$ depends on $x$,  there is a nontrivial electric field $E$
                   in the above  equation and we cannot derive an equation for the density just by inverting the free transport operator
                    and by using Fourier analysis.
                   We shall thus first make a change of variable in order to straighten the vector field
                    $$\partial_{t}+ v \cdot \nabla_{x} + E \cdot \nabla_{v} \quad \text{    into    }  \quad \partial_{t}+ \Phi(t,x,v) \cdot \nabla_{x},$$
                     where $\Phi(t,x,v)$ is a vector field close to $v$ in small time.  By using the characteristics method,
                      and a near identity change of variable,  we can then obtain an integral equation for the evolution of $(\partial_x^{\alpha(i,j)}  \rho)_{i,j}$
                       that has nice properties  (see Lemma \ref{redu} and Lemma \ref{lem-simplification}). 
                        For this stage, we need to study  integral operators under the form
                        $$  K_{G}(F) (t,x) =    \int_{0}^t \int (\nabla_{x}  F) (s,  x - (t-s) v) \cdot G(t,s,x,v) \, dv.$$
                        Note that $K_{G}$  seems to feature a loss of one derivative when acting on $F$, but we prove that it is actually a bounded operator on $L^2_{t,x}$, provided that $G(t,s,x,v)$ is smooth enough and localized  in $v$, see Proposition~\ref{propint}. This is an effect in the same spirit as kinetic averaging lemmas (see e.g. \cite{GLPS}).
 We essentially end up with the study of the integral equation (with unknown $F$)
                   $$ F =  K_{\na_v f^0}(F) + R,$$
            where  $R$ is a remainder we can control.

                       \medskip
             
                   \noindent $\bullet$ 
                      %We shall then prove that in small time this integral operator can be approximated
                    The last step of the proof  consists in relating $I- K_{\na_v f^0}$ to a semiclassical  pseudodifferential operator whose symbol is  given by the function
                     $  \mathcal{P}(\gamma, \tau, \eta, f^0_{\eps}(x,\cdot))$ (see Lemma \ref{lemKpseudo}). 
                     We therefore observe that the Penrose stability condition~\eqref{Penrose0} can be seen as a condition of \emph{ellipticity} of this symbol.
                    We can finally use a semi-classical pseudodifferential calculus with parameter in order to  invert $I- K_{\na_v f^0}$ up to  a small remainder, which yields an estimate for 
                     $\partial_{x}^{2m} \rho$ in $L^2_{t,x}$, as achieved in Proposition~\ref{proppenrose}. Note that this part of the proof
                      is very much inspired by the use of the Lopatinskii determinant in order to get estimates for initial boundary value problems for  hyperbolic systems (\cite{Kreiss, Metivier, Majda})  and the use of the Evans function in order to get estimates in singular limit problems involving stable  boundary layers
                       (\cite{Metivier-Zumbrun, Grenier-R, Gues-Metivier-Williams-Zumbrun, R-Ekman}).
                      
%                     
%                     The paper is organized as follows. We introduce in Section~\ref{prepre} some other notations and some standard estimates for weighted Sobolev spaces. The principle of the  bootstrap argument is 
%                     explained in section \ref{secPre1}.
%                      The main part of this work is made of Sections~\ref{secPre1}, \ref{secPre2}, \ref{secH2m} and \ref{secconclusion} where we provide the proof of Theorem~\ref{theomain}. 
% Section~\ref{sec23} is dedicated to the proofs of Theorems~\ref{theoquasi} and~\ref{WP}. We conclude the paper with two appendices, one for remarks about the Penrose stability condition, and one about pseudodifferential calculus with a parameter.
% 
% 

\medskip

Once these a priori estimates are obtained, it is standard to conclude the bootstrap argument, see Section~\ref{secconclusion}. In Section~\ref{sec23}, we provide the proofs of Theorems~\ref{theoquasi} and~\ref{WP}.
Theorem~\ref{theoquasi} follows from Theorem~\ref{theomain} and compactness arguments. Then, existence part in Theorem~\ref{WP} is a straightforward consequence of Theorem~\ref{theoquasi}. The uniqueness part needs a specific analysis, which is performed in the same spirit as the way we obtained a priori estimates, see Proposition~\ref{proplin} and Corollary~\ref{coruniq}.

The last section of the paper is dedicated to some elements of pseudodifferential calculus which are needed in the proof.

% \section{Additional notations and useful Sobolev estimates}
% \label{prepre}

% In addition to the functional spaces already defined in the introduction,  it will also be useful to work with the dual norm $\Hc^{k}_{-n_0}$ of $ \Hc^{k}_{n_0}$, for all $k \in \N$, $n_0 \geq 0$,
%\begin{equation}
%\| f \|_{\Hc^{k}_{-n_0}} := \sum_{|\a| + |\b| \leq k} \int_{\T^d} \int_{\R^d} (1+ |v|^2)^{-n_0} |\pa^\a_x \pa^\b_v f|^2 \, dv dx.
%\end{equation}

 \section{Proof of Theorem \ref{theomain}: setting up the bootstrap argument}
  \label{secPre1}
 
  %We shall first explain the strategy for the proof of Theorem \ref{theomain} that relies on a  bootstrap argument.
%  In the following, 
  For the proof  of the a priori estimates that will eventually lead to the proof of  Theorem \ref{theomain}, we shall systematically remove the subscripts
  $\eps$ for the solution $f_\eps$ of \eqref{VP}. The notation $A \lesssim B$ will stand as usual for $A \leq C \, B$ where $C$ is a positive number that may change from line to line but  which is independent of $\eps$ and of $A,B$. %, and is locally uniform in $T$ for $T \in \mathbb{R}_{+}$.
 Similarly, $\Lambda$ will stand for a continuous function which is independent of $\eps$ and which is non-decreasing with respect to each of its  arguments.

     \subsection{Some useful Sobolev estimates}

  Before starting the actual proof, let us state some basic product and commutator estimates  that will be very useful in the paper.
      We denote by $[A,B]= AB- BA$ the commutator between two operators.
       We shall also use in the paper the notation $\| \cdot \|_{H^k_{x,v}}$ for the standard Sobolev norm on $L^2$ for functions depending on $(x,v)$.
       In a similar way we will use the notations $\| \cdot \|_{W^{k, \infty}}$ and $\|\cdot \|_{W^{k, \infty}_{x,v}}$ for the standard Sobolev spaces
  on $L^\infty$ for functions depending on $x$ and $(x,v)$ respectively.

     \begin{lemma}
     \label{lemprod}
    Consider a smooth nonnegative  function $\chi= \chi(v)$ that satisfies $|\partial^\alpha \chi| \leq C_{\alpha} \chi$ for every $\alpha \in \mathbb{N}^d$.
     \begin{itemize}
     
            \item Consider two functions $f=f(x,v)$, $g= g(x,v)$, then we have for every $s \geq 0$,  and $k \geq s/2$
     \begin{equation}
     \label{com1}
     \| \chi f g \|_{H^s_{x,v}} \lesssim \| f\|_{W^{k, \infty}_{x,v}} \|\chi g\|_{H^s_{x,v}} +   \| g\|_{W^{k, \infty}_{x,v}} \|\chi f\|_{H^s_{x,v}}.
        \end{equation}
     
     \item Consider a  function $E= E(x)$ and a  function $F(x,v)$,  then we have  for any $s_{0}>  {d }$  and $s \geq 0$ that
      \begin{equation}
      \label{com3}
       \left \| \chi  E F  \right\|_{H^s_{x,v}} \lesssim  \|E\|_{H^{s_{0} }_{x} } \| \chi F\|_{H^{s}_{x,v}} +  \|E\|_{H^s_{x}} \|\chi F \|_{H^{s}_{x,v}}.
      \end{equation}

       \item   Consider a vector field $E= E(x)$ and a  function $f=f(x,v)$,  then we have  for any $s_{0}> 1+ {d }$ and for any $\alpha$, $ \beta  \in \mathbb{N}^d$ such that  $|\alpha |+ |\beta |= s \geq 1$ that
      \begin{equation}
      \label{com2}
       \left \| \chi \left[ \partial_{x}^\alpha \partial^\beta_{v}, E(x)\cdot \nabla_{v}\right]  f\right\|_{L^2_{x,v}} \lesssim  \|E\|_{H^{s_{0} }_{x} } \| \chi f\|_{H^{s}_{x,v}} +  \|E\|_{H^s_x} \|\chi f \|_{H^{s}_{x,v}}.
      \end{equation}

 \end{itemize}

     \end{lemma}
    Note that by taking as weight function $\chi (v)= (1 + |v|^2)^{ \pm {\sigma \over 2}}$, we can use this lemma to get estimates in $\Hc^{s}_{\pm \sigma}$. %or $\mathbb{H}^s$.
     Note that \eqref{com3},  \eqref{com1} are    not sharp  in terms of regularity but they will be sufficient for our purpose.
     \begin{proof}[Proof of Lemma~\ref{lemprod}]
      The estimate  \eqref{com1} is straightforward, using the pointwise estimates on $\chi$ and its derivatives.
       To prove \eqref{com3}, by using Leibnitz formula, we have to estimate
       $$ \| \chi \partial_{x}^\alpha  E \partial_{x}^\beta  \partial_{v}^\gamma F\|_{L^2_{x,v}}$$
        with $ | \alpha | + | \beta | + | \gamma |\leq s$.
         If $| \alpha | \leq d/2$. We write by Sobolev embedding in $x$ that
         $$   \| \chi \partial_{x}^\alpha  E \partial_{x}^\beta  \partial_{v}^\gamma F\|_{L^2_{x,v}}
          \lesssim  \|\partial_{x}^\alpha  E \|_{L^\infty_x} \| \chi F \|_{H^s_{x,v}}
          \lesssim  \| E \|_{H^{s_{0}}_x} \| \chi F \|_{H^s_{x,v}}.$$
          If $| \alpha |>d/2$,  by using again Sobolev embedding in $x$, we write
          $$  \| \chi \partial_{x}^\alpha  E \partial_{x}^\beta  \partial_{v}^\gamma F\|_{L^2_{x,v}} \lesssim \| E \|_{H^{s}} \Big( \int \sup_{x}   | \chi  \partial_{x}^\beta  \partial_{v}^\gamma F|^2\, dv\Big)^{1 \over 2 } \lesssim   \| E \|_{H^{s}} \| \chi F \|_{H^s_{x,v}}$$
           since  $ | \beta | + | \gamma | + {d \over 2} \leq s  - | \alpha| + {d \over 2} < s.$

      To prove \eqref{com2},  we proceed in a similar way. By expanding the commutator, we have to estimate
       $$ I_{\gamma} = \left\| \chi \partial_{x}^\gamma  E \cdot \nabla_{v} \partial_{x}^{\alpha - \gamma} \partial_{v}^\beta f\right\|_{L^2_{x,v}}$$
        for  $0<\gamma  \leq \alpha$ where $| \alpha |+ |\beta| =  s.$
         If $ 0< | \gamma | \leq 1 + {d \over 2}$, we write by using Sobolev embedding in $x$
         $$ I_{\gamma} \lesssim \| \partial_{x}^\gamma E\|_{L^\infty} \| \chi \nabla_{v} \partial_{x}^{\alpha - \gamma} \partial_{v}^\beta f\|_{L^2_{x,v}} \lesssim \| E \|_{H^{s_{0}}} \| \chi f\|_{H^{s}_{x,v}}.$$
          If $|\gamma|>1 + {d \over 2},$ we write by using again the Sobolev embedding in $x$
         $$ I_{\gamma} \lesssim \|\partial_{x}^{\gamma} E \|_{L^2}    \left(\int \sup_{x}  |\chi  \nabla_{v} \partial_{x}^{\alpha - {\gamma}} \partial_{v}^\beta f |^2 \, dv \right)^{1 \over 2}
          \lesssim  \|E\|_{H^s} \| \chi f\|_{H^{s}_{x,v}}$$
          since $ 1 + |\alpha |+ | \beta |- | \gamma | = 1 + s - | \gamma | < s -   {d \over 2}.$
           
     \end{proof}
     
     We shall also use the following statement.
     \begin{lemma}
     \label{lemcomdual}
      Consider  two functions $f=f(x,v)$, $g= g(x,v)$ and take $\chi(v)$ satisfying the assumptions of Lemma \ref{lemprod},   then  for every $s \geq 0$ and $\alpha, \beta \in \mathbb{N}^{2d}$
       with $| \alpha | + | \beta | \leq s$, we have the estimate
    \begin{equation}
    \label{comdual}
     \| \partial_{x,v}^\alpha f\, \partial_{x,v}^{\beta} g\|_{L^2} \lesssim  \| { 1 \over \chi} f\|_{L^\infty_{x,v}} \| \chi g \|_{H^s_{x,v}} +  \| \chi g \|_{L^\infty_{x,v}} \| {1 \over \chi} f \|_{H^s_{x,v}}.
    \end{equation}
     \end{lemma}
     \begin{proof}[Proof of Lemma~\ref{lemcomdual}]
      It suffices to notice that  since  $\chi$ and $1/\chi$ satisfy that  $|\partial^\alpha \phi | \lesssim \phi$ for every $ \alpha$, it is equivalent to estimate
       $$ \left\| \partial_{x, v}^{\tilde \alpha} \left( { 1 \over \chi} f\right) \partial_{x,v}^{\tilde \beta}( \chi g) \right\|_{L^2}$$ with $\tilde \alpha$, $\tilde \beta$ that still satisfy 
        $ | \tilde \alpha | + | \tilde \beta | \leq s$ and the result follows from the standard tame Sobolev-Gagliardo-Nirenberg-Moser inequality.
     \end{proof}

\subsection{Set up of the bootstrap}     
     
  From classical energy estimates (that we shall recall below, see Lemma~\ref{lemfacile}), we easily get that
   the Vlasov-Poisson system is locally well-posed in $\Hc^{2m}_{2 r}$ for any $m$ and $r$ satisfying $2m>1+{d}$ and $ 2r>d/2$. This means
    that if $f^0 \in \Hc^{2m}_{2r}$, there exists $T>0$ (that depends on $\eps$) and a unique  solution $f \in \mathcal{C}([0, T], \Hc^{2m}_{2r})$
     of the Vlasov-Poisson system. 
      We can thus consider a maximal solution $f \in \mathcal{C}([0, T^*), \Hc^{2m}_{2 r})$. 
     Note that since $2r>d/2$, we have for every $T\in [0, T^*)$,
   \begin{equation}
   \label{stupide}  \| \rho \|_{L^2((0,T),H^{2m})} \lesssim   T^{1 \over 2} \sup_{[0, T]} \|f\|_{\Hc^{2m}_{2r}}.
   \end{equation} 
   and hence $\Nc_{2m\, 2r}(T, f)$  is well defined for $T< T^*$.
    From this local existence result, we can thus define another maximal  time $T^\eps$ (that  a priori depends on $\eps$) as 
    \begin{equation}
     T^\eps = \sup \Big\{ T \in [0, T^*), \quad   \Nc_{2m, 2r}(T, f)  \leq R \Big\}.
        \end{equation} 
     By taking $R$  independent of $\eps$  but sufficiently large, we have by continuity that $T^\eps >0$. 
      Our aim is to prove that $R$ can be chosen large enough so that  for all  $\eps \in (0, 1]$, $T^\eps$ is uniformly  bounded from below by some time $T>0$.
       There are two   possibilities for $T^\eps$:
       \begin{enumerate}
       \item either $T^\eps= T^*$,   
       \item or $T^\eps <T^*$ and  $\Nc_{2m, 2r}(T^\eps, f) = R$.
        \end{enumerate}
        Let us first analyze the first case which is straightforward. If $T^\eps = T^*= +\infty$, then  the estimate $\Nc_{2m, 2r}(T, f) \leq R$ holds for all times and there is nothing to do.
        We shall soon show that the scenario $T^\eps = T^* <+\infty$ is impossible by using an energy estimate.
        
        \bigskip

         We shall denote by $\mathcal{T}$ the transport operator
    \begin{equation}
    \label{T}
   \mathcal{ T}:= \pa_t + v \cdot \na_x + E\cdot \na_v,
    \end{equation}
where $E$ is the electric field associated to $f$, that is $E= - \na (I - \eps^2 \Delta)^{-1} (\int_{\R^d} f \, dv -1)$.

        We first write an identity (that follows from a direct computation) which we will use many  times in this paper.
        \begin{lemma}
      For  $\a, \b \in \N^d$, we have for any smooth function $f$ the formula
      \begin{equation}
      \label{comform}
      \pa_x^\a \pa_v^\b (\Tc f) = \Tc  (\pa_x^a \pa_v^\b f) + \sum_{i=1}^d \mathds{1}_{\beta_i \neq 0} \pa_{x_i}  \pa_x^\a \pa_v^{\overline \beta^{i}} f +
       \left [ \partial_{x}^\alpha \partial_{v}^\beta, E \cdot \nabla_{v} \right] f,
%       \sum_{0<k \leq \alpha} \begin{pmatrix} \alpha \\ k \end{pmatrix} \pa_x^k E \cdot \na_v \pa_x^{\alpha-k} \pa_v^\beta f
      \end{equation}
      where $\overline \beta^{i}$ is equal to $\beta$ except that $\overline \beta^{i}_i= \beta_i - 1$. %(recall $\beta_i\neq 0$).  
      \end{lemma}
      
      The $\Hc^{2m}_{2r}$ energy estimate reads as follows.
        \begin{lemma}
        \label{lemfacile}
      For any solution
  $f$  to \eqref{VP}, we have,  for some $C>0$ independent of $\eps$, the estimate
        \begin{equation}
         \sup_{[0, T^\eps)} \|f (t) \|_{\Hc^{2m}_{2r} }^2  \leq \|f^0 \|_{\Hc^{2m}_{2r} }^2 \exp  \left[C\left( T^\eps  +{ 1 \over \eps} (T^{\eps})^{1 \over 2} R \right)\right].
         \end{equation}
          \end{lemma}
      
      \begin{proof}[Proof of Lemma~\ref{lemfacile}]
      Using \eqref{comform}, for   $f$ satisfying \eqref{VP} and thus  $\Tc f =0$, we can  use the commutator formula~\eqref{comform}, take  the scalar product with $(1+ |v|^2)^{2 r} \pa^\a_x \pa^\b_v f$, and sum for all $|\a| + |\b| \leq 2m$.
       By using \eqref{com2} with  $s = 2 m$, $\chi (v)= (1+ |v|^2)^{r}$ and $s_{0}= 2 m \,$  (recall that $2m>1+{d})$,
       we get
       $$
         \left \| \chi \left[ \partial_{x}^\alpha \partial^\beta_{v}, E(x)\cdot \nabla_{v}\right]  f\right\|_{L^2_{x,v}} \lesssim  \| E\|_{H^{2m}}\, \|f \|_{\Hc^{2m}_{2n_0} }.
       $$
       By Cauchy-Schwarz we thus have
       $$
       \left| \int \chi \left[ \partial_{x}^\alpha \partial^\beta_{v}, E(x)\cdot \nabla_{v}\right]  f \, \ \chi\pa^\a_x \pa^\b_v f\right| 
       \lesssim\| E\|_{H^{2m}}\, \|f \|_{\Hc^{2m}_{2r} }^2.
       $$
       We end up with a classical energy estimate
 $$    \frac{d}{dt}  \|f \|_{\Hc^{2m}_{2r} }^2 \lesssim \|f \|_{\Hc^{2m}_{2r} }^2 +    \| E\|_{H^{2m}}\, \|f \|_{\Hc^{2m}_{2r} }^2.$$
     By using  
       the 
    elliptic regularity  estimate  for the Poisson equation which gives 
    $$
    \| E\|_{H^{2m}} = \| \nabla_{x} V \|_{H^{2m}}\lesssim \frac{1}{\e}  \| \rho\|_{H^{2m}},
    $$
    we obtain that for $t \in [0, T^\eps)$ and for some $C>0$ independent of $\eps$,
    $$   \|f (t)  \|_{\Hc^{2m}_{2r} }^2\leq   \|f^{0} \|_{\Hc^{2m}_{2r} }^2 +  C  \int_{0}^t \left(    {1 \over \eps} \| \rho\|_{H^{2m}} + 1  \right) \|f(s) \|_{\Hc^{2m}_{2r} }^2$$
     for some $C>0$.
     Consequently, from the Gronwall inequality, we obtain that
     \begin{align*} \sup_{[0, T^\eps)} \|f (t) \|_{\Hc^{2m}_{2r} }^2 &  \leq  \|f^0 \|_{\Hc^{2m}_{2r}}^2 \exp \left[C \left(  T^\eps + {1 \over \eps} (T^{\eps})^{1 \over 2} \| \rho\|_{L^2([0, T^\eps), H^{2m})} \right) \right]
      \\
      &  \leq  \|f^0 \|_{\Hc^{2m}_{2r} }^2 \exp  \left[C\left( T^\eps  +{ 1 \over \eps} (T^{\eps})^{1 \over 2} \Nc_{2m, \,2r}(T^\eps,f) \right)\right],
      \end{align*}
            from which, since $\mathcal{N}_{2m, \, 2r}(T^\eps, f) \leq R$, we obtain the expected estimate.
      
         \end{proof}
       In particular, if $T^\eps = T^*<+\infty $, we have
       $$  \sup_{[0, T^*)} \|f (t) \|_{\Hc^{2m}_{2r} }^2 \leq   \|f^0 \|_{\Hc^{2m}_{2r}}^2 \exp \left[C \left( T^\eps  +  {1 \over \eps} (T^{\eps})^{1 \over 2}  R \right)\right]<+ \infty.$$

        This means that the solution could be continued beyond $T^*$ and this contradicts the definition of $T^*$ and so
         this case is impossible.

      \bigskip
         %  To  prove Theorem \ref{theomain}  
         
         Therefore, let us assume from now on that $T^\eps <T^*$ and  $\Nc_{2m, \, 2r}(T^\eps, f) = R$.
         We shall estimate $\Nc_{2m, \, 2r}(T, f)$ and prove that for some well chosen parameter $R$ (independent of $\eps$), there exists some time $T^\# >0$, small but independent of $\eps$,   such that the equality
       $$ \Nc_{2m, \, 2r}(T, f)= R$$ cannot hold  for any  $T \in [0, T^\#]$. %small but independent of $\eps.$ 
       We will then deduce that $T^\eps > T^\#$.
      
        To this end,  we need to estimate $\Nc_{2m, \, 2r}(T, f)$.  To estimate the  
        part   $ \| f_\e \|_{L^\infty((0,T),\Hc^{2m-1}_{2r})}$ of the quantity, we can proceed by standard energy estimates as above.  Then the main part of the work
         will be to control $ \| \rho_\e \|_{L^2((0,T),H^{2m})} $ uniformly in $\eps$ by using the Penrose stability condition. Note that we cannot use the estimate \eqref{stupide} to get a control
          that is independent of $\eps$ since estimating $\|f\|_{\Hc^{2m}_{2r}}$ in terms of $\Nc_{2m, \, 2r}(T, f)$ requires the use of the elliptic regularity provided by the Poisson equation, and thus costs negative powers  of $\eps$.

 \bigskip
% \section{Estimate of $\| f(t) \|_{\Hc^{2m-1}}$}
 
   We end this section with the  $\Hc^{2m-1}_{2r}$ energy estimate without loss in $\eps$.
     
      \begin{lemma}%[Energy estimates]
      \label{lemfacilesans}
  \label{energy}
 For $2m>2+{d }$ and $2r>d/2$,  we have for any solution
  $f$  to \eqref{VP}    the estimate
  \begin{align}
  \label{eq-energy}
  & \sup_{[0, T]}\| f \|_{\Hc^{2m-1}_{2r}}  \leq\| f^0 \|_{\Hc^{2m-1}_{2r}}+   T^{1\over 2} \Lambda (T, R),   
  \end{align}
  for every $T \in [0, T^\eps)$.
  \end{lemma}
  
      \begin{proof}[Proof of Lemma~\ref{lemfacilesans}]
      Let $\a, \b \in \N^d$ with $|\alpha |+ |\beta | = 2m-1 $.  We can use again the commutation formula \eqref{comform}  take  the scalar product with $(1+ |v|^2)^{2n_0} \pa^\a_x \pa^\b_v f$,  sum for all $|\a| + |\b| \leq 2m-1$ and use  \eqref{com2} with  $s = 2 m-1$, $\chi (v)= (1+ |v|^2)^{r}$ and $s_{0}= 2 m-1$ (which is licit since $2m>2+{d}$).
      We obtain that 
      \begin{equation}
      \label{energieformel}
         \frac{d}{dt} \|f \|_{\Hc^{2m-1}_{2n_0}}^2 \lesssim \| f\|^2_{\Hc^{2m-1}_{2n_0}} + \| E\|_{H^{2m-1}} \| f\|_{\Hc^{2m-1}_{2n_0}}^2.
    \end{equation}
   Integrating in time we obtain that for every $T \in [0, T^\eps]$, for some $C>0$
    $$ \sup_{[0, T]} \| f \|_{\Hc^{2m-1}_{2r}} \leq \| f^0 \|_{\Hc^{2m-1}_{2r}} + C  \sup_{[0, T]} \| f \|_{\Hc^{2m-1}_{2r}} \left(T  +   \int_{0}^T \| E \|_{H^{2m-1}}\, dt\right).$$
    By using  Cauchy-Schwarz in time  and the crude estimate
    \begin{equation}
    \label{estE1}
    \| E\|_{H^{2m-1}}= \| \nabla_{x} V \|_{H^{2m-1}} \lesssim \| \rho \|_{H^{2m}}
    \end{equation}
    which is uniform in $\eps$ since it does not use any elliptic regularity,  we obtain
    that
  $$  \sup_{[0, T]} \| f \|_{\Hc^{2m-1}_{2r}} \leq \| f^0 \|_{\Hc^{2m-1}_{2r}} +  C R (T + T^{ 1 \over 2} R).$$ 
  This proves the estimate \eqref{eq-energy}. 
       \end{proof}

  \section{Proof of Theorem \ref{theomain}: preliminaries for the estimates on $\rho$}
    \label{secPre2}
 
       \subsection{Definition of appropriate second order differential operators}
   
In order to estimate the $H^{2m}$ norm of $\rho$, we need to introduce appropriate differential operators of order $2m$ which are well adapted to the Vlasov equation in the quasineutral scaling. The usual basic approach is to use the vector fields  $\pa_x$, 
 $\partial_{v}$ and thus to apply $\partial^\alpha $ with $|\a|\leq 2m$ to the Vlasov equation. 
  The hope is that up to harmless commutators, $\partial^\alpha f$ will evolve according to the linearized
   equation about $f$ and thus that we will just have to understand the dynamics of this linearized equation.
      Nevertheless, there are unbounded terms arising because of  commutators.  
  the main problem  is the subprincipal term $\partial E \cdot \nabla_v \partial^{2m-1} f$ that involves $2m$ derivatives
  of $f$ and thus cannot be controlled by $\Nc_{2m}(t,f)$ uniformly in $\eps$.
   As explained above, we could try to use more complicated variable coefficients vector fields designed  to kill this commutator term. 
    But since these vector fields have to depend on $x$ they would not commute any more with the free transport operator
     $v \cdot \nabla_{x}$ and thus we would recreate another bad subprincipal commutator.
      This heuristics 
    motivates the analysis of this section. It turns out that the following second order operators (and their composition) are relevant
     since 
    they have  good commutation properties with the transport operator $\mathcal{T}$ (which was defined in~\eqref{T}).
    
    \begin{lemma}%[Introduction of appropriate second order operators]
    \label{cons}
 Let $(\varphi_{k,l}^{i,j}$, $\psi_{k,l}^{i,j})_{i,j,k,l \in \{1,\cdots,d\}}$ be smooth solutions of the system:
      \begin{equation}
      \label{eq-constraint}
      \left\{
\begin{aligned}
&\Tc  \varphi_{k,l}^{i,j} = \psi_{k,l}^{i,j} + \psi_{l,k}^{i,j}  -\sum_{1\leq k',l' \leq d}  \varphi_{k',l'}^{i,j}   \varphi_{k,l}^{k',l'}  + \delta_{k,j} \pa_{x_i}E_l + \delta_{k,i} \pa_{x_j}E_l, \\
&\Tc \psi_{k,l}^{i,j} =-   \sum_{1\leq k',l' \leq d}  \varphi_{k',l'}^{i,j}  \psi_{k,l}^{k',l'} +  \varphi_{k,l}^{i,j} \pa_{x_k} E_k.
\end{aligned}
\right.
    \end{equation}
    We assume that for all $k,l$, the matrices $(\varphi_{k,l}^{i,j})_{1\leq i,j \leq d}$ and $(\psi_{k,l}^{i,j})_{1\leq i,j \leq d}$ are symmetric. Introduce the second order operators
    \begin{equation}
    \label{def-L}
    L_{i, j} := \pa_{x_i x_j} + \sum_{1\leq  k,l \leq d} \left( \varphi_{k,l}^{i,j}\pa_{x_k} \pa_{v_l} + \psi_{k,l}^{i,j} \pa_{v_k v_l} \right).
    \end{equation}
    Then for all smooth functions $f$, we have the formula
        \begin{equation}
        \label{LT}
L_{i,j} \Tc(f) = \Tc L_{i,j} (f) +  \pa_{x_i x_j} E \cdot \na_v f + \sum_{k,l} \varphi^{i,j}_{k,l} L_{k,l} f.
    \end{equation}
  %  with the Einstein summation convention.
    
    \end{lemma}
    
    \begin{proof}[Proof of Lemma~\ref{cons}]
    We have by direct computations
    \begin{align*}
    \pa_{x_i x_j} (\Tc f) &= \Tc (  \pa_{x_i x_j} f) + \pa_{x_i x_j} E \cdot \na_v f 
    + \pa_{x_i} E \cdot \na_v \pa_{x_j} f 
    + \pa_{x_j} E \cdot \na_v \pa_{x_i} f, \\
    %%%%%%%%%%%%%%%%
    \varphi_{k,l}^{i,j}\pa_{x_k} \pa_{v_l} (\Tc f) &= \Tc ( \varphi_{k,l}^{i,j}\pa_{x_k} \pa_{v_l} f)
     + \varphi_{k,l}^{i,j} \left(\pa_{x_k x_l}f 
     + \pa_{x_k} E  \cdot \na_v \pa_{v_l} f\right)
     -\Tc  (\varphi_{k,l}^{i,j}) \pa_{x_k} \pa_{v_l} f, \\
     %%%%%%%%%%%%%%%
      \psi_{k,l}^{i,j} \pa_{v_k v_l} (\Tc f) &= \Tc ( \psi_{k,l}^{i,j} \pa_{v_k v_l} f) 
      +  \psi_{k,l}^{i,j}( \pa_{v_k} \pa_{x_l} f + \pa_{v_l} \pa_{x_k} f) 
      -  \Tc (\psi_{k,l}^{i,j}) \pa_{v_k v_l} f.
    \end{align*}
    We can rewrite
    $$
     \varphi_{k,l}^{i,j} \pa_{x_k x_l}f  =    \varphi_{k,l}^{i,j}  \left( L_{k,l} f - \sum_{k',l'} \left( \varphi_{k',l'}^{k,l}\pa_{x_{k'}} \pa_{v_{l'}} + \psi_{k',l'}^{k,l} \pa_{v_{k'} v_{l'}} \right)f \right),
    $$
    which entails that 
    \begin{align*}
   L_{i,j} \Tc (f) &= \Tc  L_{i,j} (f) +  \pa_{x_i x_j} E \cdot \na_v f + \sum_{k,l} \varphi^{i,j}_{k,l} L_{k,l} f \\
   &+ \sum_{k,l} \pa_{x_k} \pa_{v_l} f \left[-\Tc  \varphi_{k,l}^{i,j} + \psi_{k,l}^{i,j} + \psi_{l,k}^{i,j}  -\sum_{k',l'}  \varphi_{k',l'}^{i,j}   \varphi_{k,l}^{k',l'}  + \delta_{k,j} \pa_{x_i}E_l + \delta_{k,i} \pa_{x_j}E_l    \right] \\
   &+ \sum_{k,l} \pa_{v_k v_l} f \left[ -\Tc  \psi_{k,l}^{i,j} -   \sum_{k',l'}  \varphi_{k',l'}^{i,j}  \psi_{k,l}^{k',l'}  +  \varphi_{k,l}^{i,j} \pa_{x_k} E_k  \right].
   \end{align*} 
   We therefore deduce \eqref{LT}, because of \eqref{eq-constraint}.
    \end{proof}
 
  We shall now study the Sobolev regularity of the solution of the constraint equations \eqref{eq-constraint}.
        \begin{lemma}%[The constraint equation]
        \label{lem-constraint} Assume $2m>2+d$ and $2r >d$. There exists $T_{0}= T_{0}(R)>0$ independent of $\eps$ such that  for  every    $ T <  \min (T_0, T^\eps)$,  there exists a unique  solution $(\varphi_{k,l}^{i,j}, \psi_{k,l}^{i,j})_{i,j,k,l}$ on $[0,T]$ of \eqref{eq-constraint}
        satisfying   
        $$
        \varphi_{k,l}^{i,j}\vert_{t=0} =  \psi_{k,l}^{i,j}\vert_{t=0} =0.
        $$
       Moreover, we have  the estimates
              \begin{equation}
    \label{energy-constraintbis}
   \sup_{[0, T]} \sup_{i,j,k,l}  \|(\varphi_{k,l}^{i,j}, \psi_{k,l}^{i,j}) \|_{W^{p, \infty}_{x,v}} \leq  T^{1 \over 2} \Lambda(T,R),   \quad p < 2m - d/2 - 2,
       %C(  \| \rho\|_{L^\infty((0,T), H^{2m})}^2 )T^{ 1 \over 2 }  \left(\int_0^T \| \rho\|_{H^{2m}}^2 \, ds\right)^{1 \over 2 },.
    \end{equation}
    \begin{equation}
      \label{energy-constraint}
   \sup_{[0, T]} \sup_{i,j,k,l} \|(\varphi_{k,l}^{i,j}, \psi_{k,l}^{i,j}) \|_{\Hc_{-r}^{2m-2}} \leq T^{1 \over 2} \Lambda(T,R).
\end{equation}
       Finally, for all $k,l$, the matrices $(\varphi_{k,l}^{i,j})_{1\leq i,j \leq d}$ and $(\psi_{k,l}^{i,j})_{1\leq i,j \leq d}$ are symmetric.

%    {\color{red}F:  J'ai ajout\'e l'estimation ci-dessus qui d\'ecoule du principe du maximum (cette fois en $x et v$ donc c'est vrai), on en a besoin ensuite et je pense
%    qu'on en a aussi besoin pour boucler les estimations d'\'energie, avec seulement le contrôle de \eqref{energyconstraint} on n'a pas de bornes $L^\infty$ donc il y a peu
%    de chance que l'on puisse borner les termes quadratiques avec seulement \c{c}\`a ! 
%    
%    \bigskip
%    
%    D: ok, j'ai l'impression que c'est ok pour $ k < 2m - d/2 - 2$ au lieu de 
%    $ k < 2m - d - 2$ (certes ca change pas grand chose)? }
%    
        \end{lemma}
    
        \begin{proof}[Proof of Lemma~\ref{lem-constraint}]
       System \eqref{eq-constraint} is a system of  semi-linear  transport equations that is coupled only via a zero order term.   
        The existence and uniqueness  of smooth solution can be obtained by a  standard  fixed point argument  and thus
         we shall  only focus on a priori estimates.
            Note that 
       the symmetry of the matrices $(\varphi_{k,l}^{i,j})_{1\leq i,j \leq d}$ and $(\psi_{k,l}^{i,j})_{1\leq i,j \leq d}$  for all $k,l$  is a consequence of the  uniqueness of the solution.
    
     Let us set 
     $$M_{p}(T)=  \sup_{[0, T]}\sup_{i,j,k,l} \|(\varphi_{k,l}^{i,j}(t), \psi_{k,l}^{i,j}(t)) \|_{ W^{p, \infty}_{x,v} }.$$
      We first apply $ \partial_{x}^\alpha \partial_{v}^\beta$ for $|\alpha|+ |\beta| \leq p$  to \eqref{eq-constraint} and use again the commutation formula
       \eqref{comform}. By using the maximum principle for the transport operator $\mathcal{T}$, we get that
       $$ M_{p}(T) \lesssim T \left(1 + M_{p}(T)\right) M_{p}(T) + (1 + M_{p}(T)) \int_{0}^T \| \nabla E\|_{W^{p, \infty}} \, dt .$$
        By Sobolev embedding and  \eqref{estE1}, we obtain
        $$ \int_{0}^T \| \nabla E\|_{W^{p, \infty}} \, dt \lesssim  \int_{0}^T \| E\|_{H^{2 m- 1}} \, dt \lesssim T^{1 \over 2} R$$
         since $ 1 +  p+ {d \over 2} < 2m-1$
         and hence we find
       $$ M_{p}(T) \lesssim T^{1 \over 2} R + ( T + T^{1 \over 2} R)  M_{p}(T) +  M_{p}(T)^2.$$
       Consequently, there exists $\gamma_{0}>0$ independent of $\eps$ but sufficiently small   such that for every $T \leq T_{0}$
        satisfying $ ( T_{0}+ T_{0}^{1 \over 2} R) \leq  \gamma_{0}$, we get the estimate
        $$ M_{p}(T) \leq  T^{1 \over 2} R.$$

\bigskip

  For the estimate  \eqref{energy-constraint}, we apply $\partial_{x}^\alpha \partial_{v}^\beta$ with $|\alpha | + | \beta | \leq 2m-2$  to the system \eqref{eq-constraint}, 
   we use the commutation formula \eqref{comform} with the transport operator, multiply by the weight
    $( 1 + |v|^2 )^{ - {r}}$ and take the scalar product by  $\partial_{x}^\alpha \partial_{v}^\beta (\varphi_{k,l}^{i,j}, \psi_{k,l}^{i,j})$  as usual.
      Let us set
   $$ Q_{2m-2} (t)= \left( \sum _{i,j,k,l} \|(\varphi_{k,l}^{i,j}, \psi_{k,l}^{i,j}) (t) \|_{\Hc_{-r}^{2m-2}}^2\right)^{1 \over 2}.$$
    By using 
    \begin{itemize}
    
    \item the product estimate \eqref{com1} with $s= 2 m-2$,  $k=s/2= m-1$, to handle the quadratic terms in the right hand side,
    
    \item the commutator estimate  \eqref{com2}  with $s= 2 m-2$,  $s_{0}= 2m-1 \, (> 1+ d) $,  to handle the commutators with $E$,
    
    \end{itemize}
    
     we obtain for $t \in [0, T]$, $T < \min (T_{0}, T^\eps)$ that 
    \begin{multline*}
      \frac{d}{dt} \| (\varphi_{k,l}^{i,j}, \psi_{k,l}^{i,j})\|_{\Hc_{-r}^{2m-2}}^2 \lesssim (1+  M_{m-1}(T)+   \sup_{[0, T]} \| E\|_{H^{2m-1}} ) Q_{2m-2}(t)^2
       \\ + \left \| { 1 \over ( 1 + |v|^2)^{r \over 2}   } E \right\|_{H^{2m-1}_{x,v}}  \| (\varphi_{k,l}^{i,j}, \psi_{k,l}^{i,j})\|_{\Hc_{-r}^{2m-2}}.
   \end{multline*}
    Note that since  $2r>d$, we have that
  $$ \left \| { 1 \over ( 1 + |v|^2)^{s \over 2}   } E \right\|_{H^{2m-1}_{x,v}}  \lesssim \left\| E \right\|_{H^{2m-1}_{x}}.$$
  Consequently,   we can   sum  over  $i,j,k,l$,  and  use  \eqref{estE1}  and the estimate \eqref{energy-constraintbis}
  (since $m-1< 2m- { d \over 2} -2$ is equivalent to $2m>2+d$), to obtain that 
 $$ {d \over dt} Q_{2m-2}(t) \lesssim  ( 1 + T^{1 \over 2 } \Lambda(T, R)) Q_{2m-2}(t) +   \| \rho \|_{H^{2m}}.$$
   The estimate  \eqref{energy-constraint} thus  follows from the Gronwall inequality and Cauchy-Schwarz.

%    For what concerns \eqref{energy-constraint}, we proceed exactly as in the proof of Lemma \ref{energy}. The zero order linear (with constant coefficients)  and quadratic terms (via product laws in Sobolev spaces and the $W^{k,\infty}$ estimates of \eqref{energy-constraintbis}) are harmless. The only significant contribution comes from the linear terms with coefficients involving (at most $2m-1$) space derivatives of  $ E$. We end up with:
%    $$
%      \frac{d}{dt} \| (\varphi_{k,l}^{i,j}, \psi_{k,l}^{i,j})\|_{\Hc_{-n_0}^{2m-2}} \lesssim (1+ \| (\varphi_{k,l}^{i,j}, \psi_{k,l}^{i,j})\|_{W^{m-1,\infty}}+   \| E\|_{H^{2m-1}})\| (\varphi_{k,l}^{i,j}, \psi_{k,l}^{i,j})\|_{\Hc_{-n_0}^{2m-1}} + \| E\|_{H^{2m-1}},
%    $$
%    from which we infer that for some $C_0>0$, for all $t\geq 0$,
%    \begin{multline*}
%    \|(\varphi_{k,l}^{i,j}(t), \psi_{k,l}^{i,j}(t)) \|_{\Hc_{-n_0}^{2m-2}}
%    \lesssim \\
%      \left(\int_0^t  \| E\|_{H^{2m-1}} \, ds\right)  \exp \left(C_0 [ t(1+\| (\varphi_{k,l}^{i,j}, \psi_{k,l}^{i,j})\|_{L^\infty((0,t),W^{m-1,\infty})})  + \sqrt{t} \| \rho \|_{L^2((0,t), H^{2m})}]\right).
%    \end{multline*}
%    which gives \eqref{energy-constraint}.
%    
%    

    \end{proof}
    
    We now study the effect of composing $L_{i,j}$ operators.
    
        \begin{lemma}
        \label{lemrestes}
Assume $2m>3+d$ and $2r>d$. For $i, j \in \{1,\cdots, d\}^m$, define
        \begin{equation}
 f_{i, j} = L^{(i,j)}f:= L_{i_1, j_1} \cdots L_{i_m, j_m} f
    \end{equation}
 For every $T < \min (T_{0}, T^\eps)$,  we first have that for $i,j \in \{1,\cdots, d\}^m$  
    \begin{equation}
    \label{reduit}
     \int_{\R^d} f_{i,j} \, dv = \int_{\R^d} \pa_x^{\a(i,j)} f \, dv + \mathcal{R} = \pa_x^{\a(i,j)} \rho +\mathcal{ R},
         \end{equation}
 with $\alpha(i,j)=(\alpha_k(i,j))_{1\leq k \leq d}$,  $\alpha_k(i,j)$ being equal to the number of occurrences of $k$ in the set $\{i_1,\cdots, i_m, j_1, \cdots, j_m\}$ and $\mathcal{R}$ is a remainder satisfying
  for $2m>  3 + d$ 
    \begin{equation}
    \label{rest1}
 \| \mathcal{R} \|_{L^\infty((0,T),L^2_{x})} \lesssim  \Lambda(T, R).
  %\| f \|_{\Hc^{2m-1}}.
       \end{equation}
      % {\color{red}  F: je crois que l'estimation fait aussi intervenir $f$}
 Moreover, for $f$ satisfying \eqref{VP}, we have that $f_{ij}$ solves 
    \begin{equation}
    \label{eq-fij}
\Tc (f_{i,j}) +  \pa_x^{\alpha(i,j)}  E \cdot \na_v f + \mathcal{M}_{i,j} \mathcal{F} = F_{i,j}, 
       \end{equation}
 where 
  \begin{equation}
  \label{eq-Mij}
 \mathcal{F} := (f_{i,j})_{i, j \in \{1,\cdots, d\}^m}, \quad \mathcal{M}_{i,j} \mathcal{F} := \sum_{k=1}^{m} \sum_{k',l'} \varphi^{i_{k}, j_{k} }_{k',l'}   f_{i_{k \to k'}, j_{k \to l'}},
  \end{equation}
     where $i_{ s \to t}$ denotes the element of $\{1,\cdots,d\}^m$ which is equal to $i$ except for its $s$-th element which is equal to $t$ and $F= (F_{i,j})$ is a remainder satisfying
    \begin{equation}
    \label{restF}
\left\|  F  \right\|_{L^2 ((0,T), \Hc^{0}_{n_{0}})} \leq \Lambda(T,R), \quad  \forall T < \min (T_{0}, T^\eps).
       \end{equation}

    \end{lemma}
   
       \begin{proof}[Proof of Lemma~\ref{lemrestes}]
       At first, we can expand $f_{i, j}=  L_{i_1, j_1} \cdots L_{i_m, j_m} f$ in a more tractable form.
        Let us  set  $U= ( \varphi^{i_{\alpha}, j_{\beta}}_{k,l}, \psi^{i_{\alpha},j_{\beta}}_{k,l})_{1\leq k,l\leq d, \, 1 \leq \alpha, \beta \leq  m}$.
     Then, we can write
     \begin{align} f_{i,j}
     \label{Lmexp}
     & = \partial_{x}^{\alpha (i,j)} f + \sum_{s=0}^{2m-2} \sum_{e, \, \alpha, \, k_{0}  \cdots, k_{s}}  P_{s, e,\alpha}^{k_{0}}(U) P_{s,e, \alpha}^{k_{1}}(\partial U) \cdots P_{s,e,\alpha}^{k_{s}}(\partial^{s} U) \partial_{v}^e \partial^{\alpha} f 
     \\ \nonumber &  =:  \partial_{x}^{\alpha (i,j)} f  + 
       \sum_{s=0}^{2m-2} \sum_{e, \, \alpha, \, k_{0}\cdots k_{s}} \mathcal{R}_{s,e, \alpha}^{k_{0},\cdots k_{s}},
       \end{align}
     where  the sum is  taken on indices such that
      \begin{equation}
      \label{indices}
       |e|= 1, \, | \alpha |= 2m - 1-s, \, k_{0} + k_{1}+ \cdots k_{s} \leq m , \, k_{0} \geq 1, \, k_{1} + 2 k_{2}+ \cdots s k_{s}= s. 
      \end{equation}
       and   $ (P_{s,e, \alpha}^{k_{i}}(X))_{0 \leq i \leq s}$ are  polynomials of degree smaller than   $k_{i}$ (we denote by $\partial^{k} U$ the vector made of all the partial derivatives of length $k$ of all components of $U$). The existence of an  expansion under this form  can be  easily proven by induction.
         We can set
         $$ \mathcal{R}=\int_{\mathbb{R}^d}  \sum_{s=0}^{2m-2} \sum_{e, \, \alpha, \, k_{0},\cdots, k_{s}} \mathcal{R}_{s,e, \alpha}^{k_{0},\cdots, k_{s}}\, dv$$
       so that we  have to estimate $\int_{\mathbb{R}^d}  \mathcal{R}_{s,e, \alpha}^{ k_{0},\cdots, k_{s}} \, dv$.
        All the following estimates are uniform in time  for $t \in [0,T]$ with $T \leq \min(T_{0}, T^\eps)$ but we do not mention
         the time parameter for notational convenience.
        
        Let us start with the case where 
          $\mathcal{R}_{e,s, \alpha}^{ k_{0},\cdots k_{s}}$ contains the maximal number of derivatives applied to $f$
           that is to say  when   $|\alpha |= 2m-1$ so that $2m$ derivatives of $f$ are involved.
            In this case, we have $s=0$ and  hence 
            $$\int_{\mathbb{R}^d}  \mathcal{R}_{e,0, \alpha}^{ k_{0}}\, dv =  \int_{\mathbb{R}^d}   P_{s, e,\alpha}^{k_{0}}(U) \partial_{v} \partial^\alpha f \, dv$$
            where $k_{0}$ is of degree less than $m$. We can
             thus  integrate by parts in $v$  to obtain that
             $$ \left\| \int_{\mathbb{R}^d}  \mathcal{R}_{e,0, \alpha}^{k_{0}} \, dv\right\|_{L^2_{x}}
              \lesssim \Lambda (\| U\|_{W^{1, \infty}_{x,v}}) \left \| \int_{\mathbb{R}^d}  |\partial^\alpha f | \, dv \right\|_{L^2_{x}} \lesssim  \Lambda (\| U\|_{W^{1, \infty}_{x,v}} ) \|f \|_{\Hc^{2m-1}_{r}}.$$
               Consequently, by using \eqref{energy-constraintbis} (since by assumption on $m$, we have  $1< 2 m -  2  - {d \over 2}$) and \eqref{eq-energy}
               (since $2r>d$), we obtain
          that  on $[0,T]$,
              $$\left\| \int_{\mathbb{R}^d}  \mathcal{R}_{e,0, \alpha}^{k_{0}} \, dv\right\|_{L^2_{x}} \leq  \Lambda (T, R).$$
            
            It remains to estimate the  terms for which  $s \geq 1$. Note that for all these  terms the total number of derivatives applied to $f$ is at most  $2m-1$.

\medskip
            
\noindent $\bullet$  When $ s< 2m-{ d \over 2} -2$, we can use \eqref{energy-constraintbis} to obtain that
             $$ \|P_{s, e,\alpha}^{k_{0}}(U) P_{s,e, \alpha}^{k_{1}}(\partial U) \cdots P_{s,e,\alpha}^{k_{s}}(\partial^{s} U)\|_{L^\infty_{x,v}} \leq \Lambda (T,R)$$
              and hence we obtain as above that
              $$  \left\| \int  \mathcal{R}_{e,s, \alpha}^{k_{0},  \cdots , k_{s}} \, dv\right\|_{L^2_{x}} \leq  \Lambda (T, R) \|f\|_{\Hc^{2m-1}_{r}} \leq  \Lambda (T, R).$$
              
     \medskip
            
\noindent $\bullet$  Let us now consider  $s \geq 2m- 2 - {d \over 2}.$
             Let us start with  the case  where in the sequence $(k_{1}, \cdots, k_{s})$, the  bigger index $l$ such that $k_{l} \neq 0$ and $k_{p}$ = 0 for every $p >l$
              is such that $l >s/2$. In this case, since $l k_l \leq s$, we necessarily have $k_{l}= 1$. Moreover, for the indices $p<l$ such that $k_{p} \neq 0$, we must have $p\leq p k_p <s/2$.
                Thus, we can use \eqref{energy-constraintbis} to estimate $\| \partial^{p} U \|_{L^\infty_{x,v}}$ provided
                 $s/2 \leq 2m - {d \over 2} - 2$. Since $s\leq 2m-2$, this is verified thanks 
                 to the assumption that $ 2 m >2+d$.
                 % $s \geq 2m- 2 - { d \over 2}$ this is possible thanks to the assumption that $ 2 m >2+d$.
               We thus obtain   that
               $$   \left\| \int  \mathcal{R}_{e,s, \alpha}^{k_{0}, \cdots k_{s}} \, dv\right\|_{L^2_{x}} \leq \Lambda (T,R) \left  \| \int \partial^{{l}} U
                \partial_{v}^e  \partial^\alpha f \, dv \right \|_{L^2_{x}}.
%                \leq \Lambda(T, R) \left  \| \int \partial^{k_{l}} U
%                \partial_{v}^e  \partial^\alpha f \right \|_{L^2_{x}}.
                $$
                 Next, we can use that 
               \begin{align*}
                 \left  \| \int \partial^{{l}} U
                \partial_{v}^e  \partial^\alpha f \right \|_{L^2_{x}}  & \lesssim \left \| \|{ 1 \over (1 + |v|^2)^{r \over 2 } } \partial^{{l}} U \|_{L^2_{v}} \|  (1 + |v|^2)^{r \over 2}\partial_{v}^e \partial^\alpha f \|_{L^2_{v}} \right \|_{L^2_{x}} \\
                &   \lesssim  \|  U \|_{\Hc_{-r}^{2m-2}} \sup_{x} \|   (1 + |v|^2)^{r \over 2} \partial^e_v \partial^\alpha f \|_{L^2_{v}}.
                \end{align*}
                  By Sobolev embedding in $x$, we have
                  $$   \sup_{x} \|  (1 + |v|^2)^{r \over 2}\partial^e_v \partial^\alpha f \|_{L^2_{v}} \lesssim \|f\|_{\Hc^{2m-1}_{r}}$$
                   as soon as  $2m-1> 1 + | \alpha | + {d \over 2}=  1 + 2m-1 - s + {d \over 2}$ which  is equivalent to 
                    $s > 1 + {d \over 2}$. Since we are in the case where $s    \geq 2m- 2 - {d \over 2}$, the condition is matched since $ 2 m >3+d$.
                     Consequently, by using \eqref{eq-energy} and \eqref{energy-constraint}, we obtain again  that
               $$  \left\| \int  \mathcal{R}_{e,s, \alpha}^{k_{0}, \cdots k_{s}} \, dv\right\|_{L^2_{x}} \lesssim \Lambda(T, R).$$
                Finally, it remains to handle the case where $k_{l}=0$ for every $l>s/2$. Due to the assumption that  $2m>3+d$, we  have by the same argument as above 
                 that  since $s\leq 2m-2$, %$s \geq  2 m - 2 - {d \over 2}$ 
                 we necessarily have  ${s \over 2} < 2 m - {d \over 2} - 2$ 
                  and hence by using again \eqref{energy-constraintbis} we  find
                  $$ \|\partial^l U \|_{L^\infty_{x,v}} \leq \Lambda (T, R), \quad l \leq s/2.$$
                   We deduce
                   $$   \left\| \int  \mathcal{R}_{e,s, \alpha}^{k_{0}, \cdots k_{s}} \, dv\right\|_{L^2_{x}} 
                    \leq \Lambda (T, R)  \|f \|_{\Hc^{2m-1}_{r}} \leq \Lambda(T, R).$$
             %             \medskip  
           This ends the proof of \eqref{rest1}.

\bigskip

To prove \eqref{eq-fij}, \eqref{restF},  we apply $L^{(i,j)}$ to \eqref{VP} and  use  the identity \eqref{LT}.
  We get for $m \geq 2$, the expression for  the source term $F_{i,j}$ 
  \begin{equation}
  \label{Fdef}
   F_{i,j}= - \left( F_{1}+ F_{2} + F_{3} + F_{4} \right)
   \end{equation}
   where 
  \begin{align}
  \label{F1def}
 F_{1}&= \sum_{k=2}^{m-1} L_{i_{1}, j_{1} } \cdots L_{i_{m-k}, j_{m-k} } \left(( \partial^2_{x_{i_{m-k+1} }, x_{j_{m-k+1} } } E )\cdot \nabla_{v} L_{i_{m-k+2},j_{m-k+2} } \cdots L_{i_{m}, j_{m} } f \right), \\
   \label{F2def}
   F_{2}&=  L_{i_{1}, j_{1}}\cdots L_{i_{m-1},  j_{m-1}} \left( \partial^2_{x_{i_{m}}, x_{j_{m}}} E \cdot \nabla_{v} f\right) - \partial_{x}^{\alpha(i,j)} E \cdot \nabla_{v} f, \\
        \label{F3def}
       F_{3}&=  \sum_{k=2}^{m-1} L_{i_{1}, j_{1} } \cdots L_{i_{m-k},  j_{m-k}} \left(\sum_{k',l'} \varphi^{i_{m-k+1}, j_{m-k+1} }_{k',l'} L_{k',l'} L_{i_{m-k+2},j_{m-k+2} } \cdots L_{i_{m}, j_{m} } f \right)\\
      & \qquad \qquad   \qquad \qquad -  \sum_{k=2}^{m-1} \sum_{k',l'} \varphi^{i_{m-k+1}, j_{m-k+1} }_{k',l'}  L^{(i_{m-k+1 \to k'}, j_{m-k+1 \to l'})} f 
         \label{F4def}, \nonumber \\
      F_{4}&=  L_{i_{1}, j_{1}}\cdots L_{i_{m-1},  j_{m-1}}\left( \sum_{k,l} \varphi^{i_m,j_m}_{k,l} L_{k,l} f \right) - \sum_{k,l} \varphi^{i_m,j_m}_{k,l} L^{(i_{m \to k}, j_{m \to l})} .
   \end{align}

   \bigskip
   
\noindent   {\bf Estimate of ${\bf F_{1}}$.}
We shall first study the estimate for $F_{1}$. We have to estimate terms under the form
\begin{equation}
\label{termsF1}
 F_{1, k}= L^{m-k} G_{k}, \quad G_{k}= \partial^2 E \cdot \nabla_{v} L^{k-1}
 \end{equation}
 where we use the notation $L^{n}$ for the composition of $n$ $L_{ij}$ operators (the exact combination of the operators involved in the composition does not matter).
  Note that as in \eqref{Lmexp}, we can  develop  $L^n$ under the form
  \begin{equation}
  \label{Lpexp}
  L^n= \partial_{x}^{\alpha_{n}} +  \sum_{s=0}^{2n-2} \sum_{e, \, \alpha, \, k_{0}  \cdots k_{s}}  P_{s, e,\alpha}^{k_{0}}(U) P_{s,e, \alpha}^{k_{1}}(\partial U) \cdots P_{s,e,\alpha}^{k_{s}}(\partial^{s} U) \partial_{v}^e \partial^{\alpha} 
    \end{equation}
     where   $ (P_{s,e, \alpha}^{k_{i}}(X))_{0 \leq i \leq s}$ are  polynomials of degree smaller than   $k_{i}$  and $\alpha_{n}$ has length $2n$ and  the sum is  taken on indices such that
      \begin{equation}
      \label{indices2}
       |e|= 1, \, | \alpha |= 2n - 1-s, \, k_{0} + k_{1}+ \cdots k_{s} \leq n , \, k_{0} \geq 1, \, k_{1} + 2 k_{2}+ \cdots s k_{s}= s.
      \end{equation}

       Let us first establish a general useful estimate.
        We set for  any fonction $G(x,v)$
         $$ J_{p}(G)(x,v) =  \sum_{s, \, \beta,  K } J_{p, s, \beta, K}(G)$$
          with $K= (k_{0}, \cdots, k_{s})$ and  
        $$  J_{p, s, \beta, K} (G) (x,v)=  P^{k_{0}}_{s, \beta}(U) P^{k_{1}}_{s, \beta}(\partial U) \cdots P^{k_{s}}_{s, \beta}(\partial^{s} U)  \partial^{\beta}  G$$ 
        with   $ (P^{k_{i}}_{s, \beta}(X))_{0 \leq i \leq s}$   polynomials of degree smaller than   $k_{i}$ and the sum is taken over indices such that 
\begin{equation}
\label{indices1}  \, | \beta |= p -s, \, k_{0} + k_{1}+ \cdots k_{s} \leq p/2 , \, k_{1} + 2 k_{2}+ \cdots s k_{s}= s, \, 0 \leq s \leq p-2.
\end{equation}
 \begin{lemma}
 \label{JseK}
  For $2m-1 \geq p  \geq 1$,  $ 2m >d + 3$, $2r>d$ and $s, \, p, \, K$ satisfying  \eqref{indices1},  we have the estimate
  \begin{equation}
  \label{estJseK}
  \|  J_{p} (G) \|_{\Hc^0_{r}} \leq \Lambda(T, R) \Bigl(  \| G \|_{\Hc^{p}_{r}} +  \sum_{ \begin{array}{ll} {\scriptstyle l \geq 2m - {d \over 2} - 2,} \\ {\scriptstyle  l+ | \alpha | \leq p, \, | \alpha | \geq 2}  \end{array}} \| \partial^l U   \partial^\alpha G \|_{\Hc^0_{r}} \Bigr).
  \end{equation}
 \end{lemma}
 \begin{proof}[Proof of Lemma~\ref{JseK}]
 
  For the terms in the sum such that $s  < 2 m - { d \over 2} - 2$, we can use \eqref{energy-constraintbis} to obtain that
 $$ \|  J_{p,s, \beta, K} (G) \|_{\Hc^0_{r}} \leq \Lambda(T, R) \| G \|_{\Hc^{p}_{r}}.$$
  When $s \geq  2 m - { d \over 2} - 2$, we first consider the terms for which in the sequence $(k_{1}, \cdots k_{s})$ the biggest index $l$ for which 
  $ k_{l} \neq 0$ is such that $l  < 2m-{d \over 2} - 2$. Then again thanks to  \eqref{energy-constraintbis}, we obtain that
  $$ \|  J_{p,s, \beta, K} (G) \|_{\Hc^0_{r}} \leq \Lambda(T, R) \| G \|_{\Hc^{p}_{r}}.$$
   When $l \geq 2m-{d\over 2} - 2$, we first observe that we necessarily have $k_{l}= 1$. Indeed if $k_{l} \geq 2$,  because of  \eqref{indices1}, we must have $ l \leq {s \over 2 }$. This is possible only if $  2m-{d\over 2} - 2 \leq   {p - 2 \over 2}  \leq  {2m-3 \over 2}$ that is to say $m \leq  {d \over 2} + 1$  and hence this is impossible. Consequently $k_{l} = 1$.
    Moreover we note that for the other indices $\tilde l$ for which  $k_{\tilde l} \neq 0$, because of \eqref{indices1}, we must have $ \tilde l k_{ \tilde l} \leq s -l k_l$, so that
    $$ \tilde l  \leq s- l \leq s - 2m + {d \over 2} + 2
     \leq {d \over 2} - 1$$ 
     and we observe that ${d \over 2} - 1 < 2m - {d \over 2} - 2$. Consequently, by another use of \eqref{energy-constraintbis}, we obtain that 
     $$  \|  J_{p,s, \beta, K} (G) \|_{\Hc^0_{r}} \leq \Lambda(T, R)   \sum_{ \begin{array}{ll} \scriptstyle{ l\geq  2m - {d \over 2} - 2,} \\ \scriptstyle{ l+ | \alpha | \leq p, \, | \alpha | \geq 2}  \end{array}} \| \partial^l U   \partial^\alpha G \|_{\Hc^0_{r}}.$$
      The fact that $| \alpha | \geq 2$ comes from \eqref{indices1}.
      This ends the proof of Lemma \ref{JseK}.
 \end{proof}
 
  We shall now estimate $F_{1, k}$. 
  Let us start with the case where  $k \geq m/2$.
 Looking at the expansion of $L^{m-k}$ given by \eqref{Lpexp}, we have to estimate terms under the form
 $ J_{2p}(G_{k})$ for $ 2p \leq  2(m-k) \leq m$.  We can thus use Lemma \ref{JseK}. Moreover, we observe that in the right hand side of \eqref{estJseK}, we  have that 
  $l \leq  2(m-k) - 2 \leq m - 2$, consequently, by assumption on $m$, we have $ l <  2m - {d \over 2 } - 2$ and hence we can estimate
   $\| \partial^l U \|_{L^\infty}$ by using \eqref{energy-constraintbis}. This yields
    $$ \| F_{1, k} \| \leq \Lambda(T,R) \| G_{k}\|_{\Hc^{2(m-k)}_{r}}, \quad k \geq m/2.$$
 Next,  we  use  \eqref{com3} with $s=2(m-k)$ and $s_0 =2m-3 \, (>d)$, and the definition of $G_{k}$ in \eqref{termsF1} to estimate the above right hand side.
  Since $ d+2 < 2 m-1$ by assumption on $m$  and $ 2(m-k) + 2 \leq 2m-1$ (since $k \geq 2$), we obtain
 \begin{align}
 \label{F1kinterm}
  \|F_{1,k}\|_{\Hc^0_{r}} & \leq   \Lambda(T,R) \big(  \| E \|_{H^{2m-1}} \|\nabla_v  L^{k-1}f \|_{\Hc^{2(m-k)}_{r } } +  \| E \|_{H^{2(m-k) + 2}} \|\nabla_v  L^{k-1}f \|_{\Hc^{2(m-k)}_{r } }  \big) \\
\nonumber &   \leq    \Lambda(T,R) \| E \|_{H^{2m-1}} \|\nabla_v  L^{k-1}f \|_{\Hc^{2(m-k)}_{r}}.
 \end{align}
  By using again \eqref{estE1}, this yields
  $$   \|F_{1,k}\|_{L^2 ([0, T], \Hc^0_{r})}  \leq   \Lambda(T,R)  \|\nabla_v  L^{k-1}f \|_{L^\infty([0, T], \Hc^{2(m-k)}_{r})}.$$
   To estimate the above right hand side, we need to estimate
   $ \partial_{x,v}^\gamma  L^{k-1}f$ with  $| \gamma | \leq 2m - 2 k + 1$.
    By  taking derivatives using  the expression  \eqref{Lpexp},  we see that we have  to estimate terms under the form
     $ J_{p}(f)$  with $p \leq 2m-1$. Using Lemma~\ref{JseK}, we thus obtain that
   $$   \|F_{1,k}\|_{L^2([0, T], \Hc^0_{r})}  \leq   \Lambda(T,R) \Bigl(  \| f \|_{L^\infty ([0, T], \Hc^{2 m-1}_{r})} + 
   \sum_{ \begin{array}{ll} \scriptstyle{ l \geq  2m - {d \over 2} - 2,} \\ \scriptstyle{ l+ | \alpha | \leq 2m-1, \, | \alpha | \geq 2}  \end{array}} \| \partial^l U   \partial^\alpha f \|_{L^\infty([0, T], \Hc^0_{r})}
     \Bigr).$$
   To estimate the right hand side, 
         since $| \alpha | \geq 2$ and $|\alpha| - 2 + l \leq 2m-3$, we can use \eqref{comdual}, to  obtain that
         $$  \| \partial^l U  ( 1 + |v|^2)^{r\over 2} \partial^\alpha f\|_{L^2_{x,v}} \lesssim \|U\|_{\Hc^{2m-3}_{-r}} \|  ( 1 + |v|^2)^{r}\partial^2 f \|_{L^\infty} +  
          \|U \|_{L^\infty} \|f \|_{\Hc^{2m-1}_{2r}}.$$
          By using again \eqref{energy-constraint}, \eqref{energy-constraintbis} and the Sobolev embedding, we finally obtain that
          \begin{equation}
          \label{F1k1} \|F_{1, k}\|_{L^2([0, T], \Hc^0_{r})}  \leq   \Lambda(T,R) \|f\|_{L^\infty([0, T], \Hc^{2m-1}_{2r})} \leq \Lambda(T, R), \quad k \geq m/2.
          \end{equation}
          
          It remains to handle the case $k \leq m/2$. Again, by using \eqref{termsF1} and the expansion \eqref{Lpexp}, we first have to estimate
           terms under the form $J_{2(m-k)}(G_{k})$. By using again Lemma \ref{JseK},  we first obtain
        $$  \|F_{1,k}\|_{L^2([0, T], \Hc^0_{r})}  \leq   \Lambda(T,R) \Bigl(  \| G_{k} \|_{L^2([0, T], \Hc^{2 (m - k)}_{r})} + 
   \sum_{ \begin{array}{ll} \scriptstyle{ l \geq 2m - {d \over 2} - 2,} \\ \scriptstyle{ l+ | \alpha | \leq 2(m - k), \, | \alpha | \geq 2}  \end{array}} \| \partial^l U   \partial^\alpha G_{k} \|_{L^2([0, T], \Hc^0_{r})}
     \Bigr).$$
    
     By using the expression in  \eqref{termsF1} for $G_{k}$, we have to estimate terms under the form 
     $$ \|   \partial^l U  \partial^\beta \partial^2 E\,  \partial^\gamma \nabla_{v} L^{k-1} f \|_{\Hc^0_{r}}$$
      with $  l \geq  2m - {d \over 2} - 2$ and $    l+ | \beta  | + | \gamma |  \leq 2(m - k).$
       Note that this implies that  $| \beta | \leq    2(m-k)-l \leq {d\over 2} + 2  - 2k \leq {d \over 2} - 2$
        since we have $k \geq 2$ (see \eqref{F1def}). In particular this yields $|\beta | + 2  + {d \over 2} < 2m-2$ and thus 
       by using again the  Sobolev embedding (in $x$) and \eqref{estE1} we obtain that  
       \begin{align*}
          \|   \partial^l U  \partial^\beta \partial^2 E \partial^\gamma \nabla_{v} L^{k-1} f \|_{\Hc^0_{r}} 
                    &\lesssim \| \rho \|_{H^{2m-1}} \| \partial^l U   \partial^\gamma \nabla_{v} L^{k-1} f \|_{\Hc^0_{r}} \\
                    &\lesssim \| f \|_{\Hc^{2m-1}_{r} }\| \partial^l U   \partial^\gamma \nabla_{v} L^{k-1} f \|_{\Hc^0_{r}} \\
          &\leq \Lambda(T, R )   \| \partial^l U   \partial^\gamma \nabla_{v} L^{k-1} f \|_{\Hc^0_{r}}.
           \end{align*}
        Thus it remains to estimate  $ \| \partial^l U   \partial^\gamma \nabla_{v} L^{k-1} f \|_{\Hc^0_{r}} $ for 
         $  l \geq  2m - {d \over 2} - 2$ and $    l+ | \gamma |  \leq 2(m - k).$
         By using again \eqref{Lpexp}, we can expand  $ \partial^\gamma \nabla_{v} L^{k-1} f$ as an expression under the form $J_{2k+ | \gamma |-1} (f)$.
          Since we have that
          $  2 k + | \gamma | - 1 \leq  1 + {d \over 2 } < 2m- {d \over 2} - 2$, we can use \eqref{energy-constraintbis} again to estimate all  the terms involving $U$
           and its derivatives in 
           in $L^\infty$, this yields
           $$   \| \partial^l U   \partial^\gamma \nabla_{v} L^{k-1} f \|_{\Hc^0_{r}} \leq \Lambda( T, R) \sum_{\tilde \gamma}  \| \partial^l U  \partial^{\tilde \gamma} f \|_{\Hc^0_{r}}$$
            with  $ |\tilde \gamma | \leq  |\gamma |+ 2 k - 1$ and thus $ l +  |\tilde \gamma | \leq 2 m - 1$ and $|\tilde \gamma | \geq 2$ (since $k \geq 2$).
            Consequently, by using again \eqref{comdual}, we obtain that 
            $$ \| \partial^l U   \partial^\gamma \nabla_{v} L^{k-1} f \|_{\Hc^0_{r}} \leq \Lambda( T, R)  \left( \| U \|_{L^\infty}
             \| f \|_{\Hc^{2m-1}_{2r}} +  \| (1 +  |v|^2)^{r} \partial^2 f \|_{L^\infty_{x,v}} \|U \|_{\Hc^{2m-3}_{-r}}\right)$$
              and we conclude finally by using \eqref{energy-constraintbis},  \eqref{energy-constraint} and  the Sobolev embedding that
           \begin{equation}
           \label{F1k2} \|F_{1, k} \|_{L^2([0, T], \Hc^0_{r})} \leq \Lambda(T, R), \quad k \leq m/2
           \end{equation}
           (actually for this case  we even have a slightly better $L^\infty$ in time estimate).
            Looking at \eqref{F1k1}, \eqref{F1k2}, we have thus proven that
            \begin{equation}
            \label{estimationFk1}
              \|F_1 \|_{L^2([0, T], \Hc^0_{r})} \leq \Lambda(T, R).
              \end{equation}
       \noindent    {\bf Estimate of ${\bf F_2}$.}
           We shall now turn to the study of $F_{2}$.
            By using \eqref{Lpexp} again, we can expand under the form
            $$ F_{2}= \mathcal{C} + \sum_{|e|= 1} J_{2m-3}(\partial_{v}^e( \partial^2 E \cdot \nabla_{v} f))$$
             with 
            $$ \mathcal{C}= \partial_{x}^{\alpha(\tilde \imath, \tilde \jmath)} \left( \partial^2_{x_{i_{m}}, x_{j_{m}}} E \cdot \nabla_{v} f \right) -  \partial_{x}^{ \alpha (i,j)} E \cdot \nabla_{v} f
            = [ \partial_{x}^{\alpha(\tilde \imath, \tilde \jmath)}, \nabla_{v} f] \cdot\partial^2 E.$$
            where $ \alpha(\tilde \imath, \tilde \jmath)$ has length $2m-2$.
             Note that this time, we have really used that in the expansion \eqref{Lpexp}, the terms in the sum always involve at least one $v$ derivative.
   By using Lemma \ref{JseK}, we get that
   \begin{multline*} \|J_{2m-3}( \partial^e_{v}( \partial^2 E \cdot \nabla_{v} f))\|_{\Hc^0_{r}} \leq \Lambda(T,R)\Big(
     \| \partial^e_{v}( \partial^2 E \cdot \nabla_{v} f )\|_{\Hc^{2m-3}_{r}} 
   \\  +   \sum_{ \begin{array}{ll} \scriptstyle{ l \geq 2m - {d \over 2} - 2,} \\ \scriptstyle{ l+ | \alpha | \leq 2m-3, \, | \alpha | \geq 2}  \end{array}} \| \partial^l U   \partial^\alpha \partial^e_{v}(
    \partial^2 E \cdot \nabla_{v} f)\|_{\Hc^0_{r}}   \Bigr).  
    \end{multline*}  
    To estimate the first term, we can use \eqref{com3}  and   \eqref{estE1} to obtain that 
    $$  \| \partial_{v}( \partial^2 E \cdot \nabla_{v} f )\|_{\Hc^{2m-3}_{r}} \leq  \|  \partial^2 E \cdot \partial^2_{v} f \|_{\Hc^{2m-3}_{r}} 
     \lesssim \| E \|_{H^{2m-1}} \| f \|_{ \Hc^{2m-1}_{r}}$$
     and hence we can take the $L^2$ norm in time and use   \eqref{estE1} to obtain that
     \begin{equation}
     \label{provisoirebis}
        \| \partial_{v}( \partial^2 E \cdot \nabla_{v} f )\|_{L^2([0, T], \Hc^{2m-3}_{r})} \leq  \Lambda(T, R).
        \end{equation}
      To estimate the terms in the sum, we use again \eqref{comdual}, \eqref{energy-constraintbis} and the Sobolev embedding to write
      \begin{align*}  \| \partial^l U   \partial^\alpha \partial_{v}(
    \partial^2 E \cdot \nabla_{v} f)\|_{\Hc^0_{r}} 
    &  \lesssim  \|U\|_{L^\infty} \| \partial_{v} ( \partial^2 E \cdot \nabla_{v} f )\|_{\Hc^{2m-3}_{2r}} 
      + \| ( 1 + |v|^2)^{r} \partial^2 E \partial_{v}^2 f \|_{L^\infty_{x,v}} \|U \|_{\Hc^{2m-3}_{-r}} \\
    &    \leq \Lambda(T,R)  \| \partial_{v} ( \partial^2 E \cdot \nabla_{v} f )\|_{\Hc^{2m-3}_{2r}}  +   \Lambda(T,R) \| \partial^2 E \|_{L^\infty}.
    \end{align*}
     Therefore, we get from \eqref{estE1} and the Sobolev embedding  in $ x$ a bound by
     $$(1+ \| \partial_{v} ( \partial^2 E \cdot \nabla_{v} f )\|_{L^2([0, T], \Hc^{2m-3}_{2r})}
    ) \Lambda(T, R).$$ 
             By using \eqref{provisoirebis} (which is still true with $r$ changed into $2r$), we finally obtain that
       \begin{equation}
       \label{estimF21}  \|J_{2m-3}( \partial_{v} (\partial^2 E \cdot \nabla_{v} f))\|_{L^2([0, T], \Hc^0_{r})} \leq \Lambda(T,R).
       \end{equation}
        It remains to estimate $\mathcal{C}$. By expanding the commutator,  we have to estimate  terms of the form $\| ( 1 + |v|^2) ^{r\over 2} \partial_{x}^\beta \nabla_{v}f \cdot \partial_{x}^\gamma \partial^2 E \|_{L^2_{x,v}}$
          with $ |\beta| + | \gamma | \leq 2m-2, \, \beta \neq0$.
           If $ | \gamma | + 2 + {d \over 2} < 2m-1$, by using Sobolev embedding in $x$  and \eqref{estE1}, we obtain
           $$ \| ( 1 + |v|^2) ^{r\over 2} \partial_{x}^\beta \nabla_{v}f \cdot \partial_{x}^\gamma \partial^2 E \|_{L^2([0, T], L^2_{x,v})} 
            \leq \Lambda(T, R) \| f \|_{L^\infty([0, T], \Hc^{2m-1}_{r})}  \leq \Lambda(T,R).$$
              Otherwise, since $ | \beta | + 1 + {d \over 2} \leq   2 + d < 2m- 1$ and $ | \gamma | + 2 \leq 2m-1$, we  get
               that 
           \begin{align*}\| ( 1 + |v|^2) ^{r\over 2} \partial_{x}^\beta \nabla_{v}f \cdot \partial_{x}^\gamma \partial^2 E \|_{L^2([0, T],L^2_{x,v})}  
            &   \leq  \sup_{[0, T]}\sup_{x} \Big( \int_{\mathbb{R}^d} | ( 1 + |v|^2) ^{r\over 2} \partial_{x}^\beta \nabla_{v}f |^2 \, dv \Big)^{1 \over 2}
               \|E \|_{L^2([0, T], H^{2m-1}_{x})} \\
               & \leq \Lambda(T, R).
               \end{align*}
            We have thus obtained that
           $$ \| \mathcal{C} \|_{L^2([0, T], \Hc^0_{r})} \leq \Lambda(T,R).$$
            By collecting the last estimate and \eqref{estimF21}, we actually get that
           \begin{equation}
           \label{estimationF2} \|F_{2}\|_{L^2([0, T], \Hc^0_{r})} \leq \Lambda(T,R).
           \end{equation}
          This ends the proof of 
          \eqref{restF}.
          
  %   argue similarly by using \eqref{LT} and again \eqref{energy-constraint}.
    
    \noindent    {\bf Estimate of ${\bf F_{3}}$ and ${\bf F_{4}}$.} By using similar arguments,  we  also obtain that 
         \begin{equation}
         \|F_{3}\|_{L^2([0, T], \Hc^0_{r})} +   \|F_{4}\|_{L^2([0, T],\Hc^0_{r})} \leq \Lambda(T,R).
           \end{equation}
    \end{proof}

   \subsection{Straightening the transport vector field}
   \label{secburgers}
    We shall now study the equation \eqref{eq-fij} and try to get an estimate of $ \int_{\mathbb{R}^d} f_{i,j}\, dv$ which, in view
     of \eqref{reduit}, can be used to estimate $\partial_{x}^{2m} \rho$.
   The next step consists in using a change of variables in order to straighten the  transport vector field  and more precisely
   to come down from the full transport operator $\mathcal{T}$ to a twisted free transport operator
   of the form
   $$
   \partial_t + \Phi(t,x,v) \cdot \nabla_v.
   $$
   This is the purpose of the following lemma.
       \begin{lemma}%[Change of variables -- straightening the flow]
       \label{lemburg}
       Let $f_{i,j}$ be a smooth function satisfying the equation \eqref{eq-fij}.
    Consider $\Phi(t,x,v)$ a smooth solution to the Burgers equation
    \begin{equation}
    \label{eq-Burgers}
    \pa_t \Phi + \Phi \cdot \nabla_x \Phi = E. %, \quad \Phi\vert_{t=0}=v.
    \end{equation}
    such that the Jacobian matrix $(\na_v \Phi)$ is invertible.
    Then defining
        \begin{equation}
        g_{i,j}(t,x,v) := f_{i,j}(t,x, \Phi),
        \end{equation}
        we obtain that $g_{i,j}$ satisfies the equation
            \begin{equation} 
            \label{eq-straight}
            \pa_t g_{i,j} + \Phi \cdot \na_x g_{i,j} +  \pa_x^{\alpha(i,j)} E \cdot (\na_v f )(t,x, \Phi) = F(t,x,\Phi).
   \end{equation}
   
    \end{lemma}

       \begin{proof}[Proof of Lemma~\ref{lemburg}]
    This follows from a direct computation. We can check that
   \begin{align*}
                \pa_t g_{i,j} + \Phi \cdot \na_x g_{i,j} &+  \pa_x^{\alpha(i,j)}  E \cdot (\na_v f )(t,x, \Phi) \\
                &= F(t,x,\Phi) 
                +  \,^t(\na_v \Phi)^{-1} \na_v g_{i,j} \cdot \left(  \pa_t \Phi + \Phi \cdot \nabla_x \Phi - E \right).
        \end{align*}
              This yields \eqref{eq-straight}, because of \eqref{eq-Burgers}.
    \end{proof}
   
   We shall now establish Sobolev estimates for the solutions of the Burgers equation \eqref{eq-Burgers}. 
   Choosing the initial condition $\Phi|_{t=0} =v$, we will  obtain a control on the deviation from $v$ in Sobolev norms and in  particular, observe that $\Phi(t,x,v)$ remains close to $v$ for small enough times.
 
         \begin{lemma}%[Burgers' equation]
         \label{lemburgers}
          Assuming that $2m>3+d$,
there exists  $ \tilde{T}_0= \tilde T_{0}(R)>0$ independent of $\eps$  such that for every $T < \min (T_{0},  \tilde T_{0}, T^\eps)$, there is a unique smooth solution on $[0,T]$ of \eqref{eq-Burgers} together with the initial condition $\Phi|_{t=0} =v$. 

Moreover, we have the following uniform estimates for every   $T  <  \min (T_{0},  \tilde T_{0}, T^\eps)$
 \begin{equation}
 \label{estim-Burgers-infini}
  \sup_{[0,T]} \| \Phi  - v\|_{W^{k,\infty}_{x,v}}  +  \sup_{[0, T]} \left\| { 1 \over (1 + |v|^2)^{1 \over 2}} \partial_{t} \Phi \right\|_{W^{k-1, \infty}_{x,v}} \leq  T^{1 \over 2} \Lambda(T,R), \quad k < 2m-d/2-1.
\end{equation}
Furthermore, we also have that  for every  $| \alpha | \leq 2m-1$ and $|\beta| \leq 2m-2,$
 \begin{equation}
 \label{estim-Burgers}
 \begin{aligned}
  \sup_{[0,T]} \sup_{v}   &\| \partial_{x, v}^\alpha ( \Phi  - v)\|_{ L^2_{x}} + 
   \sup_{[0,T]} \sup_{v}  \left\| { 1 \over (1 + |v|^2)^{1 \over 2}} \partial_{x, v}^\beta \partial_{t} \Phi  \right\|_{ L^2_{x}}   \leq T^{1 \over 2} 
    \Lambda(T, R).
  \end{aligned}
  \end{equation}
    \end{lemma}
   
        \begin{proof}[Proof of Lemma~\ref{lemburgers}]
     Let us set $\phi= \Phi -v$. We observe that $ \phi$ solves
     $$  \pa_t  \phi + (v+ \phi) \cdot \nabla_{x}  \phi  = E$$
      with zero initial data. For any $\a \in \N^d$, applying $\pa^\alpha$ to the equation, we get that  $\pa^\a \phi= \partial_{x,v}^\alpha \phi$ satisfies
\begin{align}
\label{eq-comm}
\pa_t \pa^\a \phi +  v \cdot \na_x \pa^\a \phi +  \phi \cdot \na_{x} \pa^\a \phi = \pa^\a E  - \Big[ \sum_{\b + \g \leq \a \atop \g \neq \a }  c_{\b,\g} \pa^\b \phi \cdot \na_{x} \pa^\g \phi \Big]
 - [ \partial^\alpha, v] \cdot \nabla_{x} \phi.
\end{align}

       Let us set $M_{k}(T)= \sup_{[0, T]} \| \phi \|_{W^{k, \infty}_{x,v}}$.
      Using $L^\infty$ estimates for the  transport operator, we obtain from \eqref{eq-comm}  that
      $$  M_{k}(T) \lesssim  T( 1  + M_{k}(T)) M_{k}(T)  + \int_{0}^T \| E \|_{W^{k, \infty}} \, dt.$$
      Since $k + {d \over 2} < 2m-1$, we get by Sobolev embedding in the $x$ variable and \eqref{estE1} that
      $$ M_{k}(T) \lesssim  T( 1  + M_{k}(T)) M_{k}(T)  + T^{1 \over 2} R.$$
       Consequently, for $\tilde T_{0}$ sufficiently small depending only on $R$, we obtain that
       $$  M_{k}(T)  \leq  T^{1 \over 2} \Lambda(T, R).$$
        This proves  the first part of \eqref{estim-Burgers-infini}.
         For the estimate on the time derivative, it suffices to use the equation \eqref{eq-Burgers} and  the estimate we have  just obtained.
         
         It remains to prove \eqref{estim-Burgers}.
%     We introduce the backward characteristics
%         \begin{equation*}
%   %\label{characteristic}
%     \partial_{s} Y(s,x,v)=  \Phi(s,Y(s,x,v), v), \quad Y(s,x,v)= x.
%    \end{equation*}
%     Note that we have the bound for all $\a \neq 0$
%     $$
%     \| \pa_x^\a Y(s,x,v \|_{L^\infty} \lesssim F(\sqrt{T}, \| \rho\|_{L^2((0,T), H^{2m})}) \sqrt{T} \| \rho\|_{L^2((0,T), H^{2m})},
%     $$
%     for some smooth function $F$.
%     
%     We obtain the usual Duhamel formula for Burgers
%     $$
%     \Phi(t,x) = v + \int_0^t E(s,Y(s,x,v)) \, ds.
%     $$
%The estimates \eqref{estim-Burgers-infini} therefore easily follow.
%
We proceed by energy estimates.
Using \eqref{eq-comm} for $|\a| \leq 2m-1$,  multiplying by  $\pa^\alpha \phi$ and integrating in $x$, we obtain from a standard energy estimate ($v$ being only a parameter for the moment)
\begin{align*}
\frac{d}{dt}  \| \partial^\alpha \phi \|_{L^2_{x}}  \lesssim  ( 1 + \| \partial_{x} \phi \|_{L^\infty_{x}} ) \sum_{| \alpha | \leq 2m-1} \| \partial^\alpha \phi \|_{L^2_{x}} + \| E\|_{H^{2m-1}} + \| \mathcal C\|_{L^2_{x}}   \end{align*}
where $\mathcal C$ is the commutator term
$$ \mathcal{C}=  \Big[ \sum_{\b + \g \leq \a \atop \g \neq \a }  c_{\b,\g} \pa^\b \phi \cdot \na_{x} \pa^\g \phi \Big].$$
 Let us set
 $$ Q_{2m-1}(T, \phi) = \sup_{[0, T]} \sup_{v} \Big( \sum_{|\alpha| \leq 2m-1} \| \partial^\alpha \phi \|_{L^2_{x}} \Big).$$
We can then integrate in time and take the sup in  time and $v$ to obtain that
$$ Q_{2m-1}(T, \phi) \lesssim  Q_{2m-1}(T, \phi)  \left( T+  \int_{0}^T \| \partial_{x} \phi \|_{L^\infty_{x,v}} dt\right)+  \int_{0}^T\|\mathcal C\|_{L^\infty_{v} L^2_{x}} \, dt +  T^{1 \over 2} \| \rho\|_{L^2([0, T],  H^{2m})}$$
 where the last term comes from another use of \eqref{estE1}.
  From \eqref{estim-Burgers-infini}, we already have that 
  $$  \| \partial_{x} \phi \|_{L^\infty_{x,v}} \leq T^{1 \over 2} \Lambda(T, R),$$
  thus it only remains to estimate
   the commutator term $\mathcal{C}$.   For the terms in the sum such that $|\beta|< 2m- {d \over 2} - 1 $, we can use  \eqref{estim-Burgers-infini}  and the fact
    that $|\gamma |<| \alpha|$ to obtain that 
    $$ \| \pa^\b \phi \cdot \na \pa^\g \phi \|_{L^\infty_{v} L^2_{x}} \lesssim  T^{1 \over 2} \Lambda (T,R) Q_{2m-1}(T, \phi).$$
     In a similar way, when $|\beta| \geq 2m- {d \over 2} - 1$, we observe that $ 1 + | \gamma | \leq  {d \over 2} < 2m - {d \over 2} - 1$
      consequently, by using again \eqref{estim-Burgers-infini}, we also  obtain that
      $$  \| \pa^\b \phi \cdot \na \pa^\g \phi \|_{L^\infty_{v} L^2_{x}} \lesssim  T^{1 \over 2} \Lambda (T,R) Q_{2m-1}(T, \phi).$$
       This yields
       $$  \| \mathcal{C}\|_{L^\infty_{v} L^2_{x}} \lesssim  T^{1 \over 2} \Lambda (T,R) Q_{2m-1}(T, \phi)$$
        and hence that
        $$  Q_{2m-1}(T, \phi) \lesssim ( T  +  T^{3 \over 2 } \Lambda(T,R)) Q_{2m-1}(T, \phi)  +  T^{1 \over 2}  R.$$
         By taking $\tilde T_{0}$ small enough (depending on $R$ only), we finally obtain that
         $$  Q_{2m-1}(T, \phi) \lesssim  T^{1 \over 2}  \Lambda(T, R) $$
          and hence  the first part of  \eqref{estim-Burgers} is proven. Again the estimate on the time derivative
           follows by using the equation \eqref{eq-Burgers} and the previous estimates.

  \end{proof}
    
    By a change of variable,  we can then easily  relate the average in $v$ of $f_{ij}$ to a weighted average of $g_{i,j}$ :
    \begin{lemma}
    \label{lemequiv}
    We have
         \begin{equation}
         \label{formule-equiv}
 \int_{\R^d} g_{i,j}  J\, dv = \pa_x^{\a(i,j)} \rho + \mathcal R,
 \end{equation}  
 where $J(t,x,v) := |\det \na_v \Phi(t,x,v)|$ 
 and  $\mathcal{R}$  still satisfies the estimate  \eqref{rest1}.
    
\end{lemma}
    
   \section{Proof of Theorem~\ref{theomain}: estimate of $\| \rho \|_{L^2([0, t], H^{2m})}$ by  using the Penrose stability condition}
  \label{secH2m}
  
    Following the reduction of the previous section (from which we keep the same notations), we shall now study the system of equations
      \begin{equation} 
      \label{galphabeta}
            \pa_t g_{i,j} + \Phi \cdot \na_x g_{i,j} +  \pa_x^{\alpha(i,j)}  E \cdot (\na_v f )(t,x, \Phi) +  \mathcal{M}_{i,j} (t,x,\Phi) \mathcal{G}= \mathcal{S}_{i,j},
   \end{equation}
   with $\mathcal{G} = (g_{i,j})$ and $ \mathcal{S}_{i,j} (t,x,v)= F_{i,j}(t,x, \Phi(t,x,v))$. Note that the equations of this system are coupled only through the zero order terms $ \mathcal{M}_{i,j} (t,x,\Phi) \mathcal{G}$.

    Let us introduce the characteristic flow $X(t,s,x,v)$, $ 0 \leq s, t \leq T$
    \begin{equation}
    \label{characteristic}
     \partial_{t} X(t,s,x,v)=  \Phi(t,X(t,s,x,v), v), \quad X(s,s,x,v)= x.
    \end{equation}
    Note that the velocity variable is only a parameter in this ODE.
    
    We start with estimating the deviation from free transport (that corresponds to the case $\Phi=v$).
       \begin{lemma}
   \label{redressement2}
   For every $t, \,s$, $0 \leq s \leq t \leq T$ and $T, \, m, \, r$ as in Lemma \ref{lemburgers}, we can write
    \begin{equation}
    \label{estimXtilde}
     X(t,s,x,v)=  x  + (t-s) \left( v +   \tilde X(t,s,x,v)  \right)
     \end{equation}
     with $\tilde X$ that satisfies the estimate
      \begin{equation}
      \label{LinftyXtilde}
 \begin{aligned}
   &  \sup_{t,s \in [0,T]} \Big( \|\pa_{x,v}^\alpha  \tilde{X}(t,s,x,v) \|_ {L^\infty_{x, v}} + \left\|{ 1 \over (1+ |v|^2)^{1 \over 2}} \pa_{x,v}^\beta \partial_{t}  \tilde{X}(t,s,x,v)
   \right\|_ {L^\infty_{x, v}} \Big) \leq T^{\frac{1}{2}}  \Lambda(T,R).
   \end{aligned}
   \end{equation}
    for every  $ |\alpha|  <2 m - d/2 - 1, \, |\beta| < 2 m - d/2 - 2$.
    
   Moreover, there exists $\hat T_{0}(R)>0$ sufficiently small such that for every  $T \leq  \min (T_{0},  \tilde T_{0}, \hat T_{0},  T^\eps)$, 
    we have that $x\mapsto x+ (t-s) \tilde X(t,s,x,v)$ is a diffeomorphism and that 
    \begin{equation}
    \label{L2Xtilde}
    \begin{aligned}
  &\sup_{t,s \in [0,T]} \sup_{v} \Big(  \| \partial_{x, v}^\alpha  \tilde{X}(t,s,x,v) \|_{ L^2_{x}}
   +\left\| { 1 \over (1+ |v|^2)^{1 \over 2}} \pa_{x,v}^\beta \partial_{t}  \tilde{X}(t,s,x,v) \right\|_{ L^2_{x}}
  \Big) \leq T^{1 \over 2} \Lambda(T, R)
  \end{aligned}
  \end{equation}
  for every $|\alpha| \leq 2m-1$, $|\beta| \leq 2m-2.$
In addition, there exists  $\Psi(t,s, x,  v)$ such that for $t, \, s\in [0, T] $ and $T\leq  \min (T_{0},  \tilde T_{0}, \hat T_{0},  T^\eps)$, we have
   $$ X(t,s, x , \Psi(t,s,x, v)) = x +  (t-s) v$$
    and the following estimates 
          \begin{equation}
 \begin{aligned}
     \label{estimPsi}
     &\sup_{t,s \in [0,T]}\Big( \| \pa^\alpha_{x,v} ( \Psi(t,s,x,v)  - v) \|_ {L^\infty_{x, v}}
      +\left\|{ 1 \over (1+ |v|^2)^{1 \over 2}} \pa_{x,v}^\beta \partial_{t} \Psi (t,s,x,v)  \right\|_ {L^\infty_{x, v}} \Big)  \leq T^{\frac{1}{2}} \Lambda(T,R),\\
      & \mbox{\hspace{8cm}} \quad   |\alpha|  <2 m - d/2 - 1, \, |\beta|  < 2m- {d\over2} - 2\\
  &\sup_{t,s \in [0,T]} \sup_{v} \Big(  \| \partial_{x, v}^\alpha ( \Psi(t,s,x,v)  - v)\|_{ L^2_{x}}
   +\left\| { 1 \over (1+ |v|^2)^{1 \over 2}} \pa_{x,v}^\beta \partial_{t}  \Psi(t,s,x,v) \right\|_{ L^2_{x}}
  \Big)  \leq T^{1 \over 2} \Lambda(T, R),\\
  & \mbox{\hspace{9cm}} |\alpha|\leq 2m-1, \, | \beta| \leq 2m-2.
  \end{aligned}
  \end{equation}

   \end{lemma}
   
   \begin{proof}[Proof of Lemma~\ref{redressement2}]
   Let us set $\phi = \Phi -v$ as in the proof of Lemma \ref{lemburgers} and 
    $Y(t,s,x,v)= X(t,s,x,v)  - x- (t-s)v$. We shall first estimate $Y$. Since we have 
        \begin{equation}
    \label{eqintY}
    Y(t,s,x,v)= \int_{s}^t \phi\left(\tau, x+ (\tau -s ) v + Y(\tau, s, x, v), v \right) \, d \tau,
    \end{equation}
    we deduce from the estimates of  Lemma \ref{lemburgers} that for $|\alpha | < 2m- {d \over 2} -1$, we have for $0 \leq s, \, t \leq T$,
    $$ \sup_{|\alpha|  <2m- {d \over 2} -1 } \| \partial^\alpha_{x,v} Y(t,s) \|_{L^\infty_{x,v}} \leq \left| \int_{s}^t
      T^{1 \over 2} \Lambda(T, R) \big( 1 +  \sup_{|\alpha|  <2m- {d \over 2} -1 } \| \partial^\alpha_{x,v} Y(\tau,s) \|_{L^\infty_{x,v}}\big)\, d\tau\right|.$$
       From the Gronwall inequality, this yields
      \begin{equation}
      \label{estY1}  \sup_{|\alpha|  <2m- {d \over 2} -1 } \| \partial^\alpha_{x,v}Y(t,s) \|_{L^\infty_{x,v}} \leq |t-s| T^{1 \over 2} \Lambda(T,R).
      \end{equation}
       Consequently, we can set $\tilde{X}(t,s,x,v)=Y(t,s,x,v)/(t-s)$  and deduce from the above estimate
        that $\tilde X$ verifies the first part of  \eqref{LinftyXtilde}.  
         To estimate the time derivative, we  go back to  \eqref{eqintY}. We use a Taylor expansion to write
        \begin{multline*}
          \phi\left(\tau, x+ (\tau -s )( v + \tilde X(\tau, s, x, v)), v) \right)  \\=  \phi(s,x,v) + (\tau-s) \left( \phi_{1}(\tau, s, x,v)+ \phi_{2}(\tau, s, x, v) \cdot (v+\tilde X(\tau, s, x, v))\right)
          \end{multline*}
          where
         \begin{equation}
          \label{defphi1phi2}
         \begin{aligned}
          \phi_{1}(\tau, s, x,v) & =  \int_{0}^1 \partial_t \phi((1- \sigma) s + \sigma \tau, x, v)) d \sigma, \\
          \,\phi_{2}(\tau, s, x, v) & = \int_{0}^1 D_{x} \phi(\tau, x + \sigma (\tau-s)(v+  \tilde X(\tau, s,x, v)), v) \, d\sigma.
          \end{aligned}
          \end{equation}
           By using \eqref{eqintY}, we thus obtain that
        $$
          \tilde X(t,s,x,v)= \phi(s,x,v) + { 1 \over  t-s}  \int_{s}^t (\tau-s)Y_{1}(\tau,s,x,v)\, d\tau
       $$
         with
         \begin{equation}
         \label{Y1def} Y_{1}(\tau,s,x,v)= \left( \phi_{1}(\tau, s, x,v)+ \phi_{2}(\tau, s, x, v) \cdot (v+\tilde X(\tau, s, x, v))\right)\, d \tau
         \end{equation}
         and thus that
       \begin{equation}
        \label{eqXtildebis}
       \partial_{t} \tilde X(t,s,x,v)  = - { 1 \over (t-s)^2}   \int_{s}^t  (\tau-s)Y_{1}(\tau,s,x,v)\, d\tau+
       Y_{1}(t,s,x,v).
       \end{equation}
         By using  \eqref{estY1}, 
         \eqref{estim-Burgers-infini} with the same  arguments as above, we  get that for $| \beta | < 2m- {d \over 2} -2$, the following estimate holds:
         $$  \left\|{ 1 \over (1+ |v|^2)^{1 \over 2}} \pa_{x,v}^\beta   Y_{1}(\tau,s,x,v)
   \right\|_ {L^\infty_{x, v}}  \leq  T^{1 \over 2} \Lambda(T, R).$$
   This yields  \eqref{LinftyXtilde}.

         We now turn to the proof of the estimate \eqref{L2Xtilde}. Note that  from the  estimate \eqref{estY1} on $Y$, we can also ensure
         that  for $\hat T_{0}$ (that depends only on $R$) sufficiently small, the map $x \mapsto  y= x+ Y(t,s,x,v)$ is a diffeomorphism with  Jacobian
          $|  \det \, D_{x}y|$ such that  ${1 \over 2} \leq  |  \det \, D_{x}y| \leq 2$. 
          
          We shall next  prove the estimate \eqref{L2Xtilde}. Let us set
           $M_{2m- 1}(t,s)= \sup_{v} \sup_{|\alpha | \leq 2m-1}  \| \partial^\alpha_{x,v} Y(t,s) \|_{L_{x}^2}.$
             It will be also convenient to introduce the function $g(t,s, x,v)= ( x+ (t-s)v + Y(t,s,x,v), v)$ so that
              $\phi(t, x+ (t-s)v +  Y(t,s,x,v), v)= \phi(t) \circ g(t,s)$.
           From \eqref{eqintY}, we thus obtain that
           $$\| \partial_{x,v}^\alpha Y(t,s) \|_{L^\infty_{v} L^2_{x}}
            \lesssim \int_{s}^t \sum_{k, \alpha_{1}, \cdots, \alpha_{k}} \left\| (D^k_{x,v} \phi(\tau)) \circ g(\tau,s)\cdot \left( \partial^{\beta_{1}}_{x,v} g(\tau,s), \cdots, \partial^{\beta_{k}}_{x,v} g(\tau,s) \right) \right\|_{L^\infty_{v} L^2_{x}}
            \, d \tau$$
            where the sum is taken on indices such that $  k \leq | \alpha | \leq 2m-1$, $ \beta_{1}+ \cdots+ \beta_{k}= | \alpha |$
             with for every $j$,  $|\beta_{j}| \geq 1$ and $|\beta_{1}| \leq | \beta_{2}| \leq \cdots\leq  | \beta_{k}|.$
             
             To estimate the right hand side, we first observe that  in the sum, if $k \geq 2$, we necessarily have $|\beta_{k-1}| < 2m-{d \over 2} -1$.
              Indeed, otherwise, there holds $ |\beta_{1}| + \cdots + | \beta_{k} | \geq  4m - d - 2$ and thus $2m-1 \geq 4m-d- 2 $ which yields
                $ 2m \leq d+ 1$ and thus is impossible by assumption on $m$.  Next, 
             \begin{itemize}
             \item if  $k< 2m-{d \over 2}- 1$ and $k \geq 2$ we can write  thanks to the above observation, Lemma \ref{lemburgers} and \eqref{estY1} that
           \begin{align*}
          &  \left\| (D^k_{x,v} \phi(\tau)) \circ g(\tau,s)\cdot \left( \partial^{\beta_{1}}_{x,v} g(\tau,s), \cdots, \partial^{\beta_{k}}_{x,v} g(\tau,s) \right) \right\|_{L^\infty_{v} L^2_{x}}  \\&  \leq  \|D^k \phi\|_{L^\infty_{x,v}}\, \|  \partial^{\beta_{1}}_{x,v} g(\tau,s)\|_{L^\infty_{x,v}} \cdots \| \partial^{\beta_{k-1}}_{x,v} g(\tau,s)\|_{L^\infty_{x,v}} \| \partial^{\beta_{k}} g(\tau, s)\|_{L^\infty_{v} L^2_{x}}\\
           &  \leq  T^{1 \over 2}\Lambda(T, R)(1 +  M_{2m-1}(\tau, s)).
           \end{align*}
           If $k=1$, the above estimate is obviously still valid.
             \item  if $k>2m-{d \over 2}- 1$, we observe that  for every $j$,   $| \beta_{j}| \leq  | \beta_{k}| \leq 2m-1 - (k-1) < 1 + {d \over 2}$. Thus
              $| \beta_{j}| <   2m - {d \over 2 }-1$  by the assumption on $m$ and we get by using \eqref{estY1} that 
            $$   \| \partial^{\beta_{j}}_{x,v} g(\tau,s)\|_{L^\infty_{x,v}} \lesssim  1 + T + \|\partial^{\beta_{j}}_{x,v} Y(\tau,s)\|_{L^\infty_{x,v}}
              \lesssim   \Lambda(T, R).
            $$
        This yields
       \begin{align*}  \left\| (D^k_{x,v} \phi(\tau)) \circ g(\tau,s)\cdot \left( \partial^{\beta_{1}}_{x,v} g(\tau,s), \cdots, \partial^{\beta_{k}}_{x,v} g(\tau,s) \right) \right\|_{L^\infty_{v} L^2_{x}} &  \lesssim \left\| (D^k_{x,v} \phi(\tau)) \circ g(\tau,s) \right\|_{L^\infty_{v} L^2_{x}} \Lambda(T, R)
        \\ &  \lesssim T^{1 \over 2 }\Lambda(T, R).
       \end{align*}
       To get the last estimate, we have used that  thanks to the choice of $\hat T_{0}$, we can use the change of variable
      $ y= x+  Y(t,s,x,v)$ when computing the $L^2_{x}$ norm of $(D^k_{x,v} \phi(\tau)) \circ g(\tau,s)$
       and the estimates of Lemma \ref{lemburgers}.
             \end{itemize}
             By combining the above estimates, we obtain that
           $$ M_{2m-1}(t,s) \leq  (t-s) T^{1 \over 2} \Lambda(T, R) + \int_{s}^t T^{1 \over 2} \Lambda(T, R) M_{2m-1}(\tau,s)\, d\tau.$$
            By using again the Gronwall inequality, we thus obtain that
            $$   M_{2m-1}(t,s) \leq (t-s) T^{1 \over 2} \Lambda(T,R)$$
             and thus by using that $\tilde X(t,s,x,v)=  Y(t,s,x,v)/(t-s)$, we finally obtain \eqref{L2Xtilde}.
              To estimate the time derivative, it suffices to combine the above arguments with the expression
              \eqref{eqXtildebis} for $\partial_{t} \tilde X$.

  To construct $\Psi$, 
    it suffices to notice that the map
    $ v \mapsto v +  \tilde{X}(t,s,x,v)$ is for $T$ sufficiently small a  Lipschitz small perturbation of the identity and
      hence an homeomorphism on $\mathbb{R}^d$.  We can define $\Psi$ as its inverse.
       The claimed regularity follows easily by using the same composition estimates as above and the regularity of $\tilde X$.
   \end{proof}
   
  Define now the tensor  $\mathcal{M}$ by the formula $(\mathcal{M} H)_{i,j}=  \mathcal{M}_{i,j} H$ for all $i,j \in \{1,\cdots, d\}^m$ (with $\mathcal{M}_{i,j}$ defined in \eqref{eq-Mij}) and for $0 \leq s, t \leq T$, $x \in \mathbb{T}^d$, $v \in \R^d$, introduce  the matrix $ \mathfrak{M}(t,s,x,v)$ as the solution of
   \begin{equation}
   \label{defexp}
 \partial_t    \mathfrak{M}(t,s,x,v) = -  \mathcal{M}(t,x,\Phi(t,x,v))  \mathfrak{M}(t,s,x,v), \quad    \mathfrak{M}(s,s,x,v) = I.
   \end{equation}
   Note that by a straightforward Gronwall type argument and \eqref{energy-constraintbis}-\eqref{estim-Burgers-infini}, we have
   \begin{equation}
   \label{frakM}
 \sup_{ 0 \leq s, t \leq T} \left(    \|\mathfrak{M}\|_{W^{k,\infty}_{x,v}} +  \| \pa_t\mathfrak{M}\|_{W^{k,\infty}_{x,v}} + \| \pa_s\mathfrak{M}\|_{W^{k,\infty}_{x,v}}\right)  \leq \Lambda(T,R), \qquad k < 2m -d/2 -2.
   \end{equation}
   
   We shall now  show that the study of \eqref{galphabeta} can be reduced to that  of a system of  integral  equations with a well controlled remainder.

   \begin{lemma}
   \label{redu}
    For  a smooth vector field $G(t,s,x,v)$, define  the following integral operators $ K_{G}$ acting on functions $F(t,x)$:
 $$ K_{G}(F) (t,x) =    \int_{0}^t \int (\nabla_{x}  F) (s,  x - (t-s) v) \cdot
      G(t,s,x,v)\, dv ds.$$
    For $f$ solving \eqref{VP} and $\rho= \int f \, dv$, the function $\partial_{x}^{\alpha(i,j)} \rho$
      satisfies an equation under the form
     \begin{equation}
     \label{eqrho}\partial_{x}^{\alpha(i,j) } \rho = \sum_{k,l} K_{ H_{(k,l), \, (i,j)} }( (I - \eps^2 \Delta)^{-1} \partial_{x}^{\alpha(k,l) } \rho) + \Rc_{i,j},
     \end{equation}
          with
     \begin{multline}
   %  \label{Hdef}
       H_{(k,l), \, (i,j)} :=    \mathfrak{M}_{(k,l), \, (i,j)}(t,s,x - (t-s) v,\Psi(s,t,x,v)) ( \nabla_{v}f)(s,  x- (t-s)v,\Psi(s,t,x,v)) \\
        |\det \na_v \Phi(t,x,\Psi(s,t,x,v))|   | {\det } \nabla_{v}\Psi(s,t,x,v) |,
     \end{multline}
        and the remainder $\Rc_{i,j}$ satisfies for $T< \min (T_{0}, \tilde T_{0}, \hat T_{0}, T^\eps)$ and $2m>d+3$, $2r>d$ the estimate
      \begin{equation}
   \label{borneRlem} \|\Rc_{i,j}\|_{L^2([0, T], L^2_{x})} \lesssim      T^{1\over 2} \Lambda(T,R).% \|f^0\|_{\Hc^{2m}_{n_0}} +   T^{1 \over 2 }\Lambda(T,R).
   \end{equation}

   \end{lemma}
   
   \begin{proof}[Proof of Lemma~\ref{redu}]
   Let us introduce for notational brevity
   $$
\eta (t,x,v) := \left(\partial_{x}^{\alpha(i,j)} E(t,x) \cdot \na_v f(s,x,\Phi(t,x,v))\right)_{i,j}.
   $$
   Using the classical characteristics method, we get that $\mathcal{G}$ satisfying \eqref{galphabeta} solves the integral equation
     \begin{multline*} 
     \mathcal{G}(t,x,v) =   \mathfrak{M}(t,0,x,v)\mathcal{G}^0( X(0,t,x,v),v)-\int_{0}^t    \mathfrak{M}(t,s,x,v)\mathcal{S}(s, X(s,t,x,v), v) \, ds  \\
      -   \int_{0}^t \mathfrak{M}(t,s,x,v) \eta(s,X(s,t,x,v),v) ds
      \end{multline*}
      with $\mathcal{G}^0= (g_{i,j}^0)$.
      Hence, after  multiplying by $  J(t,x,v)=|\det \na_v \Phi(t,x,v)|$ and integrating in $v$, we get that
  \begin{equation}
    \label{eqintegrale1}
   \int_{\mathbb{R}^d}   \mathcal{G}(t,x,v) J(t,x,v) dv = \mathcal{I}_{0} + \mathcal{I}_{F} -   \int_{0}^t \mathfrak{M}(t,s,x,v) \eta(s,X(s,t,x,v),v) J(t,x,v)\, dv ds
\end{equation}
with
\begin{align*}
 \mathcal{I}_{0} &:= \int_{\mathbb{R}^d} \mathfrak{M}(t,0,x,v)\mathcal{G}^0( X(0,t,x,v),v) J(t,x,v)\, dv,  \\
 \mathcal{I}_{F} &:= -  \int_{0}^t \int_{\mathbb{R}^d}    \mathfrak{M}(t,s,x,v)\mathcal{S}(s, X(s,t,x,v), v)J(t,x,v) \,dv ds.
 \end{align*} 
We shall estimate $\mathcal{I}_{0}$ and $\mathcal{I}_{F}$.
First by using the estimates~\eqref{frakM} and \eqref{estim-Burgers-infini}, it follows that for all $x \in \mathbb{T}^d$,
    $$  \left|\int \mathfrak{M}(t,0,x,v)\mathcal{G}^0( X(0,t,x,v),v)  J(t,x,v)\, dv \right|  \leq \Lambda(T, R) \sum_{i,j} \int |g_{i,j}^0 (X(0,t,x,v),v) |\, dv.$$
     Therefore, we obtain that
     $$ \| \mathcal{I}_{0} \|_{L^2([0, T], L^2_{x})} \leq \Lambda(T, R)  \sum_{i,j}  \left\|\int_{v} \|g_{i,j}^0 (X(0,t,\cdot,v),v) \|_{L^2_{x}}\, dv \right\|_{L^2(0, T)}.$$
      By using the change of variable in $x $,  $y= X(0, t, x, v) + tv=  x - t \tilde X(0, t, x, v)$ and  Lemma \ref{redressement2}, we obtain that
      $$  \|g_{i,j}^0 (X(0,t,\cdot,v),v) \|_{L^2_{x}} \leq \Lambda(T, R) \|g_{i,j}^0(\cdot - tv, v)\|_{L^2} \leq \Lambda(T, R) \| g_{i,j}^0(\cdot, v) \|_{L^2_{x}}$$
       and hence, we get from Cauchy-Schwarz that
       $$  \| \mathcal{I}_{0} \|_{L^2([0, T], L^2_{x})} \leq T^{1 \over 2} \Lambda(T, R) \left(\int_{\mathbb{R}^d}    \frac{dv}{(1+|v|^2)^{r}}\right)^{\frac{1}{2}} \sum_{i,j} \| g_{i,j}^0 \|_{\Hc^0_{r}}.$$
        By using again Lemma \ref{lemburgers} and the fact that at $t=0$ we have  that  $L^{(i,j)}= \partial_{x}^{\alpha (i,j)}$ we get   that
         $ \| g_{i,j}^0 \|_{\Hc^0_{r}} \leq \Lambda(T, R) \| f_{i,j}^0 \|_{\Hc^0_{r}} \leq \Lambda(T, R) \| f^0\|_{\Hc^{2m}_{r}}$ and
          hence we finally obtain that
      $$  \| \mathcal{I}_{0} \|_{L^2([0, T], L^2_{x})} \leq  T^{1 \over 2}\Lambda(T, R).$$
       
%   Then, by using Lemma \ref{redressement2},  Lemma \ref{lemburgers}, we can use the change of variables
%    involving $\Psi$ to obtain that 
%     $$  \left|\int \mathfrak{M}(t,0,x,v)\mathcal{G}^0( X(0,t,x,v),v)  J(t,x,v)\, dv \right|  \leq \sup_{i,j} \int |g_{i,j}^0| (x- vt, v) p(t,x,v)\, dv$$
%      with for all $t \in [0,T]$,
%      $$\|p(t,x,v)  \|_{L^\infty_{x,v}} \lesssim  1 + T^{ 1 \over 2} \Lambda(T,R) .$$
%      Consequently %recalling     $\varphi_{k,l}^{i,j}\vert_{t=0} =  \psi_{k,l}^{i,j}\vert_{t=0} =0$
%      we get that
%    $$ \left\|  \int \mathfrak{M}(t,0,x,v)\mathcal{G}^0( X(0,t,x,v),v)  J(t,x,v)\, dv  \right \|_{L^2_{T} L^2_{x}}
% \lesssim \|f^0\|_{\Hc^{2m}_{n_0}} (   1 + T^{ 1 \over 2} \Lambda(T,R)).$$ %\lesssim \Lambda(T,R) .$$
By using similar arguments, we can estimate $\mathcal{I}_{F}$. Indeed,  we can use successively \eqref{frakM}, the change of variable $x\mapsto  X(s,t,x,v)$ with Lemma \ref{redressement2}
 and Cauchy-Schwarz to obtain   that 
\begin{align*}
\| \mathcal{I}_{F}\|_{L^2([0, T], L^2_{x})}&  \leq  \Lambda(T, R) \sum_{i,j} \left \| \int_{0}^t \int_{\mathbb{R}^d}  \|\mathcal{S}_{i,j} (s, X(s,t,\cdot,v), v) \|_{L^2_{x}} \,dv ds\, \right\|_{L^2(0, T)} \\
 & \leq  \Lambda(T, R) \sum_{i,j}\left\|  \int_{0}^t \int_{\mathbb{R}^d} \|\mathcal{S}_{i,j} (s, \cdot, v)\|_{L^2_{x}} \,dv ds \right\|_{L^2(0, T)} \\
 &  \leq  \Lambda(T, R) \left\| \int_{0}^t \| \mathcal{S}(s)\|_{\Hc^0_{r}}\, ds  \right\|_{L^2(0, T)} \\
 & \leq  \Lambda(T, R) \,T\,  \| \mathcal{S}\|_{L^2([0, T], \Hc^0_{r})}.
  \end{align*}
  Finally, since $\mathcal{S}(t,x,v)= F(t,x, \Phi(t,x,v))$, we can use Lemma \ref{lemburgers} and \eqref{restF} to obtain that
  $$  \| \mathcal{S}\|_{L^2([0, T], \Hc^0_{r})} \leq \Lambda(T, R) \| F\|_{L^2([0, T], \Hc^0_{r})} \leq \Lambda(T, R).$$
   We have thus proven that
  $$ \| \mathcal{I}_{F} \|_{L^2([0, T], L^2_{x})} \leq  T\Lambda(T, R).$$

 \bigskip
 
By using Lemma \ref{lemequiv}, we eventually obtain from \eqref{eqintegrale1} and the above estimates that
 \begin{multline*}
  \partial_{x}^{\alpha(i,j) } \rho = \Rc_{i,j} 
   -   \int_{0}^t \int  \sum_{k,l} \mathfrak{M}_{(k,l), \, (i,j)}(t,s,x,v)(\partial_{x}^{\alpha(k,l)} E) (s, X(s,t,x,v))  \\ \cdot( \nabla_{v}f)\left((s,  X(s,t,x,v),\Phi(s, X(s,t,x,v),v)\right) J(t,x,v)\, dv ds
  \end{multline*}
   with
   \begin{equation}
   \label{borneR} \|\Rc_{i,j}\|_{L^2([0, T], L^2_{x})} \lesssim      T^{1\over 2} \Lambda(T,R).
   \end{equation}
   Thanks to Lemma \ref{redressement2}, we can use   the change of variable $ v= \Psi(s,t,x,w)$ (and relabel $w$ by $v$) to  end up with the integral equation
   \begin{equation}
   \label{equationintegrale1}
    \partial_{x}^{\alpha (i,j)  } \rho =  -   \int_{0}^t \int  \sum_{k,l} ( \partial_{x}^{\alpha(k,l)} E) (s,  x - (t-s) v) \cdot
       H_{(k,l), \, (i,j)}(t,s,x,v)\, dv ds + \Rc_0
      \end{equation}
     with
     \begin{multline}
     \label{Hdef}
       H_{(k,l), \, (i,j)}=   \mathfrak{M}_{(k,l), \, (i,j)}(t,s,x,\Psi(s,t,x,v))  ( \nabla_{v}f)(s,  x- (t-s)v,\Psi(s,t,x,v)) \\
      J(t,x,\Psi(s,t,x,v)) \tilde J(s,t,x,v),
     \end{multline}
     and $\tilde{J}(s,t,x,v)= | \det \nabla_{v}\Psi(s,t,x,v) |$, which, recalling the definition of the electric field $E= -\nabla (I -\eps^2 \Delta)^{-1} \rho$, corresponds to the claimed formula \eqref{eqrho}.

     \end{proof}

       We shall now study the boundedness of the operators $K_G$ for functions in $L^2([0, T],  L^2_{x})$.
       Although $K_G$ seems to feature a loss of one derivative in $x$, we shall see that if the function $G$ is sufficiently smooth, then $K_{G}$ is actually a bounded operator on 
         $L^2([0, T],  L^2_{x})$. This means that we can recover the apparent loss of derivative by using the averaging in $v$, which is reminiscent of averaging lemmas
        (note that we nevertheless require regularity on $G$).
     \begin{proposition}
     \label{propint}
      There exists $C>0$ such that for every $T>0$ and every  $G$ with
     \begin{equation}
     \label{Hhyp} 
      \|G \|_{T, s_{1}, s_{2} } := \sup_{0 \leq t \leq T} 
      \left(\sum_{k} \sup_{0 \leq s \leq T} \sup_{\xi}  \Big( ( 1+|k|)^{s_{2}} (1 + |\xi|)^{s_{1}} |  (\mathcal{F}_{x,v} G)(t,s,k, \xi) |\Big)^2 \right)^{1 \over 2} 
     \end{equation}
     and $s_{1}>1, $ $s_{2}>d/2$, then 
     we have the estimate
     $$ \| K_{G}(F)\|_{L^2([0, T], L^2_{x})} \leq C \|G \|_{T, s_{1},s_{2}} \|F\|_{L^2([0, T], L^2_{x})}, \quad \forall F \in L^2([0, T], L^2_{x}).$$
     \end{proposition}
     \begin{remark}
     \label{remarknormekernel}
   For practical use, it is  convenient to relate   $\|G \|_{T, s_{1}, s_{2} }$ to a more  tractable norm.
    A first way to do it is to observe that if $p>1 + d$, $\sigma>d/2$, we can find $s_{2}>d/2$, $s_{1}>1$ such that
   $$\Big( ( 1+|k|)^{s_{2}} (1 + |\xi|)^{s_{1}} |  (\mathcal{F}_{x,v} G)(t,s,k, \xi) |\Big)^2  \leq
    { 1 \over ( 1+|k|)^{ 2s_{2}} } \| G(t,s) \|_{\Hc^p_{\sigma}}^2$$ and 
    thus we obtain that
    $$  \|G \|_{T, s_{1}, s_{2} } \lesssim \sup_{0 \leq s,\, t \leq T} \| G(t,s) \|_{\Hc^p_{\sigma}}.$$ 
     Note that this requires roughly $1+d$ derivatives of the function $G$.
      In the following,  we shall  need only  the above Proposition in the following two cases
      for which we can reduce the number of derivatives needed on the function $G$.
           \begin{itemize}
     \item  If $G(t,s,x,v)= G(t,x,v)$ is independent of $s$,  then we have thanks to  the Bessel-Parseval identity that 
     \begin{equation}
     \label{normkernel1}   \|G \|_{T, s_{1}, s_{2} }  \leq \sup_{0 \leq t \leq T}  \|G(t)\|_{\Hc^p_\sigma}
     \end{equation}
      for any integer $p$ such that  $p>1 + {d \over 2}$ and any $\sigma$, $\sigma>{d\over 2}.$
      \item If $G(t,t,x,v) = 0$, since
       \begin{align*} 
       &  \left( ( 1+|k|)^{s_{2}}) (1 + |\xi|)^{s_{1}} \,|\mathcal{F}_{x,v}(G)(t,s,k, \xi) |\right)^2 \\
       &  \leq  T  \left|\int_{s}^t    \left( ( 1+|k|)^{s_{2}} (1 + |\xi|)^{s_{1}} \,|\mathcal{F}_{x,v}(\partial_{s}G)(t,\tau,k, \xi) |\right)^2 \, d\tau \right|
        \\
        &  \leq  T  \left|\int_{0}^t    \left( ( 1+|k|)^{s_{2}} (1 + |\xi|)^{s_{1}} \,|\mathcal{F}_{x,v}(\partial_{s}G)(t,\tau,k, \xi) |\right)^2 \, d\tau \right|,
        \end{align*}
        we obtain, by using again the Bessel-Parseval identity that
      $$ \|G \|_{T, s_{1}, s_{2} }  \leq  T^{1 \over 2} \sup_{0 \leq t \leq T}    \left(\int_{0}^t \| \partial_{s} G(t, s) \|_{\Hc^p_{\sigma}}^2\, ds \right)^{1 \over 2}$$
       and hence that 
       \begin{equation}
       \label{normekernel2}   \|G \|_{T, s_{1}, s_{2} }  \leq  T \sup_{0 \leq t \leq T} \sup_{0 \leq s \leq t} \|\partial_{s} G(t, s) \|_{\Hc^p_{\sigma}}
       \end{equation}
        with  $p> 1+ {d \over 2}$ and $\sigma>{d\over 2}.$ 
     \end{itemize} 
   \end{remark}
   \begin{proof}[Proof of Proposition \ref{propint}]
    By using Fourier series in $x$, we can write
     that 
     $$ F(t,x)= \sum_{k \in \mathbb{Z^d}} \hat F_{k}(t) e^{i k \cdot x}.$$
      This yields
     \begin{align}
     \label{KGbis}
      K_{G} F (t,x) & =  \int_{0}^t  \sum_{k}   \hat F_{k}(s) e^{ik \cdot x} ik \cdot \int e^{- ik \cdot v(t-s)} G(t,s,x,v)\, dv\, ds
       \\
       & =  \int_{0}^t \sum_{k }   \hat F_{k}(s) e^{ik \cdot x} ik \cdot (\mathcal{F}_{v} G)(t,s,x, k(t-s))\, ds
       \end{align}
       where $\mathcal{F}_{v} G$ stands for the Fourier transform of $G(t,s,x,v)$ with respect to the last variable.
        Next, expanding also  $G$ in Fourier series in the $x$ variable, we get that
    $$ 
       K_{G} F (t,x)  =  \sum_{k }   e^{ik \cdot x}  \sum_{l}  e^{il \cdot x}  \int_{0}^t \hat F_{k}(s) ik \cdot (\mathcal{F}_{x,v} G)(t,s,l, k(t-s))\, ds.$$
       Changing $l$ into $l+k$ in the second sum, we can also write this expression as
       $$  K_{G} F (t,x)  =   \sum_{l}  e^{il \cdot x} \Big( \sum_{k }   \int_{0}^t \hat F_{k}(s) ik \cdot (\mathcal{F}_{x,v} G)(t,s,l - k, k(t-s))\, ds\Big).$$
        From the Bessel-Parseval identity, this yields 
     $$  \|K_{G}\|_{L^2_{x}}^2 = \sum_{l} \Big| \sum_{k} \int_{0}^t   \hat F_{k}(s) ik \cdot (\mathcal{F}_{x,v} G)(t,s,l - k, k(t-s))\, ds \Big|^2.$$
     By using Cauchy-Schwarz in time and $k$, we  next obtain that
    \begin{multline*} \|K_{G}\|_{L^2_{x}}^2  \leq \sum_{l} \Big(\sum_{k} \int_{0}^t | \hat  F_{k}(s)|^2  \, |k \cdot (\mathcal{F}_{x,v} G)(t,s,l - k, k(t-s)) | \,ds \\ \cdot
     \sum_{k} \int_{0}^t  |k \cdot (\mathcal{F}_{x,v} G)(t,s,l - k, k(t-s)) |\,ds \Big).
     \end{multline*}
      By integrating in time, this yields
     \begin{multline}
     \label{kernel1}
       \|K_{G}\|_{L^2([0, T], L^2_{x})}^2 \leq
        \sum_{l} \int_{0}^T \int_{0}^t  \sum_{k} | \hat  F_{k}(s)|^2  \, |k \cdot (\mathcal{F}_{x,v} G)(t,s,l - k, k(t-s)) | \,ds\, dt  \\ 
         \cdot  \sup_{l} \sup_{t\in[0, T]}
        \int_{0}^t \sum_{k}   |k \cdot (\mathcal{F}_{x,v} G)(t,s,l - k, k(t-s)) |\,ds. \leq I \cdot II.
        \end{multline}
      For the second term in the above product  that is $II$, we observe that
      \begin{multline*}  \sup_{l} \sup_{t\in[0, T]}
        \int_{0}^t \sum_{k}   |k \cdot (\mathcal{F}_{x,v} G)(t,s,l - k, k(t-s)) |\,ds\\ 
         \leq  \sup_{l} \sup_{t\in[0, T]}  \sum_{k} \Big(\sup_{0 \leq s \leq t} \sup_{\xi} (1 + |\xi|)^{s_{1}} | (\mathcal{F}_{x,v} G)(t,s,l - k,\xi)|  \int_{0}^t  { |k| \over  (1 + |k|(t-s))^{s_{1}} }\, ds \Big),
         \end{multline*} 
and by choosing $s_{1}>1$, since 
$$   \int_{0}^t  { |k| \over  (1 + |k|(t-s)|)^{s_{1}} }\, ds \leq \int_{0}^{+ \infty} { 1 \over (1 + \tau^{s_{1}})}\, d\tau <+\infty$$
is independent of $k$, 
 we obtain
 \begin{multline*}  \sup_{l} \sup_{t\in[0, T]}
        \int_{0}^t \sum_{k}   |k \cdot (\mathcal{F}_{x,v} G)(t,s,l - k, k(t-s)) |\,ds\\ 
         \leq  \sup_{l} \sup_{t\in[0, T]}  \sum_{k} \sup_{0 \leq s \leq t} \sup_{\xi} (1 + |\xi|)^{s_{1}} |( \mathcal{F}_{x,v} G)(t,s,l - k,\xi)|.
         \end{multline*}
       Next, by  choosing $s_{2}>d/2$ and  by using Cauchy-Schwarz, this finally yields
      \begin{equation}
      \label{kernel3}  II%= \sup_{l} \sup_{t\in[0, T]}
      %  \int_{0}^t \sum_{k}   |k \cdot (\mathcal{F}_{x,v} G)(t,s,l - k, k(t-s)) |\,ds\\ 
         \leq    \sup_{t\in[0, T]} \Big( \sum_{k}   \sup_{0 \leq s \leq t} \sup_{\xi}\Bigl( ( 1+|k|)^{s_{2}}(1 + |\xi|)^{s_{1}} | (\mathcal{F}_{x,v} G)(t,s,k,\xi)|  \Bigr)^2
         \Big)^{1 \over 2}.
         \end{equation}
       It remains to estimate $I$ in the right-hand side of \eqref{kernel1}.  By using Fubini, we have
    \begin{align*}
     &  \sum_{l} \int_{0}^T \int_{0}^t  \sum_{k} | \hat  F_{k}(s)|^2  \, |k \cdot (\mathcal{F}_{x,v} G)(t,s,l - k, k(t-s)) | \,ds\, dt  
     \\
     & = \int_{0}^T \sum_{k} | \hat F_{k}(s)|^2  \int_{s}^T \sum_{l} |k|  |(\mathcal{F}_{x,v} G)(t,s,l - k, k(t-s)) | \, dt \, ds \\
      & \leq \| F \|_{L^2([0, T], L^2_{x})}^2 \sup_{k} \sup_{0 \leq s \leq t} \int_{s}^T \sum_{l}  |k|  |(\mathcal{F}_{x,v} G)(t,s,l - k, k(t-s)) | \, dt.
      \end{align*}
       Next, by choosing $s_{1}>1$ and $s_{2}>d/2$ as above, we observe that
    \begin{align*}
    &  \sup_{k} \sup_{0 \leq s \leq T} \int_{s}^T \sum_{l}  |k|  |(\mathcal{F}_{x,v} G)(t,s,l - k, k(t-s)) | \, dt \\
     &  \leq   \sup_{k} \sup_{0 \leq s \leq T} \int_{s}^T { |k| \over (1 + |k| (t-s))^{s_{1}}}  \sum_{l} \sup_{\xi} (1 + |\xi|)^{s_{1}} |  (\mathcal{F}_{x,v} G)(t,s,l - k, \xi) |
       \, dt  \\
   &   \leq   \sup_{k} \sup_{0 \leq s \leq T} \int_{s}^T { |k|\, dt \over (1 + |k| (t-s))^{s_{1}}} \\
   & \qquad \qquad \qquad  \times \sup_{0 \leq s \leq T} \sup_{s \leq t \leq T}
      \left(\sum_{m}\sup_{\xi}  \Big( ( 1+|m|)^{s_{2}} (1 + |\xi|)^{s_{1}} |  (\mathcal{F}_{x,v} G)(t,s,m, \xi) |\Big)^2 \right)^{1 \over 2}.
    \end{align*}
    Since we have again
    $$    \sup_{k} \sup_{0 \leq s \leq T} \int_{s}^T { |k|\, dt \over (1 + |k| (t-s))^{s_{1}}} \leq \int_{0}^{+\infty} {d\tau \over (1 + \tau)^{s_{1}}}\, d\tau <+\infty,$$ 
    we have proven that
    \begin{equation}
    \begin{aligned}
    \label{kernel2} I &\lesssim     \| F \|_{L^2([0, T], L^2_{x})}^2 \sup_{0 \leq s \leq T} \sup_{s \leq t \leq T}
      \left(\sum_{m}\sup_{\xi}  \Big( ( 1+|m|)^{s_{2}} (1 + |\xi|)^{s_{1}} |  (\mathcal{F}_{x,v} G)(t,s,m, \xi) |\Big)^2 \right)^{1 \over 2}
        \\ &\lesssim  \| F \|_{L^2([0, T], L^2_{x})}^2  \sup_{0 \leq t \leq T}  \left(\sum_{m} \sup_{0 \leq s \leq t} \sup_{\xi}  \Big( ( 1+|m|)^{s_{2}} (1 + |\xi|)^{s_{1}} |  (\mathcal{F}_{x,v} G)(t,s,m, \xi) |\Big)^2 \right)^{1 \over 2}.\end{aligned}
        \end{equation}
     We finally get the result by combining \eqref{kernel1}, \eqref{kernel2} and \eqref{kernel3}.

   \end{proof} 
 
 We can then  use Proposition~\ref{propint} to simplify the system of equations~\eqref{eqrho} for $\rho$ in Lemma~\ref{redu}.
   \begin{lemma}
   \label{lem-simplification}
     Assume $2m >m_{0}$ and $2r >r_{0}$.  For $f$ solving \eqref{VP} and $\rho= \int f \, dv$, for every  $i,j \in \{1,\cdots, d\}^m$, the function $\partial_{x}^{\alpha(i,j)} \rho$
      satisfies an equation under the form
     \begin{equation}
     \label{eqrho1}\partial_{x}^{\alpha(i,j) } \rho = K_{\nabla_v f^0}( (I - \eps^2 \Delta)^{-1} \partial_{x}^{\alpha(i,j) } \rho) + R_{i,j},
     \end{equation}
    where the remainder $R_{i,j}$ satisfies%is already well controlled.
      \begin{equation}
   \label{borneRlem1} \|R_{i,j}\|_{L^2([0, T], L^2_{x})} \lesssim  T^{1\over 2} \Lambda(T,R).% \|f^0\|_{\Hc^{2m}_{n_0}} +   T^{1 \over 2 }\Lambda(T,R)
   \end{equation}
for    $T< \min (T_{0}, \tilde T_{0}, \hat T_{0}, T^\eps)$.

     \end{lemma}
 
 \begin{proof}[Proof of Lemma~\ref{lem-simplification}]
We keep the same notations as in Lemma~\ref{redu} and in particular we use the expression \eqref{Hdef}.
We can first write that 
 \begin{equation}
 \label{H1definition} H_{(i,j), \, (k,l)}(t,s,x,v)= H_{(i,j),\,(k,l)}(t,t,x,v) +  H_{(i,j),\,(k,l)}^1(t,s,x,v)
 \end{equation}
  with 
  $$ H_{(i,j),\,(k,l)}^1(t,s,x,v)   =H_{(i,j), \, (k,l)}(t,s,x,v) - H_{(i,j),\,(k,l)}(t,t,x,v).$$
  Since $H_{(i,j),\,(k,l)}^1(t,t,x,v)=0$, we can use \eqref{normekernel2} in  Remark \ref{remarknormekernel}
   to get  that
   $$\|K_{H_{(i,j),\,(k,l)}^1} \partial_{x}^{\alpha(k,l)}\rho\|_{L^2([0, T], L^2_{x})}
    \lesssim T \sup_{t,s}\| \partial_{s}  H_{(i,j),\,(k,l)}^1(t,s) \|_{H^{p}_{\sigma}}\| \rho \|_{L^2([0, T], H^{2m})} $$  with 
      $p = 1 + p_{0}$ ($p_{0}$ being defined in \eqref{seuil}) and  $\sigma$ such that $\sigma>{d\over 2}$
       and $ 1+ \sigma <2r$.
     We thus have  to estimate $ \sup_{t,s}\| \partial_{s}  H_{(i,j),\,(k,l)}^1(t,s) \|_{\Hc^{p}_{\sigma}}$.
     
      Note that, by assumption on $m$, we have that 
       $  p + 2 < 2m- {d \over 2}$. We use again the notation
       $$
       J(t,x,v)= |\det(\nabla_v \Phi(t,x,v))|, \quad \tilde J (s,t,x,v)= |\det(\nabla_v \Psi(s,t,x,v))|.
       $$
       According to \eqref{frakM}, \eqref{estimPsi}, \eqref{estim-Burgers-infini}, we can always put the terms involving $\mathfrak{M}$, 
     $\Psi$, $J$,   $\tilde{J}$ and their derivatives  in $L^\infty$, except when all the  derivatives hit $J$ or $\tilde{J}$.
      Note that  due to the expression \eqref{Hdef},  to compute  $\partial_{s} H^{1}$, we need the derivative of $\Psi$ and $\mathfrak{M}$
       with respect to their first argument so that we actually need estimates of $\partial_{t}\Psi$ in view of our previous
       notation.
      This yields
   \begin{multline}
   \label{estH11} \| \partial_{s}  H_{(i,j),\,(k,l)}^1(t,s) \|_{\Hc^{p}_{\sigma}} \leq \Lambda(T, R) \Big(
     \|  \partial_{t} \nabla_{v} f\|_{\Hc^{p}_{\sigma}}  + \left\| ( 1 + |v|^2)^{1 \over 2} \nabla_{v} f \right\|_{\Hc^{p+ 1}_{\sigma}}  
      \\+ \sum_{| \alpha|= p}  \left\| | \nabla_{v} f(s,x-(t-s)v, \Psi)|\, |(\partial_{x,v}^{\alpha} \nabla_{v}J) (t, \cdot, \Psi)|\right\|_{\Hc^0_{\sigma}}
       + \left\| | \nabla_{v} f(s,x-(t-s)v, \Psi)|\, |\partial_{x,v}^{\alpha}\partial_{s} \tilde{J}|\right\|_{\Hc^0_{\sigma}}\Big).
       \end{multline}
       Note that to obtain this estimate,  we have used that integrals under the form 
       $$ I= \left|  \int_{\mathbb{T}^d \times \mathbb{R}^d} |g(x-v(t-s), \Psi(s,t,x,v) )|^2 (1+ |v|^2)^n \, dx dv \right|$$
        with $n= \sigma$ or $\sigma + 1$ can be bounded by $ \Lambda(T, R)\|g\|_{\Hc^0_{n}}^2.$
          Indeed, by setting, $v \mapsto w=  \Psi(s,t,x,v)$ and by using Lemma \ref{redressement2} (in particular the fact that
           the Jacobian of the change of variable is bounded and the fact that $|w-v|$ is bounded), we get
          $$ I \leq \Lambda(T, R)  \int_{\mathbb{T}^d \times \mathbb{R}^d} |g(X(s,t,x,w), w) |^2  (1 + |w|^2)^n\, dx dw.$$
          Next, we can use again Lemma \ref{redressement2} and the change of variable $x \mapsto y= X(s,t,x,w)$
           to finally obtain
           $$   I \leq \Lambda(T, R)  \|g\|_{\Hc^0_{n}}^2.$$
            Going back to \eqref{estH11}, we observe that by using the equation \eqref{VP}, we get that
        $$  \|  \partial_{t} \nabla_{v} f\|_{\Hc^{p}_{\sigma}}  + \left\| ( 1 + |v|^2)^{1 \over 2} \nabla_{v} f \right\|_{\Hc^{p+ 1}_{\sigma}}  
         \lesssim  \|f\|_{\Hc^{2m- 1}_{2r}}$$
         since $2m > m_{0}$ implies that   $2m\geq 4 + p_{0}$. Also, 
           by  using again $L^\infty$ estimates, we have for $| \alpha |=p$, 
       \begin{multline*}
       \left\| | \nabla_{v} f(s,x-(t-s)v, \Psi)|\, |(\partial_{x,v}^{\alpha} \nabla_{v}J) (t, \cdot, \Psi)|\right\|_{\Hc^0_{\sigma}}
       + \left\| | \nabla_{v} f(s,x-(t-s)v, \Psi)|\, |\partial_{x,v}^{\alpha}\partial_{s} \tilde J|\right\|_{\Hc^0_{\sigma}}\Big) \\
        \leq \Lambda(T, R)   \Big(\left\| | \nabla_{v} f(s,x-(t-s)v, \Psi)|\, |(\partial_{x,v}^{\alpha} \nabla_{v}^2\Phi) (t, \cdot, \Psi)|\right\|_{\Hc^0_{\sigma}}
         \\+ \left\| | \nabla_{v} f(s,x-(t-s)v, \Psi)|\, |\partial_{x,v}^{\alpha}\partial_{s} \nabla_{v} \Psi|\right\|_{\Hc^0_{\sigma}}\Big)
       \end{multline*}
       and we estimate the above right-hand side by  
       $$ \Lambda(T, R) \| (1+ |v|^2)^{\sigma +1\over 2} \nabla_{v} f \|_{L^\infty_{x,v}} \Big( \|  (\partial_{x,v}^{\alpha} \nabla_{v}^2\Phi) (t, \cdot, \Psi)\|_{L^\infty_{v}L^2_{x}} 
        +  \|(1+ |v|^2)^{-{1\over 2}}\partial_{x,v}^{\alpha}\partial_{s} \nabla_{v} \Psi\|_{L^\infty_{v}L^2_{x}}\Big).$$
         Since   $2m\geq 4 + p_{0}$,  the above expression can be again  finally bounded  by $\Lambda(T, R)$
          by using \eqref{estimPsi}, \eqref{estim-Burgers} and the Sobolev embedding in $x,\, v$
           to estimate $ \| (1+ |v|^2)^{\sigma +1\over 2} \nabla_{v} f \|_{L^\infty_{x,v}}.$
          We have thus proven that 
          $$\|K_{H_{(i,j),\,(k,l)}^1} \partial_{x}^{\alpha(k,l)}\rho\|_{L^2([0, T], L^2_{x})}
    \lesssim T \Lambda(T, R)$$
    and as a consequence, that this term can be included in the remainder.

    \bigskip
     In view of \eqref{H1definition} and the above estimate, since 
    $$H_{(i,j),\,(k,l)}(t,t,x,v)= \delta_{(i,j),\,(k,l)} \nabla_{v} f(t, x, v) J(t,x,v),$$ 
    the integral system \eqref{eqrho} reduces to 
    $$ \partial_{x}^{\alpha(i,j)} \rho = K_{ \nabla_{v}f(t) J(t)}(\partial_{x}^{\alpha(i,j)}) \rho + \mathcal{R}_{(i,j)}^1$$
     with $ \mathcal{R}_{(i,j)}^1$ that satisfies
     $$\left\|  \mathcal{R}_{(i,j)}^1 \right\|_{L^2([0, T], L^2_{x})} \leq T^{1 \over 2} \Lambda(T, R).$$
      We can further simplify this integral equation by writing
      $$  \nabla_{v} f(t, x, v) J(t,x,v) =  \nabla_{v} f^0(x,v) + \left( \nabla_{v}f(t, x, v)  - \nabla_{v} f^0(x,v)\right) J(t,x,v)  + \nabla_{v} f^{0}(x,v) (J(t,x,v) - 1).$$
      Let us set 
      $$G(t,x,v) =  \left( \nabla_{v}f(t, x, v)  - \nabla_{v} f^0(x,v)\right) J(t,x,v) +   \nabla_{v} f^{0}(x,v) (J(t,x,v) - 1).$$ 
       By using Proposition \ref{propint} and  \eqref{normkernel1} in Remark \ref{remarknormekernel}, we obtain that
       $$ \|K_{G}( \partial_{x}^{\alpha(i,j)} \rho )\|_{L^2([0, T], L^2_{x})}
       \leq  (\sup_{[0, T]} \| G(t)\|_{\Hc^p_{\sigma}})  \| \rho \|_{L^2([0, T], H^{2m})},$$
        with $p= 1 + p_{0}$ and $\sigma >{d \over 2}, $ $1+ \sigma \leq 2r$.
          From the definition of $G$, we find
        $$ \| G(t)\|_{\Hc^p_{\sigma}}  \lesssim \| \nabla_{v}f(t,\cdot) - \nabla_{v}f^0 \|_{\Hc^p_{\sigma}} \|J(t)\|_{W^{p, \infty}}
         + \| \nabla_{v} f(t, \cdot) \|_{\Hc^p_{\sigma}} \|J(t)-1\|_{W^{p, \infty}}.$$
         Since we have  $ 1 + p < 2m - {d \over 2} -1 $,   $2 + p \leq 2m-1$,  we obtain by using
          \eqref{estim-Burgers-infini} and the equation \eqref{VP} that
          $$   \sup_{[0, T]} \| G(t)\|_{\Hc^p_{\sigma}}
           \leq \Lambda(T, R) \Big( T \sup_{[0, T]} \| \partial_{t}f \|_{\Hc^p_{\sigma}} + T^{1 \over 2}\Big)
            \leq T^{1 \over 2} \Lambda(T, R).$$
          This ends the proof.
         
  \end{proof}
 
 We therefore proceed with the study of the integral scalar equation
\begin{equation}
 \label{eqintegrale2}
 \tilde h(t,x)= K_{\nabla_{v}f^0} ( (I - \eps^2 \Delta)^{-1}\tilde h) + \tilde R(t,x), \quad 0 \leq t \leq T,
 \end{equation}
 where $\tilde R$ is a given source term.
 It will be useful to introduce a positive parameter $\gamma$ (which will be chosen large enough but independent of $\eps$)
  and to set
  \begin{equation}
  \label{expweight} \tilde h(t,x)= e^{  \gamma t}   h(t,x), \quad   \tilde R(t,x) = e^{\gamma t }  \mathcal{R}(t,x)
  \end{equation}
   so that   \eqref{eqintegrale2} becomes
    \begin{equation}
 \label{eqintegrale3}
  h(t,x)=e^{- \gamma t} K_{\nabla_{v}f^0} ( e^{ \gamma t } (I- \eps^2 \Delta)^{-1}h) +   \mathcal{R}(t,x), \quad 0 \leq t \leq T
 \end{equation}
    Without loss of generality, we can assume that $ \mathcal{R}$ is equal to zero for $t<0$ and for $t>T$ and we shall also set
   $h= 0$ for all $t<0$. Note that this does not affect the value of $h$ on $[0, T]$. This allows us to study the equation for  $ t \in \mathbb{R}$.
    Our aim is to prove that if the Penrose condition is satisfied by $f^0$ then we can  estimate $h$ in $L^2_{t,x}$
     with respect to $R$ in $L^2_{t,x}$.
    
    One first key step is to relate $ e^{- \gamma t}K_{\nabla_{v}f^0}(e^{\gamma t} \cdot)$ to a pseudodifferential operator.
  \begin{lemma}
  \label{lemKpseudo}
    Let us set
     \begin{equation}
     \label{adef}
     a(x, \zeta)=   \int_{0}^{+ \infty} e^{-(\gamma + i \tau)s}\,i k \cdot  ( \mathcal{F}_{v} \nabla_{v} f^0)(x,ks) \, d s,  \quad \zeta= (\gamma, \tau, k) \in (0, +\infty) \times \mathbb{R}
      \times \mathbb{R}^d \backslash\{0\}
     \end{equation}
      where again $\mathcal{F}_{v}$ stands for the Fourier transform in the $v$ variable.
      Then, we have that
   $$ e^{- \gamma t}K_{\nabla_{v}f^0}(e^{\gamma t} h) = Op_{a}^\gamma (h), \quad \forall h \in \mathcal{S}$$
    with $Op_{a}^\gamma$ the quantification of $a$ defined in Section \ref{sectionpseudo}.
  \end{lemma}
  Note that  as usual when dealing with pseudodifferential  calculus on the torus,  we manipulate symbols defined in the whole space $\mathbb{R}^d$   in the $k$ variable, though they are only used
   for  $k \in  \mathbb{Z}^d$ in the quantization formula.
  \begin{proof}[Proof of Lemma~\ref{lemKpseudo}]
   Since $h$ is $0$ in the past, we first note
    that
    $$ e^{- \gamma t}K_{\nabla_{v}f^0}(e^{\gamma t} h) =
     \int_{-\infty}^t  e^{- \gamma (t-s)}\int \nabla_{x} h (s,  x - (t-s) v) \cdot
      \nabla_{v} f^0(x,v)\, dv ds.$$
   By using the Fourier transform in $x$ and $t$, we can write
   that 
    $$h(x,s)= \int_{\mathbb{Z}^d \times \mathbb{R} } e^{i (k \cdot x + \tau s) }  \hat h(k, \tau)\, dk d\tau$$
  with the convention that $\mathbb{Z}^d $ is equipped with the discrete measure $dk$. This yields
  \begin{align*} 
  &e^{- \gamma t}K_{\nabla_{v}f^0}(e^{\gamma t} h) \\
   &=   \int_{\mathbb{Z}^d \times \mathbb{R} } e^{i (k \cdot x +  \tau t) }  \Big(  \int_{-\infty}^t   e^{- (\gamma + i \tau)(t-s) }  \int  e^{ -i k \cdot v(t-s)} ik \cdot  \nabla_v f^0(x, v)\, dv\, ds \Big)
    \hat h(k, \tau) dk d\tau\\
  &   =      \int_{\mathbb{Z}^d \times \mathbb{R} } e^{i (k \cdot x +  \tau t)}   \Big(  \int_{-\infty}^t e^{- (\gamma + i \tau)(t-s) }  ik\cdot  ( \mathcal{F}_{v}\nabla_{v} f^0)  (x, k(t-s))\, ds \Big) \hat h(k, \tau)\, dk d \tau.
\end{align*}
 Changing variable in the inside integral, we finally obtain
 $$  e^{- \gamma t}K_{\nabla_{v}f^0}(e^{\gamma t} h)=   \int_{\mathbb{Z}^d \times \mathbb{R} } e^{i (k \cdot x +  \tau t)}  a(x, \zeta) \hat h(k, \tau)\, dk d\tau$$
  with
  $$ a(X, \zeta)=   \int_{0}^{+ \infty} e^{-(\gamma + i \tau)s}\, ik \cdot  ( \mathcal{F}_{v} \nabla_{v}  f^0)(x,ks ) \, d s$$
   as claimed.

  \end{proof}

  Note that $a$ does not actually depend on the time variable $t$. We shall now prove that $a$ defined above is a good zero order symbol.
   The symbol seminorms are defined in Section \ref{sectionpseudo}, see \eqref{pseudo0}, \eqref{pseudo1}.
   \begin{lemma}
   \label{lemsymbole}
    Consider  $a(x, \zeta)$, the symbol defined in \eqref{adef} with  $\zeta= (\gamma, \tau, k)= (\gamma, \xi)$, $ \gamma  >0$
   and  take $\sigma>d/2$.
     Then we have that there exists $C_{M}>0$ that depends only on $M$  such that
   \begin{eqnarray*}
  &     |a |_{M, 0}   \leq C_{M}  \|f^0 \|_{\Hc^{2m}_{\sigma}}, \quad   M  \leq 2 m - 3, \\
   & 
  |a |_{M, 1} \leq C_{M}  \|f^0 \|_{\Hc^{2m}_{\sigma +  2}}, \quad  M  \leq 2 m-4.
   \end{eqnarray*}
   Moreover, $a$ is homogeneous of degree zero:
   $$ a(x, \zeta)= a\left(x,  {  \zeta \over \langle \zeta \rangle}\right).$$
  %  where we recall the notation
   % $$  \|f \|_{\Hc^{n}_{k}}^2 = \sum_{|\a| + |\b| \leq n} \int_{\T^d} \int_{\R^d} (1+ |v|^2)^{k} |\pa^\a_x \pa^\b_v f|^2 \, dv dx.$$
   \end{lemma}
       \begin{proof}[Proof of Lemma~\ref{lemsymbole}]
    Let us set $G(x, \eta)=   ( \mathcal{F}_{v} \nabla_{v} f^0)(x,\eta )$ so that
    $$ a(x, \zeta)=  \int_{0}^{+ \infty} e^{-(\gamma + i \tau)s}\, {ik}\cdot   G(x, ks) \, d s.$$
     and
           $$ \tilde \zeta = { \zeta \over \langle \zeta \rangle}, \quad \langle \zeta \rangle = (\gamma^2 + \tau^2 + |k|^2)^{1 \over 2}.$$
    Note that we have for $\sigma>d/2$ and every $\alpha$, $\beta$, $q$, the estimate
   \begin{equation}
   \label{estsymbol1}
     |(\mathcal{F}_{x}\partial_{x}^\alpha \partial_{\eta}^\beta G)(l, \eta)| \lesssim { 1 \over 1 + |\eta |^q} \left( \int_{\R^d}  
    ( 1 + |  v  |^2)^{\sigma + \beta} \, |(\mathcal{F}_{x} \nabla_{v} \partial_{x}^\alpha ( I - \Delta_{v})^{q \over 2} f^0)(l, v ) |^2 dv \right)^{1 \over 2}.
     \end{equation}
   By a change of variable $s= \tilde s /\langle \zeta \rangle$ in the integral defining $a$, we also easily observe  that $a$ is homogeneous of degree zero
          $$ 
       a(x, \zeta)= a\left(x,  {  \zeta \over \langle \zeta \rangle}\right).
       $$
       % \frac{1+ \eps^2 |k|^2/ \langle \zeta \rangle^2 }{1+ \eps^2 |k|^2}.$$
        Consequently, by using the definition of the symbol norms in the appendix,   it suffices to prove that
      \begin{eqnarray}
  \label{symbolepreuve1}&   \| \mathcal F_{x}(\partial_{x}^\alpha  a)(\cdot, \tilde  \zeta) \|_{L^2(\mathbb{Z}^d, L^\infty(S_{+}))} \lesssim  \|f^0 \|_{\Hc^{2m}_{\sigma}}, \quad   | \alpha|  \leq 2 m - 3 , \\
 \label{symbolepreuve2}  & 
  \| (\mathcal F_{x} \partial_{x}^\alpha \nabla_{\tilde \xi}  a)(\cdot , \tilde  \zeta) \|_{L^2(\mathbb{Z}^d, L^\infty(S_{+}))} \lesssim   \|f^0 \|_{\Hc^{2m}_{\sigma +  1}}, \quad |\alpha | \leq  2m- 4.
   \end{eqnarray}
   where $S_{+}= \{ \tilde \zeta = (\tilde \gamma, \tilde \tau, \tilde k), \, \langle \tilde \zeta \rangle = 1, \, \tilde \gamma >0, \, \tilde k \neq 0\}.$
   
    Since we have
    $$(\mathcal{F}_{x} \partial_{x}^\alpha a)(l,  \tilde \zeta)=   \int_{0}^{+ \infty} e^{-( \tilde \gamma + i  \tilde \tau)s}\, i  \tilde k \cdot   (\mathcal{F}_{x}\partial_{x}^\alpha )G(l,  \tilde k s) \, d s,  $$
     by using \eqref{estsymbol1} with $q=2$, and $\beta = 0$,  we obtain that 
     \begin{align*}   |(\mathcal{F}_{x}\partial_{x}^\alpha a)(l, \tilde  \zeta)|  & \lesssim 
     \left( \int_{\R^d}  
    ( 1 + |  v  |^2)^{\sigma} \, |(\mathcal{F}_{x} \nabla_{v} \partial_{x}^\alpha ( I - \Delta_{v}) f^0)(l, v ) |^2 dv \right)^{1 \over 2}
    \, \int_{0}^{+ \infty}  {| \tilde k | \over 1 + |\tilde k|^2 s^2 } ds  \\
    & \lesssim 
    \left( \int_{\R^d}  
    ( 1 + |  v  |^2)^{\sigma} \, |(\mathcal{F}_{x} \nabla_{v} \partial_{x}^\alpha ( I - \Delta_{v}) f^0)(l, v ) |^2 dv \right)^{1 \over 2} 
      \int_{0}^{+ \infty}  {1 \over 1 +  s^2 } ds.
      \end{align*}
     This yields by using the Bessel identity
     $$ \| \mathcal{F}_{x} \partial_{x}^\alpha a \|_{L^2(\mathbb{Z}^d, L^\infty(\mathbb{S}_{+}))} \lesssim \|f^0\|_{\Hc^{|\alpha| + 3}_{\sigma}}$$
      and hence \eqref{symbolepreuve1} is proven.
       Let us turn to the proof of \eqref{symbolepreuve2}.
     To estimate $\partial_{x}^\alpha \partial_{\tilde \xi} a$,  we have to estimate the following two types
      of  symbols
      $$ 
       I_{1}^\alpha(x,  \tilde \zeta ) =  \int_{0}^{+ \infty} e^{-(\tilde \gamma + i \tilde  \tau)s}\,  e_{j} \cdot   \partial_{x}^\alpha  G(x,\tilde k s ) \, ds, \quad    I_{1}^\alpha (x, \tilde \zeta)= \int_{0}^{+ \infty} e^{-(\tilde \gamma + i \tilde  \tau)s}\,  \tilde k s \cdot   \partial_{x}^\alpha   \partial_{\eta}^{\beta_{1}}G(x, \tilde k s) \, ds $$
        where $e_{j}$ is a unit vector and $|\beta_{1}|\leq 1$.         For $I_1^\alpha$, if  $|\tilde k| \geq  { 1  \over 2}$, we can proceed in the same way with \eqref{estsymbol1}
 for $ q = 2$, $\beta = 0$ and  obtain
\begin{align*}
 |\mathcal{F}_{x}I_{1}^\alpha(l, \tilde \zeta) | &    \lesssim \left( \int_{\R^d}  
    ( 1 + |  v  |^2)^{\sigma} \,| (\mathcal{F}_{x} \nabla_{v} \partial_{x}^\alpha ( I - \Delta_{v}) f^0)(l, v ) |^2 dv \right)^{1 \over 2} \,  \int_{0}^{+ \infty}  { 1  \over 1 + |\tilde k|^2 s^2 } ds \\
    &    \lesssim \left( \int_{\R^d}  
    ( 1 + |  v  |^2)^{\sigma} \, |(\mathcal{F}_{x} \nabla_{v} \partial_{x}^\alpha ( I - \Delta_{v}) f^0)(l, v ) |^2 dv \right)^{1 \over 2}
    \end{align*}
     Note that for this argument, we use in a crucial way that $| \tilde k|$ is bounded from below.
      Otherwise since $\tilde \zeta \in S_{+}$, we have that $|\tilde \gamma|^2 + |\tilde {\tau}|^2 \geq {3 \over 4}$ and  consequently, 
      we can integrate by parts  in $s$ in the integral to obtain that
    $$ | \mathcal F_{x} I_{1}^\alpha (l ,\zeta) |  \lesssim   | (\mathcal{F}_{x} \partial_{x}^\alpha G)(l,0)| + \int_{0}^{+\infty}  | \tilde k | \, |  ( \mathcal{F}_{x}\partial_{x}^\alpha \nabla_{\eta} \ G)(l, \tilde ks)ds$$
     and hence, by using again \eqref{estsymbol1} with  $q=2$, and $|\beta| =1$, we finally obtain that
     $$   \| \mathcal F_{x }I_{1}^\alpha \|_{L^2(\mathbb{Z}^d,L^\infty (\mathbb{S}_{+}) )}   \lesssim \|f^0\|_{\Hc^{2m}_{\sigma+1}}, \quad | \alpha | \leq   2m - 3.$$
     To estimate $I_{2}^\alpha$, we proceed as above: if $|\tilde k| \geq  { 1  \over 2}$, we rely on  \eqref{estsymbol1}
         with $q = 3$ and $|\beta|\leq 1$, otherwise we use the same integration by parts argument together with \eqref{estsymbol1} with  $q=2,3$ and $|\beta| \leq 2$. We obtain
       $$ \| \mathcal{F}_{x}I_{2}^\alpha\|_{L^2(\mathbb{Z}^d, L^\infty(S_{+}))} \lesssim \|f^0\|_{\Hc^{2m}_{\sigma+2}}, \quad |\alpha | \leq  2m - 4.$$
 This ends the proof.

      \end{proof}   
     
     We can now use symbolic calculus to estimate  the solution of the integral equation \eqref{eqintegrale3}.
     \begin{proposition}
     \label{proppenrose}
       Consider $h$ the solution of   \eqref{eqintegrale3}, assume that $2m> 4 +{ d \over 2}$, that $2r>2+{ d \over 2}$ and that  for every $x \in \mathbb{T}^d$, the profile $f^0(x, \cdot)$ satisfies the $c_{0}$ Penrose stability criterion.  
       Then there exists 
       $  \Lambda[\|f^0 \|_{\Hc^{2m}_{2r}}]$ such that for every $\gamma \geq  \Lambda[\|f^0 \|_{\Hc^{2m}_{2r}}]$, we have the estimate
       $$ \| h\|_{L^2(\mathbb{R} \times \mathbb{T}^d)}  \lesssim \Lambda[\|f^0 \|_{\Hc^{2m}_{2r}}]  \|\mathcal{R}\|_{L^2(\mathbb{R} \times \mathbb{T}^d)}.$$     
     \end{proposition}
     
 \begin{proof}[Proof of Proposition~\ref{proppenrose}]
 By using Lemma \ref{lemKpseudo}, we can write \eqref{eqintegrale3} under the form
 $$ h = Op_{a}^\gamma ( (I - \eps^2 \Delta )^{-1}h) + \mathcal{R}= Op_{b^\eps}^\gamma (h) + \mathcal{R}$$
  with the symbol $b^\eps(x, \zeta)$ defined by 
  $$ b^\eps(x, \zeta)= a(x, \zeta) { 1 \over 1 + \eps^2 |k|^2}.$$
   Note that this is exact since we are composing a pseudodifferential operator with a Fourier multiplier in the right order.
    Since $a$ is homogeneous of degree zero in $\zeta$, we have
    $$ b^\eps(x,\zeta)= b(x, \eps \zeta), \quad b(x, \zeta)=  a(x, \zeta) { 1 \over 1 + |k|^2}$$
     and thus $Op_{b^\eps}^\gamma h =Op_{b}^{\eps, \gamma}$ is a semiclassical pseudodifferential operator
      as defined in the Appendix.
     We thus have to study the equation
  \begin{equation}
  \label{int2}  Op_{1-b}^{\eps, \gamma}(h)= \mathcal{R}.
  \end{equation}
   Thanks to Lemma \ref{lemsymbole}, we  have that  $b \in S_{2m-3, 0} \cap S_{2m-4, 1}$. 
    Moreover,  we observe that
     $$1 - b(x,\gamma, \tau, k)= \mathcal{P}(\gamma, \tau, k,  f^0(x, \cdot))$$
      and consequently, since   $f^0$ satisfies the $c_{0}$ Penrose condition \eqref{Penrose}, 
       we also get that $c= { 1 \over 1 -b}  \in   S_{2m-3, 0} \cap S_{2m-4, 1}.$
     Consequently assuming that $2m> 4+{d\over 2}$, we can find $M>d/2$ such that
     $ c\in S_{M,1}$ and $ 1-b \in S_{M+1, 0}$ and moreover, 
     $$ |c|_{M, 1} + | 1- b|_{M+1, 0} \lesssim \Lambda[\|f^0 \|_{\Hc^{2m}_{2r}}].$$ 
      Consequently, by  applying $Op_{c}^{\eps, \gamma}$ to \eqref{int2} and by using Proposition \ref{propsemi}, we obtain
        that 
        $$ \|h\|_{L^2(\mathbb{R} \times \mathbb{T}^d)} \lesssim   { 1 \over \gamma}\Lambda[\|f^0 \|_{\Hc^{2m}_{n_{0}}}] \|h\|_{L^2(\mathbb{R} \times \mathbb{T}^d)} + \Lambda[\|f^0 \|_{\Hc^{2m}_{n_{0}}}]  \|\mathcal{R}\|_{L^2(\mathbb{R} \times \mathbb{T}^d)}.$$
         The result follows by choosing $\gamma$ sufficiently large. 
   \end{proof}
   As a Corollary, we  get an estimate  for the solution of \eqref{eqintegrale2} on $[0, T]$.
   \begin{corollary}
   \label{coroPen}
    Consider $\tilde h$ the solution of   \eqref{eqintegrale3}, assume that $2m>4+{ d \over 2}$,  that $2r>2+{ d \over 2}$ and that the $c_{0}$ Penrose criterion
     \eqref{Penrose} is satisfied. Then there exists $\Lambda(\cdot,\cdot)$ such that the solution of
    \eqref{eqintegrale2} verifies the estimate
    $$ \| \tilde h \|_{L^2([0, T],  L^2_{x})} \leq  \Lambda[T, \|f^0 \|_{\Hc^{2m}_{2r}}] \|\tilde R\|_{L^2([0, T], L^2_{x})}.$$
  \end{corollary}
  Note that  the assumption that $2m>4+ {d \over 2}$ is  satisfied if $2m>m_{0}$.
 \begin{proof}[Proof of Corollary~\ref{coroPen}]
  By using \eqref{expweight} and Proposition \ref{proppenrose}, we get that
  $$  \Big(\int_{0}^T  e^{- 2 \gamma t} \| \tilde h(t, \cdot)\|_{L^2(\mathbb{T}^d)}^2  dt \Big)^{1 \over 2}\lesssim   \| h \|_{L^2(\mathbb{R} \times \mathbb{T}^d)}
   \lesssim  \Lambda[ \|f^0 \|_{\Hc^{2m}_{2r}}] \|\mathcal{R} \|_{L^2(\mathbb{R} \times \mathbb{T}^d)}.$$
    Since we have taken $\mathcal{R}$ to be zero for $t \geq T$ and $t <0$, we get by using again \eqref{expweight} that
   $$    \Big(\int_{0}^T  e^{- 2 \gamma t} \| \tilde h(t, \cdot)\|_{L^2(\mathbb{T}^d)}^2  dt \Big)^{1 \over 2}\lesssim   \Lambda[ \|f^0 \|_{\Hc^{2m}_{2r}}] 
    \Big(\int_{0}^T  e^{- 2 \gamma t} \| \tilde R (t, \cdot)\|_{L^2(\mathbb{T}^d)}^2  dt \Big)^{1 \over 2}.$$
     Since  $\gamma$ was chosen as  $\gamma= \Lambda[ \|f^0 \|_{\Hc^{2m}_{2r}}]$, the result follows by taking
     $$
      \Lambda[T, \|f^0 \|_{\Hc^{2m}_{2r}}] := \Lambda[ \|f^0 \|_{\Hc^{2m}_{2r}}] \exp\left( \Lambda[ \|f^0 \|_{\Hc^{2m}_{2r}}] T \right).
     $$ 
 \end{proof}

    \section{Proof of Theorem~\ref{theomain}: conclusion}
    \label{secconclusion}
  We are finally ready to close the bootstrap argument. For $2m>m_{0}$, $2r>r_{0}$, 
  gathering the results of Lemma~\ref{redu}, Lemma~\ref{lem-simplification} and Corollary~\ref{coroPen}, 
  we get that for all $T \in [0, \min (T_{0},  \tilde T_{0}, \hat T_{0},  T^\eps))$, for all $i,j \in \{1,\cdots , d\}^m$,
  $$
  \| \pa_x^{\alpha(i,j)} \rho \|_{L^2([0,T], L^2_x)}  \leq \Lambda(T,M_0)\left( M_{0} +   T^{1\over 2} \Lambda (T, R)\right). 
  $$
  Using Lemma~\ref{lemfacilesans}, we deduce  that
  \begin{align*}
      \Nc_{2m, 2r}(T, f) &\leq M_{0}+   T^{1\over 2} \Lambda (T, R) + \Lambda(T, M_{0}) \left( M_{0} +   T^{1\over 2} \Lambda (T, R)\right).
  \end{align*}
  We choose $R$ large enough so that
  \begin{equation}
  \frac{1}{2} R > M_{0}  + \Lambda[0,M_{0} ] M_0.
  \end{equation}
  Now, $R$ being fixed, we can choose by continuity  $T^\# \in (0,\min (T_{0},  \tilde T_{0}, \hat T_{0}, T^\eps)]$  such that for all $T \in [0,T^\#]$,
    \begin{equation}
     T^{1\over 2} \Lambda (T, R)
     + \Lambda[T, M_{0}]T^{1\over 2} \Lambda (T, R) + (\Lambda[T,M_{0} ] -\Lambda[0,M_{0} ] )M_0<   \frac{1}{2} R.
    \end{equation}    
    This yields  that for all $T \in [0,T^\#]$, it is impossible to have $  \Nc_{2m, 2r}(T, f) =R$.
    Therefore, we deduce that $T^\eps>T^\#$. We have thus proven that
     \begin{equation}
     \label{conclusion}
     \Nc_{2m,\, 2r}(T, f) \leq R,
    \end{equation}
     for some $T>0$ and some $R>0$, both independent of $\eps.$
     To finish the proof of 
      Theorem~\ref{theomain}, it remains to check that the $c_{0}/2$ Penrose stability condition can be ensured.
       From the equation \eqref{VP} and~\eqref{conclusion}, we get that 
       $$
       \| \partial_t f \|_{L^\infty([0,T], \Hc^{2m-2}_{2r-1})} \leq \Lambda(T,R).
           $$
           By using a Taylor expansion, we have that  for all $t \in [0,T]$,
        \begin{multline*}
           \int_{0}^{+ \infty} e^{-(\gamma + i \tau)s}\, {i \eta \over 1 + |\eta|^2}\cdot  ( \mathcal{F}_{v} \nabla_{v}  f )(t,\eta s ) \, d s
           =            \int_{0}^{+ \infty} e^{-(\gamma + i \tau)s}\, {i \eta \over 1 + |\eta|^2}\cdot  ( \mathcal{F}_{v} \nabla_{v}  f^0_\eps)(\eta s ) \, d s  \\+ I
           (\gamma, \tau, \eta, t,x)
           \end{multline*}
           where $I(\gamma, \tau, \eta, x)$ satisfies the uniform estimate
         $$|I(\gamma, \tau, \eta, t,x)| \leq C T\sup_{t \in [0, T]}\| \partial_{t} f\|_{\Hc^2_{\sigma}}$$
          with $\sigma >d/2$.
         Since we have $2m\geq 4$ (by the assumption $2m>m_{0}$ ) and $2r> 1 + {d\over 2}$ this yields
         $$ |I(\gamma, \tau, \eta, t,x)| \leq T 
           \Lambda(T,R).
             $$
Since $f^0_\eps$ satisfies the $c_0$ Penrose stability condition, it follows by  taking a smaller time $T>0$ if necessary, that  for all $t \in [0,T]$ and all $x \in \T^d$, $f(t,x,\cdot)$ satisfies the $c_0/2$ Penrose condition.  
                     
    \section{Proofs of Theorems~\ref{theoquasi} and~\ref{WP}}
    \label{sec23}
     The proofs of Theorems~\ref{theoquasi} and~\ref{WP} will be intertwined  since in order to get the convergence of Theorem \ref{theoquasi}
      without extracting a subsequence, we shall need  the uniqueness part of Theorem \ref{WP}.
        Consequently, we shall first prove the uniqueness part of Theorem \ref{WP}.
         The result will actually be a straightforward consequence of the following  Proposition.
     \begin{proposition}
     \label{proplin}
      We consider the following linear equation:
      \begin{equation}
      \label{eqlinunique}
      \partial_{t}f + v \cdot \nabla_{x} f - \nabla_{x} \rho \cdot \nabla_{v}\overline f + E(t,x) \cdot \nabla_{v} f = F, \quad f_{| t={0}}= f^{0},
       \quad \rho(t,x) = \int_{\mathbb{R}^d} f(t,x,v)\, dv 
      \end{equation}
      where  $E(t,x)$  is a  {\emph{given}} vector field such that for some $T_{0}>0$,  $E \in L^2((0,T_{0}), H^{2m})$ and $\overline f$ is a  \emph{given } function $\overline f(t,x,v) \in L^\infty (0, T_{0}], \Hc^{2m-1}_{2r}) \cap Lip([0, T_{0}], \Hc^{2m-2}_{r- 1})$
       with $2m>m_{0} $, $2r> r_{0}$.
        Let us set
        $$ R:= \|\overline f\|_{L^\infty([0, T_{0}], \Hc^{2m-1}_{2r})} + \|\partial_{t}\overline f \|_{{L^\infty}([0, T_{0}], \Hc^{2m-2}_{2r-1})}
         + \|E\|_{L^2([0, T_{0}], H^{2m- 1})}$$
          and assume that for every $x$, the profile  $\overline f(0, x, \cdot)$ satisfies the $c_{0}$ Penrose stability condition
           for some $c_{0}>0$. Then, there exists $ T=  T(c_{0}, R) \in (0, T_{0}]$ that depends only on $c_{0}$ and $R$ 
            such that  for every $F \in L^2([0, T_{0}], \Hc^0_{r})$  and $g_{0} \in \Hc^0_{r}$, the solution
             $g$ of \eqref{eqlinunique} satisfies the estimate
       \begin{equation}
       \label{estlinunique}
         \| \rho \|_{L^2([0, T] \times \mathbb{T}^d)} \leq \Lambda({ 1 \over c_{0}}, R, T_{0}) \big( \|f^0\|_{\Hc^0_{r}} +
           \|F\|_{L^2([0, T_{0}], \Hc^{0}_{r})}  )
       \end{equation}
       where $\Lambda({ 1 \over c_{0}}, R, T)$ depends only on $c_{0}$, $R$  and $T_{0}$.
     \end{proposition}
     \begin{proof}[Proof of Proposition~\ref{proplin}]
      The proof follows closely the analysis of equation \eqref{eq-fij} in the proof  of Theorem \ref{theomain} so that we shall
       only give the main steps.
        We first  set $g(t, x, \Phi(t,x,v))= f(t,x,v)$ with $\Phi$ being the solution of the Burgers equation \eqref{eq-Burgers}
         with initial data $\Phi(0,x,v)= v$, recall Lemma~\ref{lemburgers}. Because of the regularity assumptions on $E$, Lemma \ref{lemburgers}
          is still valid: such a smooth $\Phi$ exists on $[0, T(R)]$ for some $T(R)>0$ and verifies the estimates  \eqref{estim-Burgers-infini},  \eqref{estim-Burgers}. We observe that $g$ solves
        \begin{equation}
        \label{equniqueg} \partial_{t} g + \Phi \cdot \nabla_{x} g - \nabla_{x} \rho(t,x) \cdot \nabla_{v} \overline f (t,x, \Phi) = F(t,x, \Phi)
        \end{equation}
             and that
           $$   \int_{\mathbb{R}^d}  g(t,x,v) J(t,x,v) \, dv= \rho(t,x)$$
            with $J(t,x,v ) = | \det \nabla_{v} \Phi(t,x,v)|$.
            To  solve \eqref{equniqueg}, we use the characteristics \eqref{characteristic}. Because of the previous estimates
             on $\Phi$, the estimates of Lemma \ref{redressement2} are still valid.
              Proceeding as in the proof of Lemma \ref{redu}, we can first obtain that
           \begin{align*}
            \rho(t,x) =  K_{H} \rho 
            &+ \int_{0}^t \int_{\mathbb{R}^d} F(s, X(s,t, x,v), \Phi(s,x,v)) J(t,x,v) \, dvds \\
            &+ \int_{\mathbb{R}^d} f^0(X(0, t, x, v),v) J(t,x,v) \, dv
            \end{align*}
             where
            $$  H  (t,s,x,v)=   ( \nabla_{v}\overline f)(s,  x- (t-s)v,\Psi(t,s,x,v)) J(t,x,\Psi(s,t,x,v)) \tilde J(t,s,x,v)$$ 
     and $\tilde{J}(s,t,x,v)= | \det \nabla_{v}\Psi(s,t,x,v) |$.
      Again, by Taylor expanding $H$ in time and by  using  Proposition \ref{propint} and remark \ref{remarknormekernel} we obtain that
     $$ \rho(t,x)= K_{\nabla_{v} \overline f^0} \rho +\tilde{\mathcal{R}}$$ 
     with the notation $\overline{f}^0(x,v)= \overline{f}(0, x,v)$ and
     where $\tilde{\mathcal{R}}$ is such that 
     \begin{equation}   
      \label{Runique} \| \tilde{\mathcal{R}} \|_{L^2([0, T], L^2)} \leq  \Lambda(T, R) ( \|f^0\|_{\Hc^0_{r}} + \|F\|_{\Hc^0_{r}} + T^{1 \over 2} \|\rho\|_{L^2([0, T], L^2})\end{equation}
       for every  $T \in [0, T(R)].$
  In order to estimate the solution of  the previous equation,  we can again set $\rho= e^{\gamma t} h$, 
   $\tilde{ \mathcal{R}}= e^{\gamma t } \mathcal{R},$   assume that $h$ 
   and $\mathcal{R}$ are zero for $t<0$  and that $\mathcal{R}$ is continued by zero for $t>T$. Then by using  lemma
    \ref{lemKpseudo}, we end up with the equation
    $$ h= Op_{a}^\gamma h + \mathcal{R}$$
    where $a$ is still defined by \eqref{adef} with $f^0$ replaced by $\overline{f}^0$. Because of the regularity assumptions on $\overline f$, 
     the estimates of Lemma \ref{lemsymbole} are still verified. 
        If $ a$  has the property  that
       \begin{equation}
       \label{Penrosebis} |1- a(x, \zeta)| \geq c_{0}, \quad \forall \zeta= (\gamma, \tau, k),\,  \quad \gamma>0, \, \tau \in \mathbb{R}, \, 
         k \in \mathbb{R}^d\backslash\{0\},
         \end{equation}
          then we can apply the operator $Op^\gamma_{{ 1 \over 1 - a}}$ and  use Proposition \ref{propcontinu} and Proposition \ref{propcalculus} to get that for $\gamma$ sufficiently large, we have
          $$ \|h\|_{L^2_{t,x}} \leq \Lambda( {1 \over c_{0}}, R) \| \mathcal{R}\|_{L^2_{t,x}}.$$
           In view of \eqref{Runique}, this yields
            that for every $ T \in [0, T(R)]$ we have
            $$ \| \rho \|_{L^2([0, T], L^2)} \leq \Lambda({ 1 \over c_{0}}, T, R) \bigl( 
           \|f^0\|_{\Hc^0_{r}} + \|F\|_{\Hc^0_{r}} + T^{1 \over 2} \|\rho\|_{L^2([0, T],L^2)} \bigr).$$
            Consequently, if $T$ is sufficiently small we get the estimate \eqref{estlinunique}.
            
            In order to finish the proof, we thus only have to check that the estimate \eqref{Penrosebis}
             is verified.  Let us recall that by definition of  the Penrose stability condition, we have  that for every $x \in \mathbb{T}^d$,
              the function
  $$ \mathcal{P}(\gamma, \tau, \eta,  \overline f^0(x, \cdot))= 1 -  \int_{0}^{+ \infty} e^{-(\gamma + i \tau)s}\, \frac{i \eta}{1+  |\eta|^2} \cdot  ( \mathcal{F}_{v} \nabla_{v}  \overline{f}^0)(x, \eta s ) \, d s, \quad \gamma >0,\, \tau \in \mathbb{R}, \,\eta \in \mathbb{R}^d\backslash\{0\}$$
  verifies
 $$ \inf_{(\gamma, \tau, \eta) \in [0, + \infty) \times \mathbb{R} \times \mathbb{R}^d} |\mathcal{P}(\gamma, \tau, \eta, \overline f^0(x, \cdot))|\geq c_{0}.$$
  Let us then define using polar coordinates the function $\tilde{ \mathcal  P}$ by
  $$ \tilde{ \mathcal P}(\tilde \gamma, \tilde \tau, \tilde \eta, \sigma, \overline f^0(x,\cdot))= 
    \mathcal{P}( \sigma\tilde \gamma, \sigma \tilde  \tau,  \sigma \tilde  \eta, \overline{f}^0(x,\cdot)), \quad (\gamma, \tau, \eta)= \sigma(\tilde \gamma, \tilde \tau, \tilde \eta), \, r>0, 
     \, \tilde \gamma >0, \, (\tilde \gamma, \tilde \tau, \tilde \eta) \in S_{+}$$
     where $S_+ = \{(\tilde \gamma, \tilde \tau, \tilde \eta), \, \tilde \gamma^2 + \tilde \tau^2 +  \tilde \eta^2= 1, \, \tilde \gamma > 0, 
      \tilde \eta \neq 0 \}$.
      Note that we have
      $$   \tilde{ \mathcal P}(\tilde \gamma, \tilde \tau, \tilde \eta, \sigma,  \overline f^0)= 
      1 -  \int_{0}^{+ \infty} e^{-(\tilde \gamma + i \tilde  \tau)s}\, \frac{i  \tilde \eta}{1+  \sigma^2 |\tilde \eta|^2} \cdot  ( \mathcal{F}_{v} \nabla_{v}  \overline f^0)(x, \tilde \eta s ) \, d s.
 $$    
 If $\overline f^{0} \in \Hc_{r}^2$, the function  
 $\tilde{\mathcal P}$ can be extended as a continuous function on $  S_{+} \times [0+\infty[$.
  The   Penrose  stability condition thus implies $\tilde{ \mathcal  P} \geq c_{0}$   on   $S_{+} \times [0, + \infty[$. 
   In particular for $\sigma=0$, we observe that
   $$ \tilde{ \mathcal P}(\tilde \gamma, \tilde \tau, \tilde \eta, 0, \mathbf f)=  1 -a(\tilde \gamma, \tilde \tau, \tilde \eta).$$
     We thus obtain that  $|1- a| \geq c_{0}$ on $S_{+}$. Since $a$ is homogeneous of degree zero, this yields
      that \eqref{Penrosebis} is verified. This ends the proof.
    \end{proof}  
      As an immediate corollary of the previous proposition, we get an uniqueness property for the limit equation
       \eqref{formel}.
     \begin{corollary}
  \label{coruniq}
  Let $f_1, f_2 \in \mathcal{C}([0, T], \Hc^{2m-1}_{2r})$ with  $2m>m_{0}$, $2r>r_{0}$   be two solutions of \eqref{formel} with the same initial condition $f^0$. 
  Setting $\rho_i := \int f_i \, dv$, we assume that  $ \rho_1, \rho_2 \in L^2([0, T], H^{2m})$. 
  Assume that furthermore, there is $c_0>0$ such that $f_1$ is such that $v \mapsto f_1(t,x,v)$ satisfies  the $c_{0}$ Penrose condition  for every $t\in [0, T]$ and $x \in \mathbb{T}^d$. Then  we have that $f_1 =f_2$ on $ [0, T] \times \mathbb{T}^d \times \mathbb{R}^d$.
  \end{corollary}

  \begin{proof}[Proof of Corollary~\ref{coruniq}]
  Let 
  $$R= \max_{i=1,2} \left(\| f_i \|_{L^\infty([0, T], \Hc^{2m-1}_{2r})} + \| \rho_i \|_{L^2([0, T], H^{2m})} \right).$$
  
  We set $f=f_1-f_2$, and observe that $f$ satisfies the equation
  \begin{equation}
  \label{eqdiff}
  \partial_t f + v \cdot \na_x f - \na_x \rho  \cdot \na_v f_{1}  - \na_x \rho_2 \cdot \na_v f  =0, \quad f|_{t=0}=0,
  \end{equation}
  where $\rho  := \int f\, dv$. 
   We are thus in the framework of Proposition \ref{proplin} with $ E= - \nabla_{x} \rho_{2}$, 
    $\overline{f}= f_{1}$ and zero data (that it say $F= 0$ and zero initial data).
     Moreover, we observe that  thanks to the equation \eqref{formel}, we also have that
     $$  \| \partial_{t} f_{i} \|_{L^\infty([0, T], \Hc^{2m-2}_{2r- 1})} \leq \Lambda(T,R).$$
      We are thus in the framework of Proposition \ref{proplin}. From \eqref{estlinunique}, we deduce that there exists $T({c_{0} }, R)$
       such that $\rho= 0$ in $[0, T({c_{0} },R)]$. This yields that on $[0, T({c_{0} }, R)]$, $f$ satisfies
        the homogeneous transport equation
         $$    \partial_t f + v \cdot \na_x f   - \na_x \rho_2 \cdot \na_v f  =0$$
         with zero initial data and thus $f= 0$ on $[0, T({c_{0} }, R)]$.
         We can then apply again  Proposition \ref{proplin} starting from $T({c_{0} }, R)$ (which is valid since $f_{1}(T({c_{0} }, R), \cdot)$
         still   satisfies the $c_{0}$ Penrose stability condition). 
          Since the estimate \eqref{estlinunique}
          is valid on an interval of time that depends only on $R$ and $c_{0}$, we then obtain that
           $f= 0$ on $[0, 2T({c_{0} }, R)]$. Repeating the argument, we finally obtain after a finite number of steps that
            $f= 0$ on $[0, T]$. This ends the proof.

     \end{proof}

    \subsection{Proof of Theorem~\ref{theoquasi} }
%     We first observe that since $f^0$ satisfies  the strong Penrose condition \eqref{Penrose}
%     and that $f^0_{\e} \to f^0$ in $\Hc^{2m}_{2n_{0}}$, there is $\eps_0>0$ such that for all $\eps \in (0,\eps_0]$,
%     $f^0_{\e}$ satisfies the uniform Penrose condition \eqref{Penrose}, in the sense that
%      \begin{equation}
% \sup_{\eps \in (0,\eps_0]} \sup_{(x,\gamma, \tau, k) \in \T^d \times [1,+\infty) \times \R \times \Z^d\setminus\{0\} } \left| 1 - \int_{0}^{+ \infty} e^{-(\gamma + i \tau)s}\, \frac{i k}{1+ \eps^2 |k|^2} \cdot  ( \mathcal{F}_{v} \nabla_{v}  f^0_\eps)(ks ) \, d s\right| >0.
% \end{equation}
     We start by applying Theorem~\ref{theomain} to get $T,R>0$ independent of $\eps$, such that $f_\eps \in \mathcal{C}([0, T], \Hc^{2m}_{2r})$  satisfies \eqref{VP} with
 $$\sup_{\eps \in (0, \eps_0]} \Nc_{2m, 2r}(T,f_\eps) \leq R.$$

We can now use standard compactness arguments to justify the quasineutral limit:  $f_\eps$ is uniformy bounded in $\mathcal{C}([0,T],  \Hc^{2m-1}_{2n_{0}})$ and from  \eqref{VP}, we get that $\partial_t f_\eps$ is uniformly bounded in 
$L^\infty([0,T],  \Hc^{2m-2}_{2r-1})$. 
 Consequently, by Ascoli Theorem there exists $f \in \mathcal{C}([0, T], L^2)$ and a sequence $\eps_{n}$
  such that $f_{\eps_{n}}$ converges to $f$ in   $\mathcal{C}([0, T], L^2_{x,v})$. By interpolation, we also actually
   have convergence in $ \mathcal{C}([0, T], \Hc^{2m-1- \delta}_{2r- \delta})$ for every $\delta >0$.
    By Sobolev embedding  this yields in particular that $f_{\eps_{n}}$ converges to $f$ in $L^\infty([0, T] \times \mathbb{T}^d
     \times \mathbb{R}^d)$ and that $\rho_{\eps}$ converges to $\rho= \int_{\R^d} f\, dv$
      in $L^2([0, T], L^2) \cap L^\infty([0, T]\times \mathbb{T}^d).$
       From these strong convergences, we easily obtain that $f$ is solution of \eqref{formel}
      and that $f$ satisfies the $c_{0}/2$ Penrose stability condition on $[0, T]$.
     Moreover, by standard weak-compactness arguments, we also easily obtain that
     $ f\in L^\infty([0, T], \Hc^{2m-1}_{2r}) \cap \mathcal{C}_{w}([0, T], \Hc^{2m-1}_{2r}) $
      (that is to say continuous in time with    $\Hc^{2m-1}_{2r}$ equipped with the weak topology) and that $\rho \in L^2([0, T], H^{2m})$.  With this regularity of $\rho$, we can then deduce  by standard arguments from the energy estimate for \eqref{formel}
      (which is just \eqref{energieformel} with $E= -  \nabla_{x} \rho$) that we actually infer that 
      $ f\in \mathcal{C}([0, T], \Hc^{2m-1}_{2r})$. 
            
      Thanks to the uniqueness for \eqref{formel} proved in Corollary \ref{coruniq}, we can get by standard arguments
       that we actually have the full convergence of $f_{\eps}$ to $f$ and not only the subsequence $f_{\eps_{n}}$.
       
         \subsection{Proof of Theorem~\ref{WP}}
     With the choice $f^0_\eps = f_0$ for all $\eps \in (0,1]$, Theorem~\ref{theoquasi}  provides 
      the existence part. The uniqueness is a consequence of Corollary~\ref{coruniq} and the fact that $f$ satisfies the $c_{0}/2$  Penrose stability condition \eqref{Penrose}
   for every $t\in [0, T]$ and $x \in \mathbb{T}^d$.

  \section{Pseudodifferential calculus with parameter}
  \label{sectionpseudo}
   In this section we shall prove the basic results about pseudodifferential calculus that we need in our proof. For more complete statements and results, 
   we refer for example to \cite{Metivier}, \cite{Metivier-Zumbrun}.
     We consider symbols $a(x, \gamma, \tau, k)$  on $\mathbb{T}^d \times ]0,+ \infty[ \times \mathbb{R} \times \mathbb{R}^d\backslash\{0\}$, $\gamma >0 $ has to be thought to as a parameter.
      We set  $\zeta= (\gamma, \tau, k)$ and $\xi= (\tau, k) \in \mathbb{R} \times \mathbb{R}^d\backslash\{0\} $.
   Note that we do not need to include a  dependence on the time variable $t$ in our symbols (so that we actually consider Fourier multipliers
    in  the time variable).
     We use the quantification
     $$ (Op_{a}^\gamma) u(t,x)=  (2 \pi)^{- ( d+ 1)} \int_{\mathbb{Z}^d \times \mathbb{R}} e^{i (\tau t + k \cdot x)} a(x, \zeta) \hat{u}(\xi) \, d\xi$$
      where $d \xi = dk d \tau$ and the measure on $\mathbb{Z}^d$ is the discrete measure. The Fourier transform $\hat u$ is defined as
      $$ \hat u(\xi)= \int_{\mathbb{T}^d \times \mathbb{R}} e^{-i( \tau t +  k \cdot x)} u(t,x) dt dx.$$
     We introduce the following seminorms of symbols:
     \begin{align}
   \label{pseudo0}  &  |a|_{M,0}=   \sup_{| \alpha | \leq  M } \| \mathcal{F}_{x}(\partial_{x}^\alpha a) \|_{L^2(\mathbb{Z}^d, L^\infty_{\zeta})}, \\
   \label{pseudo1}  &  |a|_{M,1}=  \sup_{| \alpha | \leq  M } \|  \langle \zeta \rangle  \mathcal{F}_{x} (\partial_{x}^\alpha \nabla_{\xi} a)\|_{L^2(\mathbb{Z}^d, 
      L^\infty_{\zeta})}.
     \end{align}
     where
     $$\langle \zeta \rangle  = (\gamma^2 + \tau^2 + |k|^2)^{1 \over 2}.$$
      We shall say that $a\in S_{M,0}$ if $|a|_{M,0}<+\infty$ and $a \in S_{M,1}$ if  $ |a|_{M,1}<+\infty.$
       The use of these seminorms compared to some more classical  ones   will allow us to avoid to lose too many derivatives
       while keeping very simple proofs.
       
        Note that we can  easily relate  $ |a|_{M,0}$ to  more classical symbol  seminorms up to loosing more derivatives.
         For example, we have for every $M \geq 0$
%         $$\sup_{| \alpha | \leq  M } \sup_{ \zeta } \| \partial_{x}^\alpha a(\cdot , \zeta) \|_{L^2(\mathbb{T}^d)} \lesssim
%            | a |_{M+s}, \quad  \sup_{| \alpha | \leq  M } \sup_{x,  \,  \zeta} |a(x, \zeta)|
%             \lesssim  | a|_{M+ 2s }$$
             $$  \sup_{| \alpha | \leq  M } \sup_{x,  \,  \zeta} |\partial_x^\alpha a(x, \zeta)|
             \lesssim  | a|_{M+ s, 0}$$
              with $s>d/2$.
              The following results refine slightly in terms of the regularity of the symbols, the classical results of $L^2$
               continuity for symbols in $S^0_{0,0}$ that are compactly supported in $x$, see for example \cite{Taylor}. 
  %   We start with a simple $L^2$ continuity result. 
     \begin{proposition}
     \label{propcontinu}
      Assume that $M> d/2$ and  that $a\in S_{M, 0}$.
      Then, there exists $C>0$ such that for every $\gamma > 0$
      $$ \|Op_{a}^\gamma u\|_{L^2(\mathbb{R} \times \mathbb{T}^d)} \leq C |a|_{M,0} \|u\|_{L^2(\mathbb{R} \times \mathbb{T}^d)}.$$
     \end{proposition}
    \begin{proof}[Proof of Proposition~\ref{propcontinu}]
    In the following $\lesssim$ means $\leq C$ with $C$ that does not depend on $\gamma$.
     We can write
     \begin{align*} 
     Op_{a}^\gamma  u(t,x) & =   (2 \pi)^{- ( d+ 1)}\int_{\mathbb{Z}^d}  e^{i x \cdot k'} \Big(  \int_{\mathbb{Z}^d \times \mathbb{R}} e^{i (\tau t + k \cdot x)} \mathcal{F}_{x}a(k', \zeta) \hat{u}(\xi) \, d\xi \Big) d k'
     \\
      & =  (2 \pi)^{- ( d+ 1)}  \int_{\mathbb{Z}^d \times \mathbb{R}} e^{i (\tau t + l \cdot x)} \Big(
      \int_{\mathbb{Z}^d} \mathcal{F}_{x}a(k- l, \gamma, \tau, k) \hat{u}(\tau, k) dk  \Big) d\tau dl
      \end{align*}
      and hence we obtain from the Bessel identity that  
     $$ \|Op_{a}^\gamma u \|_{L^2(\mathbb{R} \times \mathbb{T}^d)} \lesssim 
     \left\|  \left\| \int_{\mathbb{Z}^d} \mathcal{F}_{x}a(k- \cdot , \gamma, \tau, k) \hat{u}(\tau, k) dk \right\|_{L^2(\mathbb{Z}^d)} \right\|_{L^2(\mathbb{R}_{\tau})}.$$
      By using Cauchy-Schwarz and Fubini, we get in a classical way that 
    \begin{align*}
    &   \left\| \int_{\mathbb{Z}^d} \mathcal{F}_{x}a(k- \cdot , \gamma, \tau, k) \hat{u}(\tau, k) dk \right\|_{L^2(\mathbb{Z}^d)}^2 
       \\
       & \lesssim   \| \sup_{k}| \mathcal{F}_{x}a(\cdot, \gamma, \tau, k)|\,\|_{L^1(\mathbb{Z}^d)}
        \int_{\mathbb{Z}^d \times \mathbb{Z}^d}  | \mathcal{F}_{x}a(k-l  , \gamma, \tau, k)|\,  |\hat{u}(\tau, k)|^2 dk dl \\
         & \lesssim    \|  \sup_{k} |\mathcal{F}_{x}a(\cdot, \gamma, \tau, k)|\,\|_{L^1(\mathbb{Z}^d)}
          \sup_{k} \|\mathcal{F}_{x}a(\cdot, \gamma, \tau, k)\|_{L^1(\mathbb{Z}^d)} \|\hat u(\tau, \cdot )\|_{L^2(\mathbb{Z}^d)}^2 \\
       & \lesssim  \| \sup_{k}| \mathcal{F}_{x}a(\cdot, \gamma, \tau, k)|\,\|_{L^1(\mathbb{Z}^d)}^2
        \|\hat u(\tau, \cdot) \|_{L^2(\mathbb{Z}^d)}^2.
        \end{align*}
        By integrating in time, we thus obtain that
      $$  \|Op_{a}^\gamma u \|_{L^2(\mathbb{R} \times \mathbb{T}^d)} \lesssim   \|\mathcal{F}_{x} a  \|_{L^1(\mathbb{Z}^d, L^\infty_{\zeta})}
       \|u\|_{L^2(\mathbb{R} \times \mathbb{T}^d)}.$$
        To conclude, it suffices to notice that
        $$  \|\mathcal{F}_{x} a  \|_{L^1(\mathbb{Z}^d, L^\infty_{\zeta})} \lesssim |a|_{M,0}$$
         for $M>d/2$.

    \end{proof}

%     Note that the above estimate is especially useful for large $\gamma$, the right hand side can be made small
  %     by taking $\gamma$ sufficiently large.

    We shall now state a result of symbolic calculus.
  \begin{proposition}
  Assume that $ a \in S_{M,1}$ and that   $b\in S_{M+1, 0}$ with $M>d/2$. Then there exists $C>0$ such that for every $\gamma>0$, we have
  \label{propcalculus}
   $$ \|Op_{a}^\gamma Op_{b}^\gamma (u) - Op_{ab}^\gamma (u)\| \leq {C \over \gamma} |a|_{M, 1} |b|_{M+1, 0}\, \|u \|_{L^2(\mathbb{R} \times \mathbb{T}^d)}.$$
  \end{proposition}
  Note that the above estimate is especially useful for large $\gamma$ since  the right hand side can be made small
       by taking $\gamma$ sufficiently large.

  \begin{proof}[Proof of Proposition~\ref{propcalculus}]
   Note that  for $a \in S_{M, 0}$, $b\in S_{M,0}$ and $M>d/2$, we have  by elementary convolution estimates that 
   $$ | ab |_{M,0}
    \lesssim  |a|_{M, 0} \|\mathcal{F}_{x}b\|_{L^1(\mathbb{Z}^d, L^\infty_{\zeta})} + |b|_{M, 0} \|\mathcal{F}_{x}a\|_{L^1(\mathbb{Z}^d, L^\infty_{\zeta})} \lesssim |a|_{M, 0}|b|_{M, 0}$$
         and  thus  that $ ab \in S_{M,0}$. This yields that $Op^\gamma_{ab}$ is a well-defined continuous operator on $L^2$
      thanks to Proposition \ref{propcontinu}.
       Next, using the usual formulas for pseudodifferential operators, we find that
       $$ Op_{a}^\gamma Op_{b}^\gamma = Op_{c}^\gamma$$
        with $c$ given by
       $$ c(x, \zeta)= \int_{\mathbb{Z}^d}  e^{i k' \cdot x} a(x,\gamma,  \tau, k+ k') \mathcal{F}_{x}b(k', \zeta) \, dk', \quad \zeta= (\gamma, \tau, k).$$
        We thus get that
       $$ c(x, \zeta) - a(x, \zeta) b(x, \zeta)=  \int_{\mathbb{Z}^d}  e^{i k' \cdot x} \int_{0}^1 \nabla_{k} a(x, \gamma, \tau, k+ sk')\, ds\cdot k' \,   \mathcal{F}_{x}b(k', \zeta) \, dk': = {1 \over \gamma}d(x, \zeta).$$
        By using Proposition \ref{propcontinu}, we can just  prove that $ d \in S_{M,0}$ for $M>d/2$ and estimate its norm.
           By taking the Fourier transform in $x$, we obtain that
        $$ (\mathcal F_{x} \partial_{x}^\alpha d) (l, \gamma, \tau, k)
        =  \gamma \int_{0}^1 \int_{\mathbb{Z}^d} (il)^\alpha ( \mathcal{F}_{x} \nabla_{k} a)(l-k', \gamma, \tau, k+ sk') \cdot k' \mathcal{F}_{x} b(k', \zeta)\, dk' ds.$$
         This yields
\begin{multline*} \|(\mathcal F_{x} \partial_{x}^\alpha d) (l, \gamma,\cdot)\|_{L^\infty_{\zeta}} 
  \lesssim   \int_{\mathbb{Z}^d} |l-k'|^{|\alpha|}   \left\| \,|(\gamma, \cdot)| ( \mathcal{F}_{x} \nabla_{k} a)(l-k', \gamma, \cdot) \right\|_{L^\infty_{\zeta}} 
   |k'| \left\|  \mathcal{F}_{x} b(k', \gamma, \cdot)\right\|_{L^\infty_{\zeta}}  \, dk'
    \\+  \int_{\mathbb{Z}^d}    \left\|\, |(\gamma, \cdot)|( \mathcal{F}_{x} \nabla_{k} a)(l-k', \gamma, \cdot)\right \|_{L^\infty_{\zeta}} \,
   |k'|^{| \alpha |+ 1}\left\|  \mathcal{F}_{x} b(k', \gamma, \cdot)\right\|_{L^\infty_{\zeta}}  \, dk'.
   \end{multline*}
   From standard convolution estimates, we  obtain that
   $$ | d|_{M,0}
    \lesssim \big(  |  a |_{M, 1}\, \| |k|\mathcal{F}_{x} \nabla_{x} b \|_{L^1(\mathbb{Z}^d, L^\infty_{\zeta})}
     + | b|_{M+1, 0} \, \left\| \,|(\gamma, \cdot)| \mathcal{F}_{x} \nabla_{k} a\right\|_{L^1(\mathbb{Z}^d, L^\infty_{\zeta})}\big)$$
      and thus, for $M>d/2$, we finally get that
      $$   |d|_{M,0} \lesssim   | a|_{M, 1} |b|_{M+1, 0}.$$
        From 
       Proposition \ref{propcontinu}, we get that
       $$ \|Op^\gamma_{d} u \|_{L^2(\mathbb{R} \times \mathbb{T}^d)} \lesssim   |a|_{M, 1} |b|_{M+1, 0} \|u\|_{L^2(\mathbb{R} \times \mathbb{T}^d)}.$$
        Since by definition of $d$, we have $ Op^\gamma_{a} Op^{\gamma}_b - Op^\gamma_{ab}= { 1 \over \gamma} Op^\gamma_{d}$, the result follows.

  \end{proof}
  
  We shall finally define a semiclassical version of the above calculus.
     For any symbol $a(x, \zeta)$ as above,  we set for $\eps \in (0, 1]$,  $a^\eps(x, \zeta)= a(x, \eps \zeta)= a(x, \eps \gamma, \eps \tau, \eps k)$ and we define for $\gamma \geq 1$, 
 \begin{equation}
 \label{pseudosemi}
 (Op^{\eps, \gamma}_{a} u)(t,x)=
  (Op^{ \gamma}_{a^\eps} u)(t,x).
 \end{equation}   
 For this calculus, we have the following result:
 \begin{proposition}
 \label{propsemi}
  There exists $C>0$ such that for every $ \eps \in (0, 1]$ and for every $\gamma \geq 1$ we have 
  \begin{itemize}
  \item 
    for every  $a \in S_{M,0}$  with  $M>d/ 2$,
  $$ \|Op_{a}^{\eps, \gamma} u\|_{L^2(\mathbb{R} \times \mathbb{T}^d)} \leq C |a|_{M,0} \|u\|_{L^2(\mathbb{R} \times \mathbb{T}^d)}, $$
   \item for every $a \in  S_{M, 1}$ and for every $b \in S_{M+1, 0}$, with $M>d/2$,
  $$  \|Op_{a}^{\eps, \gamma} Op_{b}^{\eps, \gamma} (u) - Op_{ab}^{\eps, \gamma} (u)\| \leq {C \over \gamma} |a|_{M, 1} |b|_{M+1, 0}\, \|u \|_{L^2(\mathbb{R} \times \mathbb{T}^d)}.$$
  
  \end{itemize}
 
 \end{proposition}
 
 \begin{proof}[Proof of Proposition~\ref{pseudosemi}]
  The proof is a direct consequence of Proposition \ref{propcontinu} and  Proposition \ref{propcalculus}
   since for any symbol $a$, we have by definition of $a^\eps$  that  
   $$ |a^\eps|_{M, 0}= |a|_{M,0}, \quad  |a^\eps |_{M, 1}= |a|_{M,1}.$$
 
 \end{proof}
 
 \bigskip
 {\bf Acknowledgments.} We thank David G\'erard-Varet for very fruitful discussions that initiated this work.
  This work was partially supported  by the ANR projects  BoND (ANR-13-BS01-0009-01) and DYFICOLTI  (ANR-13-BS01-0003-01).
  
  %\section{The case of the classical Vlasov-Poisson system}

  \end{document}